\documentclass[11pt]{amsart}

\usepackage{graphicx}
\usepackage[margin=1in]{geometry}
\usepackage{amsmath}
\usepackage{amssymb}
\usepackage{amsthm}
\usepackage{amsfonts}
\usepackage{esint}
\usepackage[export]{adjustbox}
\usepackage{graphicx}
\usepackage[normalem]{ulem}
\usepackage{enumerate}
\usepackage[scriptsize,hang,raggedright]{subfigure}
\usepackage[cmyk,usenames,dvipsnames]{xcolor}
\usepackage{mathtools}
\usepackage[titletoc]{appendix}

\usepackage{hyperref}
\usepackage{pdfsync}
\usepackage{dsfont}
\usepackage{color}

\usepackage{cite}
\usepackage{bbm}
\usepackage{multirow}
\usepackage[font=footnotesize,labelfont=bf]{caption}
\usepackage{bibentry}
\nobibliography*

\newcommand{\Rd}{\ensuremath{{\mathbb{R}^d}}}

\newcommand{\brho}{\boldsymbol{\rho}}
\newcommand{\beff}{\boldsymbol{f}}

\newcommand{\bv}{\boldsymbol{v}}

\newcommand{\NN}{\mathbb{N}}
\newcommand{\C}{\mathcal{C}}

\newcommand{\bu}{\mathbf{u}}
\newcommand{\bm}{\mathbf{m}}
\newcommand{\bsigma}{\boldsymbol{\sigma}}

\def\P{{\mathcal P}}
\def\N{{\mathfrak{N}}}

\def\div{{\rm div}}
\newcommand{\M}{\mathcal{M}}
\newcommand{\Rn}{{\mathord{\mathbb R}^n}}

\newcommand{\R}{{\mathord{\mathbb R}}}

\newcommand{\supp}{{\mathop{\rm supp }}}
\newcommand{\id}{\boldsymbol{ \mathop{\rm id}}}

\newcommand{\la}{\left\langle}
\newcommand{\ra}{\right\rangle}

\newcommand{\G}{\mathcal{G}}

\newcommand{\bw}{\mathbf{w}}
\newcommand{\bp}{\mathbf{p}}

\newcommand{\bx}{\boldsymbol{x}}
\newcommand{\bmu}{{\boldsymbol{\mu}}}
\newcommand{\bnu}{{\boldsymbol{\nu}}}

\newcommand{\PP}{\mathfrak{P}}

\def\P{{\mathcal P}}

\def\epsilon{\varepsilon}

\DeclareMathOperator{\Tan}{Tan}

\newcommand{\be}{\begin{equation}}
\newcommand{\ee}{\end{equation}}
\newcommand{\bes}{\begin{equation*}}
\newcommand{\ees}{\end{equation*}}

\newtheorem{thm}{Theorem}[section]

\newtheorem{cor}[thm]{Corollary}
\newtheorem{prop}[thm]{Proposition}
\newtheorem{lem}[thm]{Lemma}

\newtheorem{as}[thm]{Assumption}

\theoremstyle{definition}

\newtheorem{defn}[thm]{Definition}

\theoremstyle{remark}
\newtheorem{rem}[thm]{Remark}

\numberwithin{equation}{section}

\newcommand{\dn}{{\DJ or\dj e }}

\title{Vector valued optimal transport: from dynamic to static formulations}

\author{Katy Craig, Nicol\'as {Garc\'ia Trillos}, \dn Nikoli\'c}
\date{\today}

\begin{document}

\subjclass[2020]{49Q22, 49J20, 94A08, 28A33, 49Q20, 60B10}
\keywords{Optimal transport on graphs, vector valued optimal transport, multispecies PDE} 

\begin{abstract}

Motivated by applications in classification of vector valued measures and multispecies PDE, we develop a theory that unifies   existing notions of vector valued optimal transport, from dynamic formulations (à la Benamou-Brenier) to  static formulations (à la Kantorovich). In our framework, vector valued measures are modeled as probability measures on   a product space $\Rd \times \G$, where $\G$ is a weighted graph over   a finite set of nodes and the graph geometry strongly influences the associated dynamic and static distances. We obtain sharp inequalities relating   four notions of vector valued optimal transport and prove that the distances are mutually bi-H\"older equivalent. We   discuss the theoretical and practical advantages of each metric and indicate potential applications in multispecies PDE and data analysis. In particular, one of the static formulations discussed in the paper is amenable to linearization, a technique that has been explored in recent years to accelerate the computation of pairwise optimal transport distances.

\end{abstract}

\maketitle

\section{Introduction}

The core idea of optimal transport is to quantify the difference between probability measures in terms of the effort required to rearrange one to look like another. When the notion of effort arises from the square of a distance on the domain of the measures, this induces a  {distance} between probability measures known as the \emph{Wasserstein metric}. Over the past twenty-five years, the Wasserstein metric has become a fundamental tool across mathematics, particularly in   {partial differential equations} and  {data analysis}, where its favorable geometric properties lend it to the study of stability, convergence, and long-time behavior of   evolutionary systems. Key attributes of the Wasserstein metric that have strengthened its impact are that it has equivalent characterizations both in terms of a \emph{dynamic} formulation, where the rearrangement of one measure to look like the other is expressed via an evolutionary differential equation, and a \emph{static} formulation, in terms of an (infinite dimensional) linear program.

Given that the original Wasserstein metric is limited to the space of probability measures or, more generally, finite Borel measures with equal total mass, there has been significant interest in developing generalized Wasserstein metrics that can extend to larger classes of measures. A breakthrough in this direction was the discovery of the Hellinger-Kantorovich (HK) metric \cite{liero2016optimal,liero2018optimal,kondratyev2016new,chizat2018interpolating,chizat2018unbalanced}, which quantifies the distance between finite Borel measures in terms of the effort to transform one into another by a combination of spatial transport and creation/destruction of mass. As in the Wasserstein case, the HK metric possesses both dynamic and static characterizations, which has led to a wealth of applications in reaction-diffusion equations and data analysis.

Inspired by these previous works, the present manuscript develops a generalization of the Wasserstein metric to the space of \emph{vector valued measures}. Our approach unifies existing results on vector valued optimal transport into a common mathematical framework and provides a sharp characterization of the relationship between dynamic and static formulations. 

The remainder of this introduction proceeds as follows. We begin by introducing our motivating applications and explain their connection to vector valued optimal transport (vvOT). We next describe our precise notion of vector valued measures and previous work on vvOT. With an aim toward unifying this existing theory, we present our three notions of vvOT distances and describe our main results, which show that, while these distances do not coincide, they are ordered, and, on bounded domains, they are bi-H\"older equivalent. We close by discussing the class of partial differential equations that have a gradient flow structure with respect to the dynamic vvOT metric and explaining to what extent the static vvOT metric may be \emph{linearized} in data analysis applications.

\subsection{Motivating applications} A natural motivation for vvOT distances arises from applications in data classification and partial differential equations. In data classification, a common task is to analyze the structural similarity of labeled datasets. As an illustration, consider a dataset comprised of images $\{x_k\}_{k=1}^N$, where each image is represented by a point in a high-dimensional Euclidean space,  $x_k \in \Rd$. (For example, a natural choice for $x_k$ is given by the linearized OT embedding of the image \cite{wang2013linear}.) Furthermore, suppose that each image $x_k$ is assigned a label $f(x_k) = l_i$, where the set of possible labels$ \{l_i\}_{i=1}^n$ is, typically, much smaller than the total number of images. Assume that there is both a norm on the space of images that quantifies their difference $\|x_k - x_l\|$, as well as a nonnegative matrix $\{q_{ij}\}_{i,j,=1}^n$ that quantifies the similarity between the different labels. (Large values of $q_{ij}$ indicate a high degree of similarity.) For example, in the Fashion MNIST dataset \cite{xiao2017fashion}, images are labelled in terms of the type of clothing $\{l_1,l_2,l_3\} = \{\text{coat}, \text{shirt}, \text{sandal}\}$, and certain labels $\{\text{coat}, \text{shirt} \}$ are   more similar than others $\{\text{coat}, \text{sandal} \}$; thus, in this case, $q_{12}>q_{13}$. A labeled dataset may then be represented as a vector valued measure $\bmu := [\mu_1, \mu_2, \dots \mu_n]^t$, where each component is a finite (positive) Borel measure $\mu_i \in \M(\Rd)$, representing the distribution (not necessarily normalized to one) of images with label $l_i$, 
\[ \mu_i = \sum_{k \in L_i} \delta_{x_k} , \quad L_i := \{ k: f(x_k) = l_i \} . \]

In this application, one would seek a distance $d$ on the space of vector valued measures to quantify the similarity of one labeled dataset to another, both in terms of the similarity of their underlying images  and the similarity of the labels assigned to those images. One application of such a distance would be  to compare the performance of different classification methods, that is, different methods for assigning labels to the same group of images. If  the labeled datasets $\bmu, \bnu$ arising from two different classification methods are nearly identical, except that one method classified a few shirts as coats, the performance of the methods is more similar and the value of $d(\bmu,\bnu)$ should be smaller than if  shirts were classified as sandals. In another context, having a distance on the space of vector valued measures would provide a way of comparing labeled datasets $\bmu, \bnu$ arising in two different experimental applications, such as labeled jet datasets arising at the Atlas vs CMS experiments at the Large Hadron Collider \cite{komiske2019metric,cai2020linearized}. For a discussion on the use of label similarities in classification problems, see, e.g., \cite{Silla2010,LabelSimilarity1}.

In addition to data classification, the need for a distance between vector valued measures is also motivated by applications in partial differential equations. In the classical Wasserstein case, the \emph{dynamic} characterization of the Wasserstein metric leads to a formal Riemannian structure, and thus of a notion of gradient, which makes possible the study of \emph{gradient flows} in the space of probability measures: given an energy $E: \mathcal{P}(\Rd) \to \R\cup \{+\infty\}$, its \emph{Wasserstein gradient flow} is (formally) the solution of the initial value problem
\[ \begin{cases} \partial_t \rho - \nabla \cdot  \left( \rho \nabla \frac{\partial E}{ \partial \rho}  \right) = 0 ,  \\ \left. \rho_t \right|_{t=0} = \rho_0,  \end{cases} \]
in the sense that the solution of this equation is the curve $t \mapsto \rho_t \in \P(\Rd)$ that makes the energy $E(\rho_t)$ decrease as quickly as possible, with respect to the Wasserstein structure. Such flows inherit intrinsic stability properties with respect to perturbations in the Wasserstein distance.

Motivated by this, given an energy $E: \M(\Rd)^n \to \R \cup \{+\infty\}$ on the space of vector valued measures, one seeks a notion of distance that likewise has a formal Riemannian structure and, hence, can again provide a link between partial differential equations (PDE) and minimizing the energy as quickly as possible. For example, in either generative modeling or sampling, one may encounter an energy 
\[ E(\brho) = \begin{cases} \mathcal{L}(\brho , \bmu) &\text{ if } \brho \in C ,\\ +\infty &\text{ otherwise,}\end{cases} \]
defined in terms of a loss function $\mathcal{L}$ and a constraint set $C$ that expresses the distance between   $\brho \in \C$ and an observed labeled dataset $\bmu$. In generative modeling, one seeks $\brho$ in a parametric class $C$ that minimizes $E$; in sampling, one seeks $\brho$ belonging to the class of vector valued empirical measures $C$ that minimizes $E$. In both cases, having PDE dynamics characterizing a gradient flow provides a natural approach to flow toward a candidate minimizer that is stable with respect to the corresponding distance on vector valued measures.

\subsection{Vector valued measures and previous work} 
Inspired by the preceding applications, we seek a distance $d$ on the following space of vector valued measures: given a closed, convex subset of Euclidean space $\Omega \subseteq \Rd$ and $n$ labels represented by a graph $\G$ with $n$ nodes, we consider the space of measures 
\[ \P(\Omega \times \G) := \left\{ \bmu = [\mu_1, \dots, \mu_n]^t : \mu_i \in \mathcal{M}(\Omega)  \ \forall i =1, \dots, n \ \text{ and } \ \sum_{i=1}^n \mu_i(\Omega) =1 \right\} . \]
The sum of the components of the vector valued measures is a key quantity in what follows, and we denote it by
\[ \bar{\mu} := \sum_{i=1}^n  \mu_i \in \P(\Omega) . \]
The normalization requirement that $\bar{\mu}$ is a probability measure reflects the fact that, given any $\bmu, \bnu \in \P(\Omega \times \G)$, it is possible to transform $\bmu$ into $\bnu$ via spatial transport in the domain $\Omega$ and mutation of mass between the different labels in $\G$. 

In recent years, several previous works have considered different notions of vvOT distance.  Chen, Georgiou, and Tannenbaum \cite{chen2018vector}, for example, introduced a notion of transport distance between $\bmu,\bnu \in \P(\Omega \times \G)$ by considering a \emph{dynamic} formulation via a   product metric structure that combined the classical Wasserstein metric on $\P(\Omega)$ with a graph Wasserstein metric on $\P(\G)$, taken with an upwinding interpolation that induces a Finslerian geometry  \cite{esposito2021nonlocal}. In subsequent work, Zhu et al. \cite{zhu2022vwcluster} used this metric to propose a method for clustering of collections of vector valued measures. 

A different formulation was presented by Bacon  \cite{bacon2020multi}, who introduced a \emph{static} formulation of vvOT on $\P(\Rd \times \G)$, by considering all transport plans $\Gamma_{i,j} \in \M(\Rd \times \Rd)$ that rearrange a part of $\mu_{i}$ to look like a part of $\nu_{j}$,
\[ \Pi(\bmu,\bnu) := \left\{\boldsymbol{\Gamma} = [\Gamma_{i,j}]_{i,j=1}^n , \Gamma_{i,j} \in \mathcal{M}(\Rd \times \Rd) : \sum_{i=1}^{n} d\Gamma_{ij}(x,y) = d\nu_{j}(y), \ \sum_{j=1}^{n} d\Gamma_{i,j}(x,y) =d\mu_{i}(x) \right\} . \]
For a given matrix of cost functions $c_{ij}(x,y)$, Bacon then defined the distance between $\bmu$ and $\bnu$ in terms of the minimal effort of transporting $\bmu$ to $\bnu$, 
\begin{equation} \label{BaconKantorovich}
d(\bmu,\bnu):=    \inf \left\{ \sum_{i,j=1}^{n} \iint_{\Rd \times \Rd} c_{ij}(x,y)d\Gamma_{ij}(x,y) : \boldsymbol{\Gamma} \in \Pi(\bmu,\bnu) \right\} .\end{equation}
Bacon likewise    developed appropriate conditions on the cost functions to ensure that, indeed, this is a distance. 
    
    Finally, many previous works have considered vvOT distances of 1-Wasserstein type, by interpreting $\P(\Omega \times \G)$ as a subset of the dual space of vector valued Lipschitz or bounded Lipschitz functions \cite{ciosmak2021optimal,ryu2018vector, todeschi2025unbalanced}. The motivation in these papers included applications to Full Waveform Inversion, in which signed measures are represented as vector valued measures with two components. The distances proposed in those papers, however, lack the formal Riemannian structure of the classical 2-Wasserstein metric and, hence, the connection to gradient flow PDEs and linearization.
 
 \subsection{Main results}
In the present work, we introduce three notions of vvOT distances --- one dynamic and two static -- that are natural adaptations of previous works on vvOT. While our main result shows that the dynamic and static distances do not coincide, in contrast to the classical Wasserstein setting, we succeed in proving sharp inequalities relating the distances with each other and with the dual bounded Lipschitz distance, thereby providing a framework for unifying the previous results on vvOT.

Inspired by Chen, Georgiou, and Tannenbaum \cite{chen2018vector}, we define our dynamic distance via a product space structure between the classical and graph Wasserstein metrics. Given an $n$-node graph $\G$, we suppose its edges are weighted according to a nonnegative matrix $[q_{ij}]_{i,j=1}^n$. Due to our motivating applications, in which $q_{ij} = q_{ji}$ represents the similarity between labels $l_i$ and $l_j$, we suppose $q_{ij}$ is symmetric, and we also suppose the associated weighted graph $\G$ is \emph{connected}.  In what follows, it can be seen that the weights along the diagonal $q_{ii}$ have no effect on the metric properties, so they may be chosen arbitrarily. 

In this setting, Maas \cite{maas2011gradient} defined a Wasserstein metric over the space of probability measures on the graph $\mathcal{P}(\G)$ via a \emph{dynamic} characterization, in which one considers all ways of flowing from $p \in \P(\G)$ to $\tilde{p} \in \P(\G)$ via solutions of a graph continuity equation and defines the notion of distance between $p$ and $\tilde{p}$ in terms of a least  {action} principle. In the past few years, multiple works have studied the geometry of optimal transport on graphs  \cite{maas2011gradient,erbar2012ricci, mielke2011gradient,mielke2013geodesic,chow2012fokker,gangbo2019geodesics,erbar2019geometry},  as well as its scaling limits for different graph ensembles \cite{esposito2021nonlocal, Gigli2013,Slepev2023,NGT2020,gladbach2020homogenisation,gladbach2020scaling,gladbach2023homogenisation}.  
 See section \ref{sec:GraphOT} for background on the graph Wasserstein metric.
 
To generalize to the vector valued case, where we also have a spatial coordinate, we consider the vector valued continuity equation, 
 \begin{align} \label{vvcty}   
& \partial_t \brho + \nabla \odot (\brho \odot \bu) + \nabla_\G \cdot(\check{\brho}  \bv) = 0  \ ,  \quad  (\check{\brho} \bv)_{ij}(x) = \theta(\rho_i(x), \rho_j(x)) v_{ij}(x),  
 \end{align}where $ \nabla \odot$ denotes component-wise divergence, $\odot$ denotes component-wise multiplication, and we impose no flux boundary conditions
$\brho_i(x) \bu_i(x) \cdot \vec{n} = 0 $ for all $x \in \partial \Omega$, $i = 1, \dots , n$. (See Definition \ref{def:ContEqMultiSpecies} for the precise notion of weak solution.) As in the original work by Maas, we consider a range of \emph{interpolation functions} $\theta$, via which we compute an average value of the density of $\rho_i$ and $\rho_j$ at location $x$ to induce a notion of   density along the edge connecting labels $l_i$ and $l_j$. Examples of valid interpolation functions include the arithmetic, geometric, and logarithmic averages.

\begin{as}[Interpolation function] \label{interpolationassumption} The function $\theta:[0,+\infty) \times [0,+\infty) \mapsto [0,+\infty)$ satisfies the following properties: 
\begin{enumerate}[(\text{A}1)]
    \item  (\it{Regularity}) $\theta$ is continuous on $[0,+\infty) \times [0,+\infty)$  and $C^\infty$ on $(0,+\infty) \times (0,+\infty)$. \label{thetaregularity}

    \item (\it{Symmetry}) $\theta(t,s)= \theta(s,t)$ for $t,s\geq 0.$ \label{thetasymmetry}

    \item  (\it{Positivity}) $\theta$ is positive on $(0,+\infty) \times (0,+\infty).$ \label{thetapositivity}

    \item (\it{Normalization}) $\theta(1,1) = 1$. \label{thetanormalization}

    \item (\it{Monotonicity}) $\theta(r, t)\leq \theta(s,t)$ for $t \geq 0$ and $0\leq r \leq s.$ \label{thetamonotonicity}

    \item  (\it{Positive homogenity}) $\theta(\lambda t,\lambda s)= \lambda \theta(t,s)$ for $\lambda > 0$ and $t,s\geq 0.$ \label{thetapositivehomogenity}

    \item(\it{Concavity}) $\theta$ is concave. \label{thetaconcavity}

 \end{enumerate}
\end{as}
\noindent Since $\theta$ is 1-homogeneous,   we use the abbreviation 
\[ \theta(\rho_i, \rho_j) := \theta\left( \frac{d \rho_i}{d \bar{\rho}}, \frac{ d \rho_j}{d \bar{\rho}} \right) d \bar{\rho} \in \mathcal{M}(\Omega); \]
note that, indeed, $\rho_i$ is absolutely continuous with respect to $\overline{\rho}$, and we use $\frac{d\rho_i}{d \overline{\rho}}$ to denote the corresponding Radon-Nikodym derivative.

  In the vector valued continuity equation (\ref{vvcty}), $\brho: [0,T] \to \P(\Omega \times \G)$ represents a vector valued measure evolving according to spatial velocity $\bu(x) = [\bu_i(x)]_{i=1}^n$, which describes how the mass of each component moves in the domain $\Omega$, and a graph velocity $\bv(x) = [ v_{ij}(x)]_{i,j=1}^n$, which describes how mass of component $i$ mutates to component $j$ at location $x \in \Omega$. In the special case that the graph velocity is symmetric, the evolution of the $i$th component $\rho_i$ is given by the following simplified PDE:
 \begin{align*}
 \partial_t  \rho_{i,t} + \nabla \cdot( u_{i,t} \rho_{i,t}) = \sum_{j=1}^n\theta ( \rho_{i,t},  \rho_{j,t} ) v_{ij,t}q_{ij} .
 \end{align*}
 Given $\bmu, \bnu \in \P(\Omega \times \G)$, we denote
 \begin{align} \label{dynamicconstraintdef} \mathcal{C}(\bmu,\bnu) := \{ (\brho, \bu,\bv):  (\brho, \bu,\bv) \text{ satisfies   vv continuity equation on $[0,1]$, }   \brho_0= \bmu , \ \brho_1 = \bnu  \}. 
 \end{align}

 With this notion of continuity equation in hand, we define the  dynamic distance  on vector valued measures in terms of the minimal action required to flow between $\bmu, \bnu \in \P(\Omega \times \G)$,
\begin{align}  \label{dynamic}
W_{\Omega\times \G}^{2}(\bmu, \bnu) \coloneqq \inf_{(\brho, \bu, \bv) \in \C(\bmu,\bnu)}  \int_0^1  \sum_{i=1}^n \int_\Omega | u_i|^2     \rho_i dt  +  \frac12 \sum_{i,j=1}^n \int_0^1 \int_\Omega |v_{ij} |^2 \theta ( \rho_i ,\rho_j ) q_{ij}    dt .
\end{align}
 This distance is very similar to the dynamic distance studied by Chen et al. \cite{chen2018vector}, with one key difference:  Chen et al. considered the \emph{upwinding} interpolation function $(\check{\rho}v)_{ij} = \rho_i (v_{i,j})_+ - \rho_j (v_{ij}))_i$ in the continuity equation, which likewise corresponds to $\frac12 \sum_{i,j=1}^n \int_\Omega (v_{ij} )_+^2   \rho_i  q_{ij}   dt$ for the second term in the action. We instead follow the choices of interpolation function studied by Maas \cite{maas2011gradient}, since, in this setting,  the graph Wasserstein metric induces a smooth Riemannian structure on the space of strictly positive probability measures on $\G$, which is a crucial ingredient in the proof of our main theorem relating dynamic and static metrics.

  Our first main result is that the dynamic metric \eqref{dynamic} is indeed a distance function on $\P_2(\Omega \times \G)$, the subset of $\P(\Omega \times \G)$ with finite second moments, and the infimum in its formulation is   achieved.
 \begin{thm} \label{maindynamictheorem}
 Suppose $\Omega\subseteq \Rd$ is closed and convex, $\G$ is connected and symmetric, and the interpolation function $\theta$ satisfies Assumption \ref{interpolationassumption}. Then $W_{\Omega \times \G}$ is a metric on $\P_2(\Omega \times \G)$ and the infimum in equation (\ref{dynamic}) is achieved for any  $\bmu,\bnu \in \P_2(\Omega \times \G)$.
 \end{thm}
 
 Next, we turn to static formulations of the vvOT distance. Our approach is inspired by analogy with  the  Hellinger-Kantorovich distance between finite Borel measures.  On one hand, this distance admits a dynamic formulation in terms of a least action principle along solutions of a continuity equation allowing transport and reaction. Given $\mu , \nu \in \mathcal{M}(\Omega)$, the HK distance is given by
 \begin{align*}
   H \hspace{-1mm} K^2(\mu,\nu) := \inf  \left\{ \int_0^1  \int_\Omega (|v_t|^2+ 4 \xi_t^2) \rho_t dt :  \partial_t \rho_t + \nabla \cdot (\rho_t  v_t) = 4 \rho_t \xi_t, \ \  \rho_0 = \mu, \rho_1 = \nu \right\}.
\end{align*} 
On the other hand, Liero, Mielke, and Savar\'e  \cite{liero2016optimal,liero2018optimal} show that HK   admits an equivalent static formulation in terms of suitable projections on the cone  $\mathfrak{C}_{\Omega}:= \Omega \times [0,+\infty)/(x_1,0) \sim (x_2, 0)$ over  $\Omega$  \cite{burago2001course}. Indeed, given the projection operator $\PP_{H \hspace{-1mm} K}: \mathcal{M}_{2}(\mathfrak{C}_{\Omega}) \rightarrow  \mathcal{M}(\Omega): \lambda \mapsto \pi_{\Omega} \#( r^2d \lambda(x,r))$ they show that \begin{equation*}  H \hspace{-1mm} K (\mu,\nu) = \inf \{ W_{\mathfrak{C}_{\Omega}}(\lambda_{\mu},\lambda_{\nu}) | \hspace{1mm} \PP_{H \hspace{-1mm} K}(\lambda_{\mu}) = \mu, \PP_{H \hspace{-1mm} K}(\lambda_{\nu}) = \nu \},\end{equation*}
where $W_{\mathfrak{C}_{\Omega}}$ is the standard 2-Wasserstein metric between probability measures over  $\mathfrak{C}_\Omega$, endowed with the cone metric.

In the HK setting of finite Borel measures,  the role of the cone $\mathfrak{C}_\Omega$ is to allow an extra dimension $r \in [0,+\infty)$ that represents the amount of mass at each location $x \in \Omega$. Seeking an analogy in the vvOT setting, we allow the role of the cone $\mathfrak{C}_\Omega$ to be played by the simplex $ \Delta^{n-1}$, leveraging the fact that there  is a natural bijection between $\Delta^{n-1}$  and probability measures on the graph $\P(\G)$,
\begin{align}\label{Idef}
\bp&:   \Delta^{n-1} \to \P(\G)  : [r_i]_{i=1}^{n-1} \mapsto \left[r_1, r_2, \dots, r_{n-1}, 1- \sum_{i=1}^{n-1} r_i \right] , \\
\bp^{-1}&: \P(\G) \to \Delta^{n-1}   : [p_i]_{i=1}^n \to [p_i]_{i=1}^{n-1} .
\end{align}
 See Figure \ref{firstfigure} for an illustration of the isometry.
Thus, we may consider the simplex $\Delta^{n-1}$ to be endowed with a distance $d_{\Delta^{n-1}}$ induced by the graph Wasserstein metric on $\P(\G)$. The simplex allows an extra dimension $r \in \Delta^{n-1}$ that represents how mass at location $x \in \Omega$ is distributed between each of the $n$ labels.
\begin{figure}[ht]
      \centering
        \includegraphics[width=0.7\textwidth]{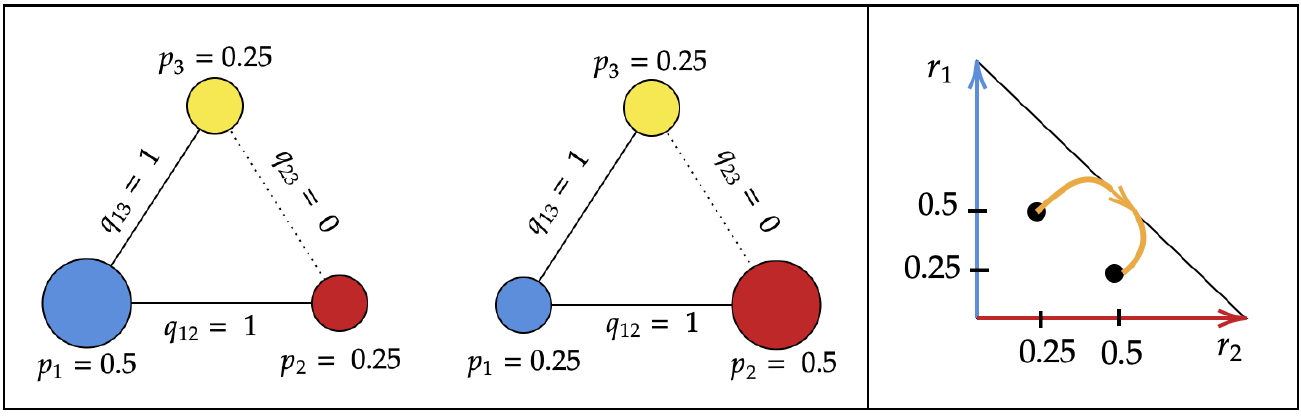}  
        \caption{An illustration of the isometry $\bp$ between the simplex $\Delta^2 \subseteq \R^2$ and probability measures on a three node weighted graph $\G$. Left: Two examples of probability measures, $\sum_{i=1}^3 p_i \delta_i \in \P(\G)$. Right: The corresponding  two points in the simplex. When the simplex is endowed with the distance coming from the Wasserstein metric on $\P(\G)$,  Gangbo, Li, and Mou \cite{gangbo2019geodesics} showed that there exist geodesics between points in the interior that touch the boundary.}
        \label{firstfigure}
   \end{figure}

In this way, following the analogy with the HK distance, we seek  a static distance defined by lifting vector valued measures to probability measures on the product space $\Rd \times \Delta^{n-1}$, via the  projection operator
 \begin{align} \label{projectiondef}
\PP: \P(\Rd \times \Delta^{n-1}) \rightarrow \P(\Rd \times \G) : \lambda \mapsto \pi_{\mathbb{R}^d} \#(\bp(r)  \lambda(x,r)) . 
 \end{align}
 This operator is clearly surjective, since for any $\bmu \in \P(\Rd \times \G)$, it has a  \emph{canonical lifting} given by 
\begin{align} \lambda_{\bmu}(x,r) := \sum_{j=1}^{n} \mu_j(x)\otimes \delta_{e_{j}}(r)  \in \P(\Rd \times \Delta^{n-1}),
\label{eqn:Introcanonical}
 \end{align}
that satisfies $\PP \lambda_\bmu = \bmu$, where $\{e_j\}_{j=1}^n$ denote the vertices of the simplex $\Delta^{n-1}$, so that $\bp(e_j) = \delta_j  \in \P(\G)$ is the probability measure concentrated on the $j$th node of the graph. See Figure \ref{secondfigure} for an illustration of a vector valued measure $\bmu \in \P(\Rd \times \G)$, its canonical lifting $\lambda_{\bmu} \in \P(\Rd \times \Delta^{n-1})$, and another lifting $\lambda \in \P(\Rd \times \Delta^{n-1})$ that projects down to $\bmu$.
\begin{figure}[ht]
        \centering
        \includegraphics[width=0.6\textwidth]{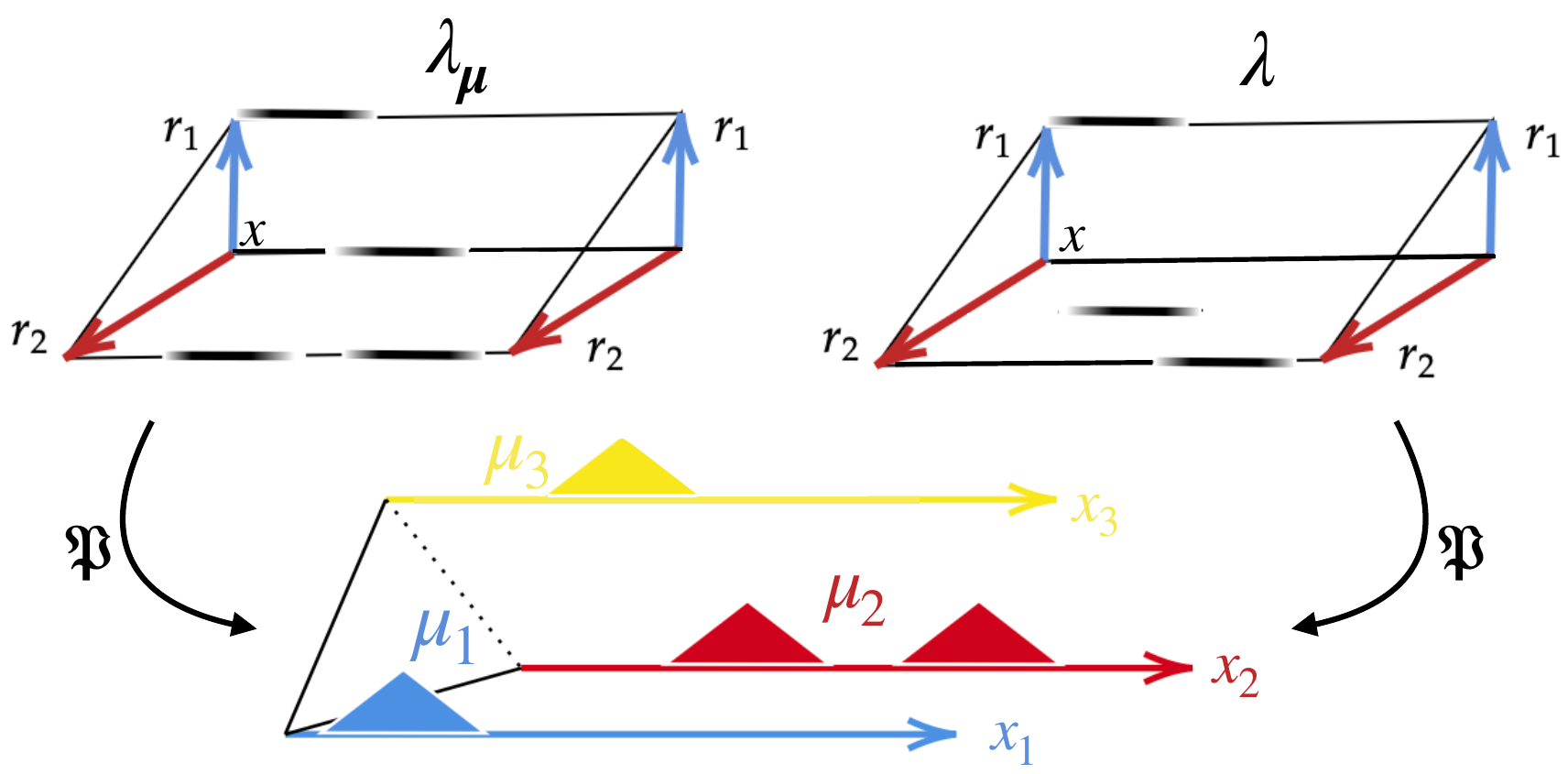} 
        \caption{An illustration of the way in which probability measures on $\R \times \Delta^{2}$ project down to vector valued measures $\P(\R \times \G)$, in the case of a three node graph. Bottom: a vector valued measure, chosen so that the distribution of yellow mass coincides  in physical space with the distribution of the first bump of red mass. Top left: the canonical lifting, where the color of the mass determines the corner of the simplex it is lifted to;  shading is used to indicate the distribution of the mass, which is peaked at the center of each bump. Top right: an alternative lifting, where the yellow bump and first red bump are lifted to a distribution along a line halfway between the edge corresponding to pure yellow mass and the edge corresponding to pure red mass.}
        \label{secondfigure}
\end{figure}

With these definitions in hand, we consider two static approaches for quantifying the difference between vector valued measures $\bmu,\bnu \in \P(\Rd \times \G)$,
\begin{align} \label{semimetricdef}
 D_{\Rd\times \mathcal{G}}  (\bmu,\bnu) := \inf \left\{W_{\Rd \times \Delta^{n-1}}(\lambda_1, \lambda_2) : \PP \lambda_1 = \bmu, \PP \lambda_2 = \bnu \right\},
 \end{align}
and 
\begin{align} \label{staticsingletypedistance}   
        W_{2, \mathcal{W}}(\bmu,\bnu) := W_{\Rd \times \Delta^{n-1}}(\lambda_{\bmu},\lambda_{\bnu}),
\end{align}
where $W_{\Rd \times \Delta^{n-1}}$ denotes the Wasserstein metric on $\P_2(\Rd \times \Delta^{n-1})$ induced by the ground distance $d^2_{\Rd \times \Delta^{n-1}} = d^2_\Rd + d^2_{\Delta^{n-1}}$.
(In these   static formulations, we may consider $\P(\Omega \times \G) \subseteq \P(\Rd \times \G)$, via the canonical injection, so that  they likewise apply to vector valued measures in $\P(\Omega \times \G)$. Only in the   case of $W_{\Omega \times \G}$ does the definition of the distance depend on the spatial domain $\Omega \subseteq \Rd$.)

$W_{2, \mathcal{W}}(\bmu,\bnu)$ can be easily seen to be a distance on $\P_2(\Rd \times \G)$, since $W_{\Rd \times \Delta^{n-1}}$ is a distance on $\P_2(\Rd \times \Delta^{n-1})$ and any lifting $\lambda$ of $\bmu \in \P_2(\Rd \times \Omega)$ satisfies  $  \lambda \in \P_2(\Rd \times \Delta^{n-1})$; see inequality \ref{secondmomentupstairs}. We use the notation $W_{2, \mathcal{W}}$ since  this static metric reduces to a Kantorovich-type metric on $\P(\G)$ studied by Erbar and Maas when   $\supp \mu_i = \supp \nu_i = \{x_0\}\subseteq \Rd$ for all $i =1, \dots, n$. In addition, we also show that $W_{2, \mathcal{W}}$  admits a Kantorovich formulation of the form studied by Bacon  \cite{bacon2020multi}, equation (\ref{BaconKantorovich}), for the cost function $d_\Rd^2(x, y) + W_\G^{2}(\delta_i,\delta_j) $. See Proposition \ref{W2WKantorovich} for both of these results.

On the other hand, in contrast to the $H\hspace{-1mm}K$ case, we prove  $D_{\Rd\times \mathcal{G}} (\bmu,\bnu)$ is merely a \emph{semimetric}. 
\begin{prop} \label{semimetricprop} 
Suppose  $\G$ is connected and symmetric and the interpolation function $\theta$ satisfies Assumption \ref{interpolationassumption}.  Then $D_{\Rd\times \mathcal{G}} (\bmu,\bnu)$ is a \emph{semimetric} over the space $\mathcal{P}_2(\R^d \times \G)$, that is, it is a nonnegative, symmetric, finitely valued function that vanishes only when $\bmu=\bnu$. \end{prop}
\noindent In Proposition \ref{sharpnesscor} below, we give an example for which $D_{\Rd\times \mathcal{G}} (\bmu,\bnu)$ fails the triangle inequality.

Our main theorem shows that the three vvOT distances we have defined are linearly ordered.
\begin{thm} \label{thm:Main}
Suppose $\Omega \subseteq \Rd$ is closed and convex, $\G$ is connected and symmetric, and the interpolation function $\theta$ satisfies Assumption \ref{interpolationassumption}. Then for every $\bmu,\bnu \in \P_{2}(\Omega \times \G)$ we have
\begin{equation*}
    W_{\Omega \times \G}(\bmu,\bnu) \leq D_{\Rd \times \G}(\bmu,\bnu)  \leq   W_{2, \mathcal{W}}(\bmu,\bnu). 
\end{equation*}\end{thm} 
\noindent While the second inequality is an immediate consequence of the definitions of the static metric $D_{\Rd\times \mathcal{G}}$, our proof of the first inequality requires careful use of the Riemannian structure on the interior of $(\P(\G), W_\G) \cong (\Delta^{n-1}, d_{\Delta^{n-1}})$ and hence on the interior of $(\Rd \times \Delta^{n-1}, (d^2_\Rd \oplus d^2_{\Delta^{n-1}})^{1/2})$. Not only does this Riemannian structure degenerate at the boundary, but  Gangbo, Li, and Mou \cite{gangbo2019geodesics} show that there exist geodesics between points in the interior that touch the boundary; see Figure \ref{firstfigure} and Remark \ref{geodesicintersectingtheboundary}. This both raises the possibility that the ground space is \emph{branching} (see  Remark \ref{branchingremark}), and necessitates an approximation argument, in which we show geodesics may be approximated by curves that avoid the boundary; see Lemma \ref{prop:RegularMeasures}. With these tools in hand, the key idea of our argument is to use Lisini's characterization of absolutely continuous curves \cite{lisini2007characterization} to show that approximate geodesics on the lifted space satisfy a continuity equation with respect to the Riemannian structure and then show that the solutions of this equation project down to solutions of the vector valued continuity equation with lower action; see Proposition \ref{coveringspaceinduced}. 

As an immediate consequence of Theorem \ref{thm:Main}, we show  that all three metrics dominate the bounded Lipschitz distance  on $\P(\Omega \times \G)$ (see equation \ref{BLdef} for the definition), and when the domain $\Omega$ is bounded, all metrics are bi-H\"older equivalent, hence topologically equivalent.
  \begin{cor}[Relations between (semi)-metrics on $\P_2(\Omega \times \G)$]
\label{thm:EquivalenceTopologies}
Suppose  $\G$ is connected and symmetric, the interpolation function $\theta$ satisfies Assumption \ref{interpolationassumption}, and $\Omega \subset \Rd$ is closed and convex. Define $Q:= { \max_{i} \sum_{j=1}^n  q_{ij}}$. Then, for all $\bmu,\bnu \in \P_2(\Omega \times \G)$, 
\begin{align} \label{inequalityallmetrics}
 \min\{ 1, Q^{-1/2}\}  d_{BL} (\bmu,\bnu) \leq W_{\Omega \times \G}(\bmu,\bnu) \leq D_{\Rd \times \G}(\bmu,\bnu) \leq W_{2, \mathcal{W}}(\bmu,\bnu) .
\end{align}
Furthermore, if $\Omega$ is bounded, then for $C_{\Omega, \Delta^{n-1}} := \mathrm{diam}(\Omega \times \Delta^{n-1})$,
\begin{align*}
W_{2, \mathcal{W}} (\bmu, \bnu)   \leq  n^{1/4} C_{\Omega,\Delta^{n-1}} (1+ C_{\Omega,\Delta^{n-1}}^2)^{1/4}  \sqrt{d_{BL}(\bmu, \bnu)},
\end{align*}
and all metrics are topologically equivalent.
\end{cor}
 
 We conclude with   examples that show the preceding inequalities are sharp and  $D_{\Rd \times \G}$ fails to satisfy the triangle inequality.
\begin{prop}[Examples of sharpness and inequality] \label{sharpnesscor}
Suppose $\G$ is a two-node graph with  $q_{ij}\equiv 1$ and $\Omega = \R$. Then there exist $\theta$ satisfying assumption  \ref{interpolationassumption}  so that   
\begin{enumerate}[(i)]
\item for each pair of inequalities in (\ref{inequalityallmetrics}), the constants are sharp and there exist  measures for which strict inequality holds; \label{sharppart}
\item $D_{\Rd \times \G}$ fails to satisfy the triangle inequality. \label{triangleinequalitycor}
\end{enumerate}
\end{prop}

\subsection{Perspectives on gradient flows and linearization}

Returning to our original motivations in PDEs and data analysis, we now discuss the merits of the vvOT distances in each of these applications. Due to its definition in terms of a PDE constrained optimization problem, the dynamic metric $W_{\Omega \times \G}$ offers a clear connection to evolutionary PDEs.
For example, given an energy $E: \P(\Omega \times \G) \to \R \cup \{+\infty\}$ of the form
 \begin{align*}
E(\bmu) =   \int_{\mathbb{R}^{d}} f(\bmu(x))dx  +  \sum_{i=1}^{n} \int_{\mathbb{R}^{d}} V_i(x)d \mu_{i}(x) + \frac12  \sum_{i,j=1}^{n} \iint_{\mathbb{R}^{d} \times \mathbb{R}^{d}}  W_{ij}(x-y) d\mu_i(x)d\mu_{j}(y),
\end{align*}
 for $W_{ij} = W_{ji}$ even, a heuristic computation shows that the gradient flow of $E$ with respect to $W_{\Omega \times \G}$ satisfies the following system of PDEs:
\begin{align} \label{ourpde}
\partial_t \mu_i  &= \nabla \cdot \left (\mu_i  \left (\nabla \partial_i f(\bmu ) +\nabla V_i + \sum_{k=1}^n \nabla W_{ik}*\mu_k  \right)\right )\\
&\quad -  \sum_{j=1}^n \left( \partial_i f(\bmu )- \partial_j f(\bmu ) +V_i -V_j +\sum_{k=1}^n  (W_{ik}-W_{jk})*\mu_k \right)\theta\left( \mu_{i} ,\mu_j  \right) q_{ij}, \nonumber
\end{align}
for $i =1, \dots, n$. As in the Hellinger-Kantorovich case, we see the role of the energy both in the spatial transport and mutation terms. 

With regard to the motivating examples described earlier, a natural choice of  energy would be $E(\brho)= \mathcal{L}(\brho,\bmu) = \sum_{i=1}^n KL(\rho_i||\mu_i)$, where $\bmu$ is a target vector valued measure. On the other hand, PDEs of the above form have appeared in various contexts  as models of \emph{multispecies dynamics}. For example, for $f(\bmu) =   \sum_{i=1}^n a_i \mu_i \log(\mu_i) + \frac{1}{m} \left( \sum_{i=1}^{n} b_{i}\mu_{i}\right)^{m}$, $m \geq 2,$ and $V_i \equiv  0$,   in the limit as $q_{ij} \to 0$, we recover   
 systems of aggregation-diffusion equations \cite{doumic2024multispecies, carrillo2024well, giunta2022local}   and  multispecies cross-diffusion models \cite{carrillo2024interacting}, \begin{equation*}
    \partial_{t}\mu_{i} = a_{i}\Delta \mu_{i} + b_{i}\div(\mu_{i}\nabla P(\bmu))+ \div (\mu_{i} \sum_{j=1}^{n}\nabla W_{ij}\ast\mu_{j}), \quad P(\bmu) = \left( \sum_{i=1}^{n} b_{i}\mu_{i}\right)^{m-1} .
\end{equation*} 
Similar PDEs  have also arisen in   work on multispecies reaction diffusion systems  \cite{heinze2025discrete},
%
game theory applications \cite{conger2024coupled}, and in machine learning tasks such as neural architecture search \cite{NAS1, NAS2}.

\medskip 

With regard to the previously described motivation for comparing labeled datasets, the static distance $W_{2, \mathcal{W}}$, joined with linearization techniques introduced in \cite{wang2013linear} (see also \cite{LinearizationThorpe,Moosmller2022}), can be used to speed up the computation of pairwise distances between vector valued measures in image classification problems, and other machine learning applications. We briefly describe the idea of linearization in the context of vector valued measures and introduce one final, \textit{linearized} vector valued distance on $\mathcal{P}_2(\R^d \times \G)$.

Suppose that $\{ \bmu^1, \dots, \bmu^N \}$ is a collection of vector valued measures with finite second moments. This data set can be embedded into $\mathcal{P}_2(\R^d \times \Delta^{n-1})$ by considering the canonical lifts $\{ \lambda_{\bmu^1}, \dots, \lambda_{\bmu^N} \}$ defined   in equation \eqref{eqn:Introcanonical}. Then, after selecting a reference measure $\lambda_{\mathrm{ref}} \in \mathcal{P}_2(\R^d \times \Delta^{n-1})$, we can find a representation of the elements of the data set as points in a suitable vector space associated with the reference measure $\lambda_{\mathrm{ref}}$. This construction relies on the following assumption.


\begin{as}
	\label{Assumption:Uniqueness}
	We assume that $\lambda_{\mathrm{ref}}$ is such that, for every vector valued measure $\bmu$ with finite second moment, there exists a unique optimal transport plan $\Gamma_{\bmu}$ for the Wasserstein distance $W_{\Rd \times \Delta^{n-1}}$ between $\lambda_{\mathrm{ref}}$ and $\lambda_\bmu$. Moreover, we assume that this transport plan is induced by a map $T_{\bmu} : \R^d \times \Delta^{n-1} \rightarrow \R^d \times \Delta^{n-1}$, that is,  $\Gamma_\bmu = (\mathrm{Id} \times T_\bmu)_{\sharp} \lambda_{\mathrm{ref}}$.
\end{as}

Given  $\lambda_{\mathrm{ref}}$ satisfying Assumption \ref{Assumption:Uniqueness}, we associate to each $\bmu^i$  its corresponding $T_{\bmu^i} \in L^2(\lambda_{\mathrm{ref}})$. Taking advantage of the fact that $\R^d \times \Delta^{n-1}$ is a subset of $\R^d \times \R^{n-1}$, it is then reasonable to consider the following \textit{linearized} distance on vector valued measures:
\begin{equation}
d_{\mathrm{LOT}}(\bmu^i, \bmu^j):= \left( \int_{\R^d \times \Delta^{n-1}} \lVert T_{\bmu^i}(x,r) - T_{\bmu^j}(x,r) \rVert^2 d\lambda_{\mathrm{ref}}(x,r) \right)^{1/2},     
\label{def:LOT}
\end{equation}
where we use $\lVert  \cdot \rVert$ to denote the norm in $\R^d \times \R^{n-1} $. As discussed in \cite{LinearizationThorpe,Moosmller2022,wang2013linear}, the advantage of computing \eqref{def:LOT} for all $i,j=1, \dots, N$ is that it only requires solving the $N$ optimization problems needed to obtain the maps $T_{\bmu^i}, \dots, T_{\bmu^N}$, instead of requiring to solve $N^2$ optimization problems to obtain $W_{\R^d \times \Delta^{n-1}}(\lambda_{\bmu^i}, \lambda_{\bmu^j})$ for all $i,j=1, \dots, N$. Besides being useful for its improved computational scalability, linearization provides a concrete way to define methodologies such as PCA as well as other popular kernel methods in the setting of the space of vector valued measures. 


We conclude this section by stating the topological equivalence between the linearized OT distance and all other distances on the space of vector valued measures discussed in this paper, and by making a remark about Assumption \ref{Assumption:Uniqueness}.

\begin{prop}
	\label{prop:LOT}
	Suppose that $\Omega$ is a convex and compact subset of $\R^d$, $\G$ is connected and symmetric, and the interpolation function $\theta$ satisfies Assumption \ref{interpolationassumption}. Suppose further that Assumption \ref{Assumption:Uniqueness} holds. Then the metric $d_{\mathrm{LOT}}$ induces the same mode of convergence in the space $\mathcal{P}_2(\Omega \times \G)$ as  $W_{2, \mathcal{W}}$, $W_{\Omega \times \G}$, and $D_{\R^d \times \G}$.
\end{prop}

\begin{rem} \label{branchingremark}
	In the Euclidean setting, the analog of Assumption \ref{Assumption:Uniqueness} is satisfied when the reference is absolutely continuous with respect to the Lebesgue measure, due to Brenier's theorem; see \cite{Brenier}. A version of Brenier's theorem for smooth and complete Riemannian manifolds holds when the reference is assumed to be absolutely continuous with respect to the manifold's volume form; see \cite{McCann2001}. The crucial geometric property used in this extension of Brenier's theorem, as discussed in more generality in \cite{Gigli2012,Cavalletti2017}, is that a smooth complete Riemannian manifold is a \textit{non-branching} geodesic space. This geometric property, which essentially says that two geodesics cannot coincide for a positive length unless one of them is completely contained in the other, is not necessarily satisfied by the metric space $(\R^d \times \Delta^{n-1}, d_{\R^d \times \Delta^{n-1}})$. Indeed, given the geometric pathologies that $\Delta^{n-1}$ may have, as described in Remark \ref{geodesicintersectingtheboundary}, we can see that the space $\Delta^{n-1}$ endowed with the metric $d_{\Delta^{n-1}}$ may be a branching space. As a result, Brenier's theorem may not hold for the manifold $\R^d \times \Delta^{n-1}$.

	Despite this, we would suggest $\lambda_{\mathrm{ref}}$ to be chosen as 
	\[\lambda_{\mathrm{ref}} = \mathrm{Unif}(B_1(0)) \otimes \mathrm{Unif}(\Delta^{n-1}), \]
	i.e., the product measure of the uniform distribution over the unit ball in $\R^d$ and the standard uniform distribution over $\Delta^{n-1}$, even if we cannot a priori guarantee that Assumption \ref{Assumption:Uniqueness} holds for this choice. We believe, however, that, since the canonical lifts $\lambda_{\bmu}$ are all concentrated on the corners of the simplex $\Delta^{n-1}$, Assumption \ref{Assumption:Uniqueness} is strictly weaker than the statement of a Brenier-like theorem for the space $\R^d \times \Delta^{n-1}$ and we actually conjecture that Assumption \ref{Assumption:Uniqueness} is satisfied by the above choice of $\lambda_{\mathrm{ref}}$ if the graph weights $q_{ij}$ for $\G$ are chosen generically. This is left as an interesting problem to explore in the future.  
\end{rem}

%
%
%
%
%
%

 \subsection{Organization of paper}
 The  paper is organized as follows. In section \ref{graphsection}, we recall the theory of optimal transport on graphs and its induced geometry on the simplex, and we develop a technique for approximating arbitrary geodesics by curves that avoid the boundary of the simplex. In section \ref{sec:Distances}, we introduce our three main vvOT distances. Section \ref{subsection:Dynamicdistances} considers basic properties of the dynamic distance and, in Section \ref{subsection:Existenceofminimizers}, we prove that there exist solutions of the vector valued continuity equation that minimize the action, so that the infimum in the dynamic distance is attained. Section \ref{subsection:Liftedspaceandstaticsemimetrics} considers the static metrics. In Section \ref{sec:Comparison}, we turn to the proof of our main Theorem \ref{thm:Main}, in which we prove an inequality relating the three vvOT distances. Section \ref{sec:ConstructionRegularized} develops an approach for approximating geodesics in $\P(\Rd \times \Delta^{n-1})$ in a way that avoids the boundary of $\Rd \times \Delta^{n-1}$. In Section \ref{sec:UpstairsDownstairs}, we show that solutions of the continuity equation on the interior of $\Rd \times \Delta^{n-1}$ project down to solutions of the vector valued continuity equation with smaller action. In Section  \ref{sec:ProofMainTheorem}, we combine these results to prove our main Theorem \ref{thm:Main}. Then, in Section \ref{sec:TopoEquiv}, we prove Corollary \ref{thm:EquivalenceTopologies}, which ensures that, on a bounded domain $\Omega$, the metrics are topologically equivalent, and  Proposition \ref{prop:LOT}, which establishes the topological equivalence of all three metrics with the linearized distance on vector valued measures defined in \eqref{def:LOT}. Finally, in section \ref{sec:Examples}, we consider several examples in the case of a two node graph that show the preceding inequalities are sharp, proving Proposition \ref{sharpnesscor}.

\section{Optimal transport on graphs and induced geometry on simplex} \label{graphsection}

\subsection{Notation and Preliminaries}

Given $\Omega \subseteq \Rd$ convex and  closed set, let $\mathcal{M}(\Omega)$ denote the set of finite Borel measures on $\Omega$, and let $\mathcal{M}_s(\Omega)$ denote the set of finite signed Borel measures, that is, the set of signed Borel measures with finite total variation norm. Since $\G$ contains $n$ nodes, we may identify the collection of points $\{(x,i) : i \in \mathcal{G}\}$ with a vector $\bx = (x_1, \dots, x_n)  \in \Omega^n$ and we may identify  $\mathcal{M}(\Omega \times \mathcal{G}) \simeq \mathcal{M}(\Omega)^n$. Similarly, $\P(\Omega \times \G) \simeq \{ \bmu \in \mathcal{M}(\Omega)^n : \sum_{i=1}^n \mu_i(\Omega) = 1 \}$. Bold font is used to denote a vector valued measure $\bmu \in \mathcal{M}(\Omega)^n$ and subscripts $\mu_i \in \mathcal{M}(\Omega)$ are used to denote the components of the vector valued measure. We will write $\P_2(\Omega \times \G)$ for the subset of $\P(\Omega \times \G)$ with finite second moment, or equivalently, the set of $\bmu \in \P(\Omega \times \G)$ such that  $M_2(\mu_i) = \int_\Omega |x|^2 d \mu_i(x) < +\infty ,$ for all $i =1, \dots, n .$
 Likewise, we will write $\bmu \in \P_{ac} (\Omega \times \G)$ if each component $\mu_i$ is absolutely continuous with respect to Lebesgue measure, $\mu_i \ll \mathcal{L}^d |_\Omega$, and in this case, we will abuse notation and let $\mu_i$ denote both the measure and its density, $d \mu_i = \mu_i(x) d \mathcal{L}^d(x)$. Finally, we write $\mathcal{M}_s(\Omega)^n$ to denote vector valued finite, signed measures. 
Given a metric space $(X,d)$, a Borel measurable function $t: X \to X$, and a probability measure $\mu \in \mathcal{P}(X)$, the \emph{push forward of $\mu$ under $t$}, $t \# \mu \in \P(X)$ is given by 
$(t \# \mu)(A)  = \mu(t^{-1}(A)) \hspace{1mm} \text{  for all Borel measurable sets $A \subset X$}.$

 In general, we will consider both $\P(\Omega \times \G)$ and $\M(\Omega \times \G)$ to be endowed with the topology of component-wise narrow convergence. We will use similar notation to denote continuity properties of functions from $\R$ to $\P(\Omega \times \G)$. For example $\brho \in C([0,T]; \P(\Omega\times \G))$ if $t \mapsto \rho_{i,t}$ is narrowly continuous for all $i =1, \dots, n$. We use $\brho$ to denote a solution to vector valued continuity equation \ref{vvcty}, where $\rho_{i,t}$ represents evolution of $i$-th coordinate of $\brho$.  
 
 Given $\bmu,\bnu \in \P(\Rd \times \G),$ a classical metric on $\P(\Rd \times \G)$ is given by the \emph{bounded Lipschitz distance}   \cite[Section 11]{dudley2018real}, which is defined by
\begin{align} \label{BLdef}
d_{BL}(\bmu,\bnu) &:= \left( \sum_{i=1}^{ n} \|\mu_i - \nu_i\|_{BL}^{2} \right)^{\frac{1}{2}}  , \quad \|\mu_i - \nu_i\|_{BL} : =  \sup_{\substack{\eta \in C^\infty_c(\Rd), \\ \|\eta\|_{W^{1,\infty}_{2}(\Rd)} \leq 1}} \int \eta d(\mu_i- \nu_i) , \\
 \label{W1inftynormdef}
\|\eta \|_{W^{1,\infty}_{2}(\Rd)} &:= \left( \|\eta\|_{L^\infty(\Rd)}^{2} + \| \nabla \eta \|_{L^\infty(\Rd)}^{2} \right)^{1/2} .
\end{align}
We recall that $d_{BL}$ metrizes narrow convergence on $\P(\Rd\times \G)$ \cite[Remark 8.3.1]{bogachev2007measure}.
 
%
%

Given a complete metric space $(X,d)$ and a curve $x:(0,1) \to X$ is \emph{absolutely continuous}, denoted $x \in AC(0,1;X)$, if there exists $m \in L^1(0,1)$ so that 
\begin{align} \label{ACdef} d(x(t),x(s)) \leq \int_s^t m(r) dr , \quad \forall  \ 0<s \leq t <1 .
\end{align}
For any $x \in AC(0,1;X)$, the \emph{metric derivative}
\[ |x'|(t) := \lim_{s \to t} \frac{ d(x(s),x(t))}{|s-t|} \]
exists for Lebesgue a.e. $t \in (0,1)$, $t \mapsto |x'|(t)$ belongs to $L^1(0,1)$, and $m(t)=|x'|(t)$ is admissible in (\ref{ACdef}). Furthermore, for any $m \in L^1(0,1)$ satisfying (\ref{ACdef}), $|x'|(t) \leq m(t)$ for a.e. $t \in (0,1)$. Sometimes, to emphasize the role of the metric we will write either $|x'|_X(t)$ or $|x'|_d(t)$ to denote the metric derivative.

\subsection{Optimal transport on graphs}
\label{sec:GraphOT}
Let $\G$ denote a weighted graph with nodes $\{1, \dots, n \}$ and edge-weights $\{q_{ij}\}_{i,j =1}^n$.  
In what follows, we assume that the graph $\G$ is symmetric and connected, that is,  the matrix of edge weights $[q_{ij}]_{i,j =1}^n$ is symmetric  and  , between any nodes $i$ and $j$, there exists a sequence of edges $\{e_{kl}\}_{k,l}$ connecting them on which $q_{kl}>0$ for all $k, l$. 
 Let $\R^\G \simeq \R^n$ and $\R^{\G \times \G} \simeq \R^{n^2}$ denote the sets of scalar functions on the nodes and edges of the graph, respectively. Observe that any probability measure over $\G$ may be expressed as
\[ p = \sum_{i=1}^n p_i \delta_i   \quad \text{where} \quad \forall i=1,...,n ,  \quad p_i \geq 0 \  \quad \text{and} \quad \sum_{i=1}^n p_i = 1 .\]
In this way, there is a one-to-one correspondence between probability measures on the graph $p \in \P(\G)$ and points in the simplex $r  \in \Delta^{n-1}$ (see \cite{gangbo2019geodesics}), via the bijections defined in equation (\ref{Idef}).
Through this, we can consider $\P(\G)$ and $\Delta^{n-1}$ endowed with the usual Euclidean topologies, so that $\mathbf{p}^{-1}$ is an isomorphism. The (relative) interior of $\mathcal{P}(\G)$ is
\[ (\P(\G))^\circ = \{ p \in \P(\G): p_i >0  \ \forall i = 1, \dots, n \} . \] 

As originally introduced by Maas \cite{maas2011gradient}, one may define a notion of Wasserstein distance on $\P(\G)$ by considering the least amount of effort required to flow between two probability measures on $\G$ via a discrete analogue of the continuity equation, which we now describe. First, we consider the usual notions of graph gradient and divergence operators 
 \begin{align*}
    \nabla_{\G}:\R^{n} \to \R^{n\times n}(\R): \phi \mapsto [ \phi_j -\phi_i]_{i,j=1}^n , 
 \quad \nabla_\G \cdot :\R^{n\times n} \to \R^{n} : v \mapsto \Big[-\frac{1}{2}\sum_{j}(v_{ij}-v_{ji})q_{ij} \Big]_{i=1}^n. 
\end{align*} 
As the domain of the graph divergence is $\R^{n \times n}$, one must choose an appropriate notion of \emph{interpolation function}  $\theta: [0,+\infty)\times[0,+\infty) \to [0,+\infty)$ to map a probability measure on the graph $p \in \mathcal{P}(\G) \subseteq  \R^n$ to a function $[\theta(p_i,p_j)]_{i,j=1}^n$ on the edges, allowing one to define a notion of \emph{flux} at the graph level. This leads to the following notion of discrete continuity equation; see, e.g.,  \cite[equation 2.8]{erbar2012ricci}.
\begin{defn}[Solution of graph continuity equation] \label{graphctydef} Given an interpolation function $\theta: [0,+\infty)\times[0,+\infty) \to [0,+\infty)$, a pair $(p,v)$ is a \emph{solution of the graph continuity equation on the time interval $[0,T]$} if all the following hold:
\begin{enumerate}[(i)]
\item $p: [0,T] \to \P(\G) \subseteq \R^n$ is continuous;
\item $v: [0,T] \to \R^{n \times n}$ is measurable;
\item $t \mapsto \check{p}_t v_t$ is locally integrable on $(0,T)$, where ${\check{p}}_{ij} v_{ij} := \theta(p_i,p_j) v_{ij}$;
\item \label{continuity equation on graph}
$\partial_t p + \nabla_\G \cdot (  \check{p} v)  = 0,  $  in the sense of distributions.  \label{graphctyeqn} \end{enumerate}
\label{def:ContEqGraph}
In the case $T=1$, we let $\mathcal{C}_\G(p_0,p_1)$ denote the set of solutions $(p,v)$ to the graph continuity equation satisfying $\left. p_t \right|_{t=0} = p_0$ and $\left. p_t \right|_{t=1} = p_1$.
\end{defn}

In the present manuscript, we will consider interpolation functions $\theta$ that satisfy Assumption \ref{interpolationassumption}.
Key examples of interpolation functions satisfying Assumption \ref{interpolationassumption} are the \emph{arithmetic} $\theta(s,t) :=  (s+t)/{2},$ \emph{geometric} $\theta(s,t) :=  \sqrt{st}$ and \emph{logarithmic} $\theta(s,t) :=  \int_0^1 s^{1-\alpha} t^\alpha d\alpha$ means.

\begin{rem}[Comparison with arithmetic mean]\label{interpfunorder} 
For any $\theta$ satisfying Assumption \ref{interpolationassumption},  
    \begin{equation*}
        \theta(s,t) \leq \frac{s+t}{2}, \ \forall s,t \geq 0.
    \end{equation*}
\end{rem}

\begin{rem}[Vanishing at boundary]
Assumption \ref{interpolationassumption} differs slightly from the assumptions in previous work by Maas \cite{maas2011gradient} and  Erbar and Maas \cite{erbar2012ricci}, in that it does not require $\theta(s,t)$ to vanish if either $s$ or $t$ is zero. The results of Erbar and Maas that we rely upon  continue to hold under these weaker hypotheses on $\theta$, and we will further explain why our hypotheses are sufficient as we appeal to these results in our proofs.
\end{rem}
%
%
%

With a notion of discrete continuity equation in hand, it remains to define the \emph{least amount of effort} required to flow between two probability measures $p_0, p_1 \in \P(\G)$ in order to define a notion of Wasserstein metric on $\P(\G)$. Given a time-dependent discrete velocity field $v$,
the \emph{effort} required to realize the corresponding flow $p_t$, $t \in [0,T]$, is given by $ ( \int_0^1 \| v_t \|^2_{{\rm Tan}_{p_t} \P(\G)}dt)^{1/2}$, where for any $u, v \in \R^{n \times n }$,
\begin{align} \label{innerproductgraphdef}
 \la u,v \ra_{\Tan_p \P(\G) } &:= \frac12 \sum_{i,j=1}^n u_{ij} v_{ij} \theta(p_i,p_j) q_{ij}.
\end{align}
The graph Wasserstein distance between $p_0$ and $p_1$ is then defined in terms of a least action principle,
\begin{align} \label{wassersteindistanceongraphs}
W_\G^2(p_0, p_1 ) := \inf \left\{ \int_0^1 \|v\|_{\Tan_p \P(\G)}^2 dt : (p,v)  \in \C_{\G}(p_0,p_1)  \right\} . 
\end{align}

Throughout, we will use the following elementary facts about two convex functions $\alpha$ and $\beta$ that are induced by the choice of interpolation function $\theta$. For the reader's convenience, we provide a proof of the following lemma in  Appendix \ref{realappendix}. 
\begin{lem} \label{Klem} 
Define $\alpha, \beta : \R \times \R \times \mathbb{R}^d \to \mathbb{R} \cup \{+\infty\}$ by \begin{align}  
\alpha(t,s,x) &:= \begin{cases} \frac{\|x\|^2}{\theta(t,s)} &\text{ if }t,s \geq 0, \theta(t,s) \neq 0 , \\ 0 &\text{ if }t, s \geq 0 \text{ and }\|x\|= \theta(t,s) = 0 , \\ +\infty &\text{ otherwise.} \end{cases}  \label{definitionofpsi} \\
\beta(a,b,c) &:= \sup_{t,s \geq 0} \{ at + bs + \frac{\|c \|^2}{4}\theta(t,s) \} . \label{definitionofvarphi} 
\end{align}
Then,
\begin{enumerate}[(i)]
\item $\exists \ K \subseteq \R^{2+d}$ closed and convex so that $\{ a+\frac{\|c\|^{2}}{8}\leq 0 \} \cap \{ b+\frac{\|c\|^{2}}{8}\leq 0 \} \subseteq K$, $\beta(t,s,x) = 0$ for $(t,s,x) \in K$, and $\beta(t,s,x) = +\infty$ otherwise.
\item $\beta^{*}=\alpha$.
\end{enumerate}
In particular, both $\alpha$ and $\beta$ are proper, convex, and lower semicontinuous.
\end{lem}

We now collect some basic facts about $W_\G$  developed by Maas \cite{maas2011gradient} and Erbar and Maas \cite{erbar2012ricci}. See appendix \ref{realappendix} for the proofs and  references.
\begin{prop}[Properties of $W_G$] \label{basicfactWG} 
Suppose $\G$ is symmetric and connected and the interpolation function $\theta$ satisfies Assumption \ref{interpolationassumption}.
\begin{enumerate}[(i)]

\item \label{lem:ComparisonMetricsDiscrete}
 There exists  a constant $C_{\theta,\G}$  so that
\begin{equation*}
\frac{1}{\sqrt{2}} | p_1- p_0 | \leq W_\G(p_0, p_1) \leq C_{\theta, \G} \sqrt{ | p_1 - p_0 |}, \quad \forall p_0,p_1 \in \mathcal{P}(\mathcal{G}),
\end{equation*}
where $| \cdot |$ is the   Euclidean norm on $\P(\G) \subseteq \R^n$. In particular, $(\P(\G), W_\G)$ has finite diameter, and the topology induced by $W_\G$ on $\P(\G)$ is equivalent to the   Euclidean topology on $\P(\G)  \subseteq \R^n$, so   
$(\P(\G), W_\G)$ is a complete metric space.
\item \label{WGgeodesics} For any $p_0, p_1 \in \P(\G)$ there exists a constant speed geodesic $p \in AC(0,1; \P(\G))$ from $p_0$ to $p_1$. For any such constant speed geodesic, there exists a velocity field $v$ to that it solves the graph continuity equation and   satisfies
\begin{align*}   W_\G(p_t,p_s) = |t-s| W_\G(p_0, p_1) \ \text{ for all } s,t \in [0,1]\\ \| v_t \|_{\Tan_{p_t}\P(\G)} = W_\G(p_0,p_1) \ \text{ for a.e. } t \in [0,1] . \nonumber 
\end{align*}
\item The restriction of $W_\G$ to $(\P(\G))^\circ$ is a Riemannian distance induced by the following Riemannian structure: \label{riemannianpart}
\begin{enumerate}[(a)]
\item the tangent space of $p_0 \in (\P(\G))^\circ$ can be identified with the set 
\[ {\rm Tan}_{p_0} \P(\G) = \left\{ \nabla_\G \psi  : \psi \in \R^n , \ \sum_{i=1}^n \psi_i = 0 \right\} , \]
in the sense that, for any smooth curve $p:(-\epsilon,\epsilon) \to (\P(\G))^\circ$ with $\left. p_t \right|_{t =0} = p_0$,  $\exists ! \nabla_\G \psi_0 \in {\rm Tan}_{p_0} \P(\G) $ so that $(p, \nabla_\G \psi)$ solves the continuity equation on the graph at $t=0$;
\item the Riemannian metric on $\nabla_\G \psi_0, \nabla_\G \varphi_0 \in {\rm Tan}_{p_0} \P(\G)$ is given by $\la \nabla_\G \psi_0, \nabla_\G \varphi_0 \ra_{{\rm Tan}_{p_0} \P(\G)}$.

\end{enumerate}
\end{enumerate}
\end{prop}

\begin{rem}[pathologies of geodesics] \label{geodesicintersectingtheboundary} For $\theta$ the arithmetic interpolation function, and $\G$ a tree graph on three nodes, there is an example of a constant speed geodesic connecting two interior points $p_0,p_1 \in (\P(\G))^\circ$ that intersects the boundary of $(\P(\G))^\circ$; see \cite[Proposition 3.11.]{gangbo2019geodesics}.
\end{rem}

We close this section with two elementary lemmas that show how curves in $\P(\G)$ can be regularized to avoid the boundary. These constructions are inspired by \cite[Corollary 2.8]{erbar2012ricci}; the proofs are in appendix \ref{realappendix}.

\begin{lem}
\label{lem:MollificationSimplex}
Suppose $\G$ is symmetric and connected and the interpolation function $\theta$ satisfies Assumption \ref{interpolationassumption}.
Let $  (p, v)$ and $ (z, u)  $ be   solutions to the graph continuity equation on the time interval $[0,1]$
. Suppose  that $ z_t \in (\P(\G))^\circ$ for all $t \in [0,1]$, and for  $a\in (0,1)$, define
\[ p_{t}^{a}:= (1-a) p_t + a z_t, \quad t \in [0,1]. \]
Then   there exists $v^{a} $ so   $ (p^{a}, v^{a})$ solves the graph continuity equation on the time interval $[0,1]$ and 
\begin{equation}
\|v_{t}^{a}\|_{\Tan_{p_{t}^{a}} \P(\G)}^2  \leq (1-a)  \|v_{t}\|_{\Tan_{p_{t}} \P(\G)}^2 + a\|u_{t}\|_{\Tan_{z_{t}} \P(\G)}^2 , \quad \forall t \in [0,1] , a \in (0,1) . 
   \label{eqn:LengthsConvexComb}
\end{equation}
\end{lem}

As an immediate corollary, the following result shows that any solution of the graph continuity equation can be perturbed to lie in the interior $(\P(\G))^\circ$, in a manner that only increases the kinetic energy of the velocity by an arbitrarily small amount; again, we defer the proof to appendix \ref{realappendix}.
\begin{cor}
\label{cor:IneqConvex}
Suppose $\G$ is symmetric and connected and the interpolation function $\theta$ satisfies Assumption \ref{interpolationassumption}.
Let $  (p, v)$ be a solution of the graph continuity equation on the time interval $[0,1]$. Given   $ {m}_0 , {m}_1 $ in $(\P(\G))^\circ$  and  $a\in(0,1)$,  define 
\[ z_t := (1-t){m}_0 + t {m}_1 \quad \text{ and } \quad  p_{t}^{a}:= (1-a)p_t + a z_t .  \]
Then there exists $ v^{a} $ so   $    (p^{a}, v^{a}) $ solves the graph continuity equation on the time interval $[0,1]$ and
\[ p^a_{i,t} \geq C_{m_0,m_1, \G}  \text{ and }\|v_{t}^{a}\|_{\Tan_{p_{t }^{a}} \P(\G)}^2  \leq (1-a)   \|v_{t}\|_{\Tan_{p_{t}} \P(\G)}^2    + a C_{ {m}_0,  {m}_1, \G}, \quad \forall t \in [0,1], \ i \in \G , \ a \in (0,1), \]
where   $C_{ {m}_0,  {m}_1,\G}>0$    depends on  $\min_i  m_{i,0} >0$, $ \min_i m_{i,1}>0$, and  the graph $\G$.
\end{cor}

\subsection{Induced geometry on simplex} \label{Deltametricsection}
Due to the bijection $\bp$ between $\Delta^{n-1}$ and $\mathcal{P}(\G)$ defined in \eqref{Idef}, the graph Wasserstein distance $W_\G$ induces  a distance on $\Delta^{n-1}$ according to
\begin{equation*}
   d_{\Delta^{n-1}}(r, \tilde r):= W_\G(\bp(r), \bp(\tilde r)). 
\end{equation*}
As an immediate consequence of Proposition \ref{basicfactWG}, we deduce several basic properties.
\begin{lem}
\label{lem:EquivalenceTopologies}
Suppose $\G$ is symmetric and connected and the interpolation function $\theta$ satisfies Assumption \ref{interpolationassumption}.
 $(\Delta^{n-1},d_{\Delta^{n-1}})$ had finite diameter and a topology is equivalent to the Euclidean topology. In particular, it is a complete and separable metric space.
\end{lem} 
\begin{proof}
The result follows from Proposition \ref{basicfactWG} and the fact that $\bp$ is bi-Lipschitz.
\end{proof}

We now describe how, through the isometry $\bp$, $(\Delta^{n-1})^\circ$ inherits the Riemannian structure of $(\P(\G))^\circ$.
For $p \in \P(\G)$, let $B(p)$ denote the $p$-weighted graph Laplacian,
\begin{align} \label{Bpdef}
B: \P(\G) \to M_{n \times n},  \ \ B(p)\phi :=  -\nabla_\G \cdot ( \check{p} \nabla_\G \phi) ,
 \ \  B(p)_{ij} := \begin{cases} \sum_{k \neq i} \theta(p_i,p_k) q_{ik}   &\text{ if } i = j , \\ -\theta(p_i,p_j) q_{ij} &\text{ otherwise.} \end{cases} \end{align}
Then, as shown in \cite[Lemma 3.17]{maas2011gradient}, for all $p \in (\P(\G))^\circ$, 
\begin{align*}
{\rm Ker}B(p) &= \{ c \textbf{1} \in \R^n: c \in \R \} ,  \quad {\rm Ran }  B(p) =   \left\{ \psi \in \R^{n} : \sum_{i=1}^n \psi_i = 0 \right\} .
\end{align*}
 The left inverse of $B(p)$, denoted $ B(p)^\dagger :{\rm Ran }  B(p)  \to  \R^n   / {\rm Ker}B(p)$, is a bijection, and choosing representatives of each equivalence class, we may identify  $\R^n   / {\rm Ker}B(p) \cong  \left\{ \psi \in \R^{n} : \sum_{i=1}^n \psi_i = 0 \right\} $.
Likewise, since $\bp(r)$ is linear, its  Jacobian determinant is independent of $r$ and given by 
\begin{align*} 
\Xi&:  \mathbb{R}^{n-1} \to {\rm Ran }  B(p) :   [f_i]_{i=1}^{n-1} \to \left[f_1, f_2, \dots, f_{n-1}, - \sum_{i=1}^{n-1} f_i \right] . \end{align*}

As described in Proposition \ref{basicfactWG}(\ref{riemannianpart}), for any $p_0 \in (\P(\G))^\circ$, smooth curves $p: (-\epsilon, \epsilon) \to (\mathcal{P}(\G))^\circ$ with $\left. p \right|_{t=0} = p_0$ are related to elements of the tangent space $\nabla_\G \psi_0 \in {\rm Tan}_{p_0} \P(\G)$ via
\[ \dot{p}_0 = B(p_0) \psi_0 \iff B(p_0)^\dagger \dot{p}_0 = \psi_0 . \]
 Thus, for any $r_0 \in  (\Delta^{n-1})^\circ$, smooth curves $r:(-\epsilon, \epsilon) \to  (\Delta^{n-1})^\circ$ with $\left. r \right|_{t=0} =r_0$ are in one to one correspondence with smooth curves $\bp(r): (-\epsilon, \epsilon) \to (\P(\G))^\circ$, and the elements of the tangent space ${\rm Tan}_{r_0}(\Delta^{n-1}) \cong \R^{n-1}$ are related to elements of the tangent space ${\rm Tan}_{p_0} \P(\G)$ by 
 \[ \Xi \dot{r}_0 =  \left.  \frac{d}{dt} \bp(r_t) \right|_{t=0} =  B(\bp(r_0)) \psi_0 \iff B(\bp(r_0))^\dagger \Xi \dot{r}_0 = \psi_0 . \]
 Thus, the Riemannian structure on $(\P(\G))^\circ$ induces a   Riemannian metric for  $r \in (\Delta^{n-1} )^\circ$,
 \begin{align} \label{metricDeltatoEuclidean} \la s, \tilde{s} \ra_{{\rm Tan}_r \Delta^{n-1}} =  \la   \nabla_\G B(\bp(r))^\dagger \Xi  {s}, \nabla_\G B(\bp(r))^\dagger \Xi \tilde{s} \ra_{{\rm Tan}_{\bp(r)}\P(\G)} = \la   s,  \Xi^t B(\bp(r))^\dagger \Xi \tilde{s} \ra_{\R^{n-1}} .
 \end{align}
 

 The gradient with respect to this Riemannian structure for $r \in (\Delta^{n-1})^\circ$ is
 \begin{align} \label{gradientdeltaformula} \nabla_{\Delta^{n-1}} h(r) = \Xi^{-1} B(\mathbf{p}(r)) (\Xi^{-1})^t \nabla_{\R^{n-1}} h(r)  . 
 \end{align}
 
We close this section with the following lemma regarding the structure of geodesics in $\Delta^{n-1}$ and a regularization whereby geodesics may be approximated by curves that avoid the boundary of the simplex. For $c>0$,   define
\begin{align} \label{strictinteriordef}\Delta^{n-1}_{c}:= \left\{ r \in \Delta^{n-1} \text{ s.t. } r_j \geq c \quad \forall j=1, \dots, n-1 \text{ and } \sum_{j=1}^{n-1} r_j \leq 1 - c  \right\}. 
\end{align}
The lemma is a direct consequence of Corollary \ref{cor:IneqConvex} and standard metric geometry arguments, so we defer the proof to appendix \ref{realappendix}.
\begin{lem}[Geodesics in $\Delta^{n-1}$]   
Suppose $\G$ is symmetric and connected and the interpolation function $\theta$ satisfies Assumption \ref{interpolationassumption}.
 \label{simplexgeodesicapproximationlem}
\begin{enumerate}[(i)]
\item $\forall \ r_0, r_1 \in \Delta^{n-1}$, there exists a constant speed geodesic $\gamma_{r_0, r_1}\in AC(0,1;  \Delta^{n-1})$ from $r_0$ to $r_1$. \label{existenceofconstantspeedgeodesicDelta}
\item The map $ \gamma: \Delta^{n-1} \times \Delta^{n-1} \to AC(0,1; \Delta^{n-1}) : (r_0, r_1) \mapsto \gamma_{r_0, r_1}  $
can be chosen to be Borel measurable. \label{measurabilityofgeodesicmap}
\item For any   geodesic $\gamma_{r_0, r_1}$,  $a\in(0,1)$, and $s_0, s_1 \in (\Delta^{n-1})^\circ$, there exists $\gamma_{r_0,r_1}^a \in AC(0,1; \Delta^{n-1}) $, defined by 
  \begin{align} \label{gammaa}  \gamma_{r_0, r_1,t}^a := (1-a) \gamma_{r_0,r_1,t} + a((1-t) s_0 + t s_1) 
  \end{align}
that satisfies
\begin{align}\label{eq:AwayBoundary}
&\gamma_{r_0, r_1,t}^{a} \in \Delta^{n-1}_{a C } \quad \text{ and } \quad
\lim_{a\rightarrow 0}  \gamma_{r_0, r_1,t}^{a} =  \gamma_{r_0, r_1,t}  , \text{ for all } t \in [0,1], \\
    \label{eq:AuxRegularized1}
  &  \left| ({\gamma}^{a}_{r_0, r_1})' \right|^2_{\Delta^{n-1}}(t)  \leq   (1-a) d_{\Delta^{n-1}}^2(r_0,  r_1)    + aC  , \text{ for almost every }t \in [0,1],
\end{align}
where $C= C_{ {s}_0,  {s}_1,\G}>0$ only depends on the minimal distance of $s_0$ and $s_1$ to the boundary of the simplex, and the graph $\G$.
\end{enumerate}
\end{lem}

%

%
%

\subsection{Induced geometry on product space}
The Riemannian metric induced on $(\Delta^{n-1})^\circ$ by the Wasserstein metric on the graph, as recalled in the previous section, leads to a natural Riemannian structure on $\Rd \times (\Delta^{n-1})^\circ$. Specifically, for $x \in \Rd$ and $r \in (\Delta^{n-1})^\circ$, we take the tangent space to be
  \begin{equation*}
    \Tan_{(x,r)} \Rd \times \Delta^{n-1}     = \Tan_{x}\Rd \oplus \Tan_{r}\Delta^{n-1} \cong \Rd \oplus \R^{n-1} ,
\end{equation*}
with the  Riemannian metric  
 \begin{align} \label{Ninnerproduct} \la (y,\phi), (z,\psi) \ra_{\Tan_{(x,r)}\Rd \times \Delta^{n-1}} = \langle y, z \rangle_{\Rd}+  \la \psi, \phi \ra_{{\rm Tan}_r\Delta^{n-1}},
 \end{align} 
 where $\langle y, z \rangle_{\Rd}$ is the usual Euclidean inner product on $\R^{d}$. 

This Riemannian structure induces the product space distance on $\Rd \times \Delta^{n-1}$,
 \[ d^2_{\Rd \times \Delta^{n-1}}((x,r), (\tilde{x},\tilde{r})):= |x - \tilde{x}|^2 + d_{\Delta^{n-1}}^2(r,\tilde{r}) . \]
 Due to the inherited geometric and topological properties from $(\P(\G), W_\G)$ —see Proposition \ref{basicfactWG}, and Lemma \ref{lem:EquivalenceTopologies}—$(\Rd \times \Delta^{n-1}, d_{\Rd \times \Delta^{n-1}})$ is a complete, separable metric space, with topology equivalent to the usual Euclidean topology. 
As a consequence, for any $h: \Rd \times \Delta^{n-1} \to \R$ continuously differentiable, its gradient at $(x, r) \in \Rd \times (\Delta^{n-1})^\circ$  is given by 
\begin{align}
    \nabla_{\Rd \times \Delta^{n-1}} h(x,r) &= \left( \nabla_{x} h(x,r),\nabla_{\Delta^{n-1}} h(x,r) \right)  .
\end{align}

Using $d_{\Rd \times \Delta^{n-1}}$ as the ground distance, one can define a notion of Wasserstein distance on   $\P(\Rd \times \Delta^{n-1})$ via a classical Kantorovich formulation \cite[equation 6.0.2]{ambrosiogiglisavare}:  given $\lambda, \tilde{\lambda} \in \P(\Rd \times \Delta^{n-1})$, we define
\begin{align}
\begin{split}
    W_{\Rd \times \Delta^{n-1}}^{2}(\lambda,\tilde{\lambda}) := \min_{\substack{ \Gamma \in \P((\Rd \times \Delta^{n-1}) \times (\Rd \times \Delta^{n-1})) \\  \pi_1 \# \Gamma = \lambda, \ \pi_2 \# \Gamma = \tilde{\lambda} }}   \int   d^2_{\Rd \times \Delta^{n-1}}((x,r), (\tilde{x},\tilde{r}))  \ d\Gamma ((x,r),(\tilde{x},\tilde{r})),
    \end{split}
    \label{eqn:Kantorovich}
\end{align}  
noticing that the above may be equal to $+\infty$ if the measures don't have finite second moments. Here, $\pi^i: (\Rd \times \Delta^{n-1}) \times (\Rd \times \Delta^{n-1})$ denote the standard projections, so that $\pi^1(x,r,\tilde{x},\tilde{r})= (x,r)$ and $\pi^2((x,r,\tilde{x},\tilde{r})) = (\tilde{x},\tilde{r})$.

We likewise have a notion of continuity equation on $\Rd \times (\Delta^{n-1})^\circ$.

\begin{defn}[Weak solution of continuity equation on $\Rd \times (\Delta^{n-1})^\circ$] \label{def:ContEqN} A pair $(\lambda, \bw)$ is a weak solution of the   continuity equation on $\Rd \times (\Delta^{n-1})^\circ$, for the time interval $[0,T]$, if 
\begin{enumerate}[(i)]
\item $\lambda \in C([0,T]; \P(\Rd \times \Delta^{n-1}))$, $\supp (\lambda) \subseteq \Rd \times (\Delta^{n-1})^\circ$; \label{continuityNinterior}
\item $\bw(x,r,t) = (w_1(x,r,t), w_2(x,r,t))$ for $(w_1,w_2):[0,T] \times \Rd \times \Delta^{n-1} \to \Rd \times \R^{n-1}$ measurable and  $ \int_0^T \int_{\Rd \times \Delta^{n-1}} \| \bw(x,r,t) \|_{\Tan_{(x,r)}(\Rd \times \Delta^{n-1})} d \lambda_{t}(x,r) dt <+\infty$. \label{continuityNgrowth}
\item The continuity equation  holds in the sense of distributions: for all $\eta \in C^\infty_c(\Rd \times \Delta^{n-1} \times [0,T])$,
  \begin{align} \label{ctyupstairs}
&\int_0^T \int_{\Rd \times \Delta^{n-1}} \partial_t \eta(x,r,t) d \lambda_t(x,r) dt \\
&\quad  = - \int_0^T \int_{\Rd \times \Delta^{n-1}}  \la \nabla_{\Rd \times \Delta^{n-1}} \eta(x,r,t) , \bw(x,r,t) \ra_{\Tan_{(x,r)}(\Rd \times \Delta^{n-1})}  d\lambda_{t}  (x,r ) dt    .\nonumber
 \end{align}
\end{enumerate}
\end{defn}
\begin{rem}[Alternative form of continuity equation on $\Rd \times (\Delta^{n-1})^\circ$] \label{altformlifted}
An immediate consequence of the previous definition is that, if $(\lambda,\bw)$ is a solution of the continuity equation on  $\Rd \times (\Delta^{n-1})^\circ$, then  for any $\eta \in C^\infty_c(\Rd \times \Delta^{n-1})$, the function $t \mapsto \int_{\Rd \times \Delta^{n-1}} \eta d \lambda_t$ is absolutely continuous and we have, in $D(0,T)'$,
 \begin{align} \label{dualCTYN}
\frac{d}{dt} \int_{\Rd \times \Delta^{n-1}} \eta(x,r) d \lambda_t(x,r) = \int_{\Rd \times \Delta^{n-1}} \la \nabla_{\Rd \times \Delta^{n-1}} \eta(x,r) , \bw(x,r) \ra_{\Rd \times \Delta^{n-1}} d\lambda_t  (x,r) \ .
 \end{align}
 \end{rem}

\section{Vector valued optimal transport}
\label{sec:Distances}

We now build on the definition of the 2-Wasserstein metric on $\P(\G)$, recalled in the previous section, to define vector valued optimal transport distances between $\bmu, \bnu \in \mathcal{P}(\Omega \times \G)$. By leveraging the geometric structure of the graph, we develop a framework that unites dynamic and static approaches.

 \subsection{Dynamic distance} \label{subsection:Dynamicdistances}

 As described in the introduction, the first notion of vector valued Wasserstein distance that we consider is via a dynamic formulation,
 inspired by   work by Chen, Georgiou, and Tannenbaum \cite{chen2018vector}, but using the graph geometry introduced by Maas \cite{maas2011gradient}.
 We begin by defining our notion of solution to the vector valued continuity equation, as in equation (\ref{vvcty}).
 
\begin{defn}[Vector valued continuity equation] A triple $(\brho, \bu, \bv)$ is a  solution of the vector valued continuity equation on the time interval $[t_0,t_1]$ if 
\begin{enumerate}[(i)]
\item $\brho \in C([t_0,t_1]; \P(\Omega\times \G))$, where $\mathcal{P}(\Omega \times \G)$ is endowed with the topology of narrow convergence;
\item $\bu:    \Rd \times [t_0,t_1]  \to (\mathbb{R}^{d})^n    $ is   measureable and $ \int_{t_0}^{t_1} \int_\Omega |u_{i,t}(x)| d \rho_{i,t}(x) dt <+\infty$ for all $i = 1, \dots, n$;
\item $\bv:   \Rd \times [t_0,t_1]\to M_{n \times n}(\R)$ is   measurable and $\int_{t_0}^{t_1} \int_\Omega |v_{ij,t}(x)| \theta \left( \frac{ d \rho_{i,t}}{d \bar{\rho}_t }(x),  \frac{ d \rho_{j,t}}{d \bar{\rho}_t}(x) \right)  d \bar{\rho}_t(x) dt < +\infty$ for all $i,j = 1, \dots, n$, where $\bar{\rho} :=  \sum_{k=1}^n \rho_k \nonumber$;
\item for all $\eta \in C^\infty_c( \Rd \times [t_0,t_1])$ and $i =1 , \dots, n$,
  \begin{align} \label{dualvvcty2}
&\int_{t_0}^{t_1} \int_{\Rd} \partial_t \eta(x,t) d \rho_{i,t}(x) dt  \\
 &\quad= - \int_{t_0}^{t_1} \int_{\Rd}  \nabla \eta(x,t) \cdot u_{i,t}(x)   d\rho_{i,t}  (x ) dt  \nonumber \\
&\quad \quad\quad  - \frac12 \sum_{j = 1}^n \int_{t_0}^{t_1} \int_{\Rd} \eta(x,t) \theta \left( \frac{ d \rho_{i,t}}{d \bar{\rho}_t }(x),  \frac{ d \rho_{j,t}}{d \bar{\rho}_t}(x) \right) (v_{ij,t}(x) - v_{ji,t}(x)) q_{ij}  d \bar{\rho}_t(x) dt. \nonumber 
& 
 \end{align}
 
\end{enumerate}
 \label{def:ContEqMultiSpecies}
\end{defn}

\begin{rem}[Alternative form of vector valued continuity equation] \label{altcty}
As in Remark \ref{altformlifted}, if  $(\brho, \bu, \bv)$ solves the vector valued continuity equation, then, for all $\eta \in C^\infty_c(\Rd)$, the function $t \mapsto \int \eta d \rho_{i,t}$ is absolutely continuous and we have, in $\mathcal{D}(0,T)'  $,
   \begin{align} \label{dualvvcty}
 \frac{d}{dt} \int_{\Rd} \eta(x) d \rho_{i,t}(x)   &= \int_{\Rd}  \nabla \eta(x) \cdot u_{i,t}(x)   d\rho_{i,t}  (x)   \\
&\quad  \quad  + \frac12 \sum_{j = 1}^n \int_{\Rd} \eta(x) \theta \left( \frac{ d \rho_{i,t}}{d \bar{\rho}_t }(x),  \frac{ d \rho_{j,t}}{d \bar{\rho}_t}(x) \right) (v_{ij,t}(x) - v_{ji,t}(x)) q_{ij}  d \bar{\rho}(x). \nonumber 
 \end{align}
 \end{rem}

With this notion of continuity equation in hand, our  dynamic vector valued distance $W_{\Omega \times \G}$, can now be defined via a least action principle, as shown in equation (\ref{dynamic}),

\begin{align*} 
W^2_{\Omega\times \G}(\bmu, \bnu) := \inf \left\{  \int_0^1 \| (\bu_t,\bv_t)\|^2_{{\brho_t}} dt  :  (\brho,\bu,\bv) \in \mathcal{C}(\bmu,\bnu)  \right\} ,
\end{align*}
where $\mathcal{C}(\bmu,\bnu)$ denotes solutions of the vector valued continuity equation on $[0,1]$ that flow from $\bmu$ to $\bnu$, see equation (\ref{dynamicconstraintdef}), and for   $(\brho,\bu,\bv) \in \mathcal{C}(\bmu,\bnu)$, we abbreviate the action by
\begin{align*}  
 \| (\bu ,\bv )\|^2_{\brho} := \sum_{i=1}^n \int_\Omega | u_i|^2   d \rho_i  +  \frac12 \sum_{i,j=1}^n \int_\Omega |v_{ij} |^2 \theta \left(\frac{d\rho_i}{d\bar{\rho}},\frac{d\rho_j}{d\bar{\rho}} \right) q_{ij} d\bar{\rho}.
 \end{align*}
 

\begin{rem}[Restriction to gradients of potentials] \label{gradsofpot}
%
In fact, the definition of $W_{\Omega \times \G}$  is unchanged if we restrict the vector fields $(\bu,\bv) $ in our continuity equation to be given by (limits of) gradients of potentials, where the notion of gradient is the one associated to the geometric structure of  $\Omega \times \G$. Given $\bmu, \bnu \in \P(\Omega \times \G)$, suppose $(\brho,\bu,\bv) \in C(\bmu,\bnu)$. In the spatial domain, the same argument as in Ambrosio, Gigli, and Savar\'e \cite[Proposition 8.4.5]{ambrosiogiglisavare} shows there exists $\tilde{\bu}$  satisfying
 \[ \tilde{u}_i(x,t) \in \overline{ \{ \nabla \varphi : \varphi \in C^\infty_c(\Rd) \}}^{L^2(\rho_{i,t})} , \quad \text{ for almost every $t \in [0,1]$ and for all }i=1, \dots, n, \]
 for which $(\brho,\tilde{\bu}, \bv) \in \mathcal{C}(\bmu,\bnu)$ and $   \int_0^1 \| (\tilde{\bu}_t,\bv_t)\|^2_{{\brho_t}} dt \leq   \int_0^1 \| (\bu_t,\bv_t)\|^2_{{\brho_t}} dt$.
 
  On the graph, we see that, for each pair $(i,j)$, we may add the same function $f_{t}$ to both $v_{ij,t}$ and $v_{ji,t}$,  and the mutation term is unchanged. Since $\theta$ and $q_{ij}$ are   symmetric and
\[  \min_{f}  \frac12 \int_\Omega \left( |v_{ij}+ f|^2 +  |v_{ji}+ f|^2\right) \theta \left( \frac{d \rho_i}{d \bar{\rho}} , \frac{d \rho_j}{d \bar{\rho}} \right)q d \bar{\rho}\]
is minimized by $f = -(v_{ij} + v_{ji})/2$, we see that, replacing each $v_{ij}$ and $v_{ji}$ by $v_{ij} +f$ and $v_{ji}+f$, there exists $\tilde{\bv}$ satisfying 
\[ \tilde{v}_{ij,t} = -\tilde{v}_{ji,t}, \text{ for almost every $t \in [0,1]$ and for all $i,j, = 1, \dots, n ,$} \]
for which $( \brho, \tilde{\bu},\tilde{\bv}) \in \mathcal{C}(\bmu,\bnu)$ and $   \int_0^1 \| (\tilde{\bu}_t,\tilde{\bv}_t)\|^2_{{\brho_t}} dt \leq   \int_0^1 \| (\bu_t,\bv_t)\|^2_{{\brho_t}} dt$.  
\end{rem}

\begin{rem}[Two node case] \label{twonoderemark}

For a graph $\G$ with   two nodes, suppose   $v_{12}(x) = - v_{21}(x)$ and $q_{ij}  \equiv  q \geq 0$. Then, for $\eta \in C_c^\infty(\Rd)$,  the mutation term in the vector valued continuity equation becomes
\begin{align} \label{simplergraphdivergencetwonodes}
& \frac12 \sum_{j = 1}^2 \int_{\Omega} \eta  \theta \left( \frac{ d \rho_1}{d \bar{\rho} } ,  \frac{ d \rho_j}{d \bar{\rho}}  \right) (v_{1j}  - v_{j1} )) q_{1j}  d \bar{\rho} =   \int_{\Omega} \eta  \theta \left( \frac{ d \rho_1}{d \bar{\rho} } ,  \frac{ d \rho_2}{d\bar{\rho}}  \right)  v_{12}   q d\bar{\rho}   ,  \nonumber
\end{align}
and the energy becomes
\begin{align*} 
\| (\bu ,\bv )\|_{\brho}^{2} &= \int_\Omega |u_1|^{2}  d {\rho}_1  +  \int_\Omega |u_2|^{2}  d {\rho}_2  +  \int_\Omega v_{12}^{2}\theta \left( \frac{d\rho_1}{d \bar{\rho}} ,\frac{d\rho_2}{d\bar{\rho}} \right) q  d\bar{\rho}.
\end{align*}

\end{rem}

In our first lemma, we show that we may reparametrize solutions of the vector valued continuity equation in time.   The proof is analogous to \cite[Thm 8.1.3]{ambrosiogiglisavare}, and we defer it to appendix \ref{realappendix}.

\begin{lem}[Time Rescaling and Gluing] Suppose $\G$ is symmetric and connected,   $\theta$ satisfies Assumption \ref{interpolationassumption}, and $\Omega \subseteq \Rd$ is closed.  \label{timerescaling}
 \begin{enumerate}[(i)]
 \item   Let $\tau:[0,T] \rightarrow [s_0,s_1]$ be a strictly increasing absolutely continuous function with absolutely continuous inverse $\tau^{-1}.$ Then $(\brho,\bu,\bv)$ is a distributional solution of vector valued continuity equation \ref{dualvvcty2} on $[s_0,s_1]$ if and only if \begin{equation*}
        \tilde{\brho}_t:= \brho_{\tau(t)}, \hspace{1mm} \tilde{\bu}_{t} := \tau'(t) \bu_{\tau(t)}, \hspace{1mm}  \tilde{\bv}_{t} := \tau'(t) \bv_{\tau(t)},
    \end{equation*} is a   solution of vector valued continuity equation \ref{dualvvcty2} on $[0,T].$
    
    \item   Let $\bmu,\bnu,\bsigma \in \mathcal{P}(\Omega \times \G),$ and
    let $(\brho^{1},\bu^{1},\bv^{1}) \in \mathcal{C}(\bmu,\bsigma),$ and $(\brho^{2},\bu^{2},\bv^{2})\in \mathcal{C}(\bsigma,\bnu)$. Denote the measure $\brho$ and the velocity fields $\bu$, $\bv$   by 
\begin{align*}
(\brho_{t}, \bu_{t},\bv_{t} ) &= \begin{cases} (\brho_{2t}^{1},2\bu_{2t}^{1}, 2\bv_{2t}^{1}) &\text{ if $t\in [0,\frac{1}{2}]$,} \\
(\brho_{2t-1}^{2},2{\bu}_{2t-1}^{2},2 \bv_{2t-1}^{2}) &\text{ if $t\in (\frac{1}{2},1],$} \end{cases}
\end{align*} Then, $(\brho,\bu,\bv) \in \mathcal{C}(\bmu,\bnu)$ and \begin{equation*}
    \int_0^1 \| (\bu_t,\bv_t)\|_{{\brho_t}}dt = \int_0^1 \|(\bu_{t}^{1},\bv_{t}^{1})\|_{{\brho_{t}^{1}}}dt + \int_0^1 \|(\bu_{t}^{2},\bv_{t}^{2})\|_{{\brho_{t}^{2}}}dt.
\end{equation*}
\end{enumerate}
\end{lem}

In our next lemma, we collect two basic properties of solutions of the vector valued continuity equation,   estimating the (distributional) regularity in time and identifying the equation satisfied by the sum of the measures $\bar{\rho} = \sum_{i=1}^n \rho_i$. Both properties follow immediately from the definition; see appendix \ref{realappendix} for the proof.
\begin{lem}[Properties of vector valued continuity equation] \label{energyestimate}
   Suppose $\G$ is symmetric and connected, $\theta$ satisfies Assumption \ref{interpolationassumption}, and $\Omega \subseteq \Rd$ is closed.   Consider $\bmu,\bnu \in \P(\Omega \times \G)$ and $(\brho,\bu,\bv)\in \C(\bmu,\bnu)$. Then, for all $0\leq s_{0}\leq s_{1}\leq 1$ and $\eta \in C^\infty_c(\Rd),$ 
\begin{enumerate}[(i)]
\item if   $v_{ij,t}(x) = -v_{ji,t}(x)$ for all $x \in \Omega$ and $t\in[0,1]$, then for all $i =1 , \dots, n$, \\
$     \left| \int_\Rd \eta d\rho_{i,s_{1}}  - \int_\Rd \eta d\rho_{i,s_{0}}  \right|   \leq      \int_{s_{0}}^{s_{1}}\int_{\Rd}|\nabla \eta| |u_{i,\tau}|d\rho_{i,\tau} d\tau   +  \sum_{j=1}^{n}  \int_{s_{0}}^{s_{1}}\int_{\Rd}|\eta| \theta\left( \frac{d \rho_{i,\tau}}{d \bar{\rho}_{\tau}} , \frac{ d \rho_{j,\tau} }{ d \bar{\rho}_{\tau}}  \right) |v_{ij,\tau}| q_{ij}d\bar{\rho}_{\tau} d\tau$; 
 \label{energyestimatepart}

 \item $
\frac{d}{dt} \int_\Rd \eta d \bar{\rho}_t 
= \sum_{i=1}^n \int_\Rd \nabla \eta \cdot u_{i,t} d \rho_{i,t} .
$ \label{rhobarequation}

    \end{enumerate}
\end{lem}

With these lemmas in hand, our first result compares the dynamic distance $W_{\Omega \times \G}(\bmu,\bnu)$ to the bounded Lipschitz distance $d_{BL}(\bmu,\bnu)$ on $\P(\Omega \times \G)$ and the bounded Lipschitz norm $\|\bar{\mu}- \bar{\nu}\|_{BL}$ on $\P(\Rd)$, as defined in equation (\ref{BLdef}). 
\begin{prop} \label{BLlem}
Suppose $\G$ is symmetric and connected, $\theta$ satisfies Assumption \ref{interpolationassumption}, and $\Omega \subseteq \Rd$ is closed.  Given $\bmu, \bnu \in \P(\Omega \times \G)$, 
\begin{enumerate}[(i)]
\item $d_{BL}(\bmu,\bnu) \leq \max \{ 1, Q^{1/2} \}  W_{\Omega \times \G}(\bmu,\bnu)$, \quad for $Q:= { \max_{i} \sum_{j=1}^n q_{ij}}$; \label{BLvsdynamic}
\item $\|\bar{\mu} - \bar{\nu}\|_{BL} \leq W_{\Omega \times \G}(\bmu,\bnu)$. \label{BLbarrhovsdynamic}
\end{enumerate}
\end{prop}

\begin{proof}
We begin with part (\ref{BLvsdynamic}). Fix $(\brho,\bu,\bv) \in \mathcal{C}(\bmu,\bnu)$. Since we seek to bound $W_{\Omega \times \G}(\bmu,\bnu)$ from below, by Remark \ref{gradsofpot}, we may assume without loss of generality that $v_{ij}(x) = -v_{ji}(x)$ for all $x \in \Omega$. Applying  Lemma \ref{energyestimate} (\ref{energyestimatepart}), with $s_0=0$, $s_1=1$,   taking the supremum over all $\eta\in C^\infty_c(\Rd)$ such that $\|\eta\|_{W^{1,\infty}_2(\Rd)} \leq 1$, and applying the Cauchy-Schwarz inequality and Jensen's inequality, we obtain 
\begin{align*}
&\| \mu_i - \nu_i \|_{BL} \\
&\quad \leq  \sup_{\|\eta\|_{W^{1,\infty}_2(\Rd)} \leq 1} \|\nabla\eta \|_\infty \int_{0}^{1}\int_{\Rd}   | u_{i,\tau}| d\rho_{i,\tau}   d \tau + \|  \eta\|_\infty \sum_{j=1}^{n}   \int_{0}^{1}\int_{\Rd}  \theta\left( \frac{d \rho_{i,\tau}}{d \bar{\rho}_{\tau}} , \frac{ d \rho_{j,\tau} }{ d \bar{\rho}_{\tau}}  \right) |v_{ij,\tau} |q_{ij}  d\bar{\rho}_{\tau} d\tau \\
&\quad \leq  \sup_{\|\eta\|_{W^{1,\infty}_2(\Rd)} \leq 1} \sqrt{ \|\nabla\eta \|_\infty^2 +  \|  \eta\|_\infty^2} \\
&\quad \quad\quad \quad\quad \quad\quad \quad \cdot \left(\int_{0}^{1}  \left(\int_{\Rd}   | u_{i,\tau}| d\rho_{i,\tau}   \right)^2 d \tau  +   \int_{0}^{1}\int_{\Rd}\left( \sum_{j=1}^{n}  \theta\left( \frac{d \rho_{i,\tau}}{d \bar{\rho}_{\tau}} , \frac{ d \rho_{j,\tau} }{ d \bar{\rho}_{\tau}}  \right) |v_{ij,\tau} |q_{ij} \right)^{2}  d\bar{\rho}_{\tau} d\tau \right)^{1/2} 
\end{align*}
Recall that, via the definition of $d \bar{\rho}_\tau = \sum_{i=1}^n d \rho_{i, \tau}$ and Remark \ref{interpfunorder}, we have $ \sum_{i=1}^{n} \frac{d\rho_{i,\tau}}{d\bar{\rho}_{\tau}} = 1 $ and
\begin{align}  
Q_i:= \sum_{j=1}^{n} \theta\left( \frac{d \rho_{i,\tau}}{d \bar{\rho}_{\tau}} , \frac{ d \rho_{j,\tau} }{ d \bar{\rho}_{\tau}} \right) q_{ij} &\leq \frac{1}{2} \sum_{j=1}^{n} \left( \frac{d \rho_{i,\tau}}{d \bar{\rho}_{\tau}} + \frac{ d \rho_{j,\tau} }{ d \bar{\rho}_{\tau}} \right)  q_{ij}  \leq  \frac{1}{2} \sum_{j=1}^n q_{ij} .
\end{align}
If $Q_i >0$, then
\begin{align*}
\left( \sum_{j=1}^{n}   \theta\left( \frac{d \rho_{i,\tau}}{d \bar{\rho}_{\tau}} , \frac{ d \rho_{j,\tau} }{ d \bar{\rho}_{\tau}}  \right) |v_{ij,\tau}| q_{ij} \right)^2 &= \left( \sum_{j=1}^{n} \frac{\theta\left( \frac{d \rho_{i,\tau}}{d \bar{\rho}_{\tau}} , \frac{ d \rho_{j,\tau} }{ d \bar{\rho}_{\tau}}  \right) q_{ij}}{Q_{i}}   Q_{i} |v_{ij,\tau}|   \right)^{2}    \leq  \sum_{j=1}^n Q_{i}    |v_{ij,\tau}|^{2}\theta\left( \frac{d \rho_{i,\tau}}{d \bar{\rho}_{\tau}} , \frac{ d \rho_{j,\tau} }{ d \bar{\rho}_{\tau}}  \right) q_{ij} ,
\end{align*}
and, likewise, if $Q_i=0$, then $\theta\left( \frac{d \rho_{i,\tau}}{d \bar{\rho}_{\tau}} , \frac{ d \rho_{j,\tau} }{ d \bar{\rho}_{\tau}} \right) q_{ij} = 0$ for all $j =1, \dots, n$, so the quantity on the left hand side is still bounded by the quantity on the right hand side.
 
 Therefore, applying the previous bounds and  Jensen's inequality for the probability measure $d \bar{\rho}_\tau d \tau$, we obtain
\begin{align*}
d_{BL}^{2}(\bmu,\bnu) & = \sum_{i=1}^n \| \mu_{i} - \nu_{i} \|_{BL}^{2}  \\
&\leq  \sum_{i=1}^{n}\int_{0}^{1}  \left( \int_{\Rd}    |u_{i,\tau}|   \frac{d \rho_{i,\tau}}{d \bar{\rho}_\tau} d \bar{\rho}_\tau \right)^2 d \tau   + \sum_{i=1}^{n} \int_{0}^{1}\int_{\Rd} \left( \sum_{j=1}^{n} \theta\left( \frac{d \rho_{i,\tau}}{d \bar{\rho}_{\tau}} , \frac{ d \rho_{j,\tau} }{ d \bar{\rho}_{\tau}}  \right) |v_{ij,\tau}| q_{ij} \right)^{2}  d\bar{\rho}_{\tau} d\tau\\
&\leq \sum_{i=1}^{n} \int_{0}^{1}\int_{\Rd}   |u_{i,\tau}|^{2}d\rho_{i,\tau} d\tau + \frac12 \sum_{i,j=1 }^{n} 2Q_{i} \int_{0}^{1}\int_{\Rd}    |v_{ij,\tau}|^{2}\theta\left( \frac{d \rho_{i,\tau}}{d \bar{\rho}_{\tau}} , \frac{ d \rho_{j,\tau} }{ d \bar{\rho}_{\tau}}  \right) q_{ij}   d\bar{\rho}_{\tau} d\tau 
\end{align*}
Since $(\brho,\bu,\bv) \in \mathcal{C}(\bmu,\bnu)$ was arbitrary and $Q_i \leq  \frac{1}{2} \sum_{j=1}^n q_{ij} $, this completes the proof of part (\ref{BLvsdynamic}).

We conclude by showing part (\ref{BLbarrhovsdynamic}). Fix $(\brho,\bu,\bv) \in \mathcal{C}(\bmu,\bnu)$ arbitrary. For any $\eta$ such that $\| \eta \|_{W^{1,\infty}_2}\leq 1,$ applying Lemma \ref{energyestimate} (\ref{rhobarequation}), H\"older's inequality, and Jensen's inequality, \begin{align*}
    \int_{\Rd} \eta d(\bar{\mu}  -\bar{\nu})
    &= \sum_{i=1}^{n} \int_{0}^{1} \int_{\Rd} \nabla \eta \cdot u_{i,\tau} d\rho_{i,\tau} d\tau  \leq \sum_{i=1}^{n} \int_{0}^{1} \int_{\Rd}  |u_{i,\tau}| d\rho_{i,\tau} d\tau  
    \leq \left( \int_{0}^{1}  \int_{\Rd} \left( \sum_{i=1}^{n} |u_{i,\tau}|\frac{d\rho_{i,\tau}}{d\bar{\rho}_{\tau}} \right)^{2}d\bar{\rho}_{\tau} d\tau \right)^{\frac{1}{2}} .
\end{align*}
Thus, by convexity, since $\sum_{i=1}^n \frac{d \rho_{i, \tau}}{d \bar{\rho}_\tau} = 1$, \begin{equation*}
    \| \bar{\mu} - \bar{\nu}  \|_{BL,1}^{2} \leq \sum_{i=1}^{n} \int_{0}^{1} \int_{\Rd} |u_{i,\tau}|^{2}d\rho_{i,\tau} d\tau \leq \int_{0}^{1} \| (\bu,\bv)\|_{\brho_{\tau}}^{2}d\tau ,
\end{equation*}
which gives the result.

\end{proof}

We now prove that  $W_{\Omega \times \G}$ is nondegenerate, symmetric, and satisfies the triangle inequality, so that, once we   exclude the possibility $W_{\Omega \times \G}(\bmu,\bnu)=+\infty$, we will have that $W_{\Omega \times \G}$ is indeed a metric. Our proof of the triangle inequality follows the strategy of \cite[Lemma 1.1.4]{ambrosiogiglisavare} and \cite[Theorem 5.4]{dolbeault2009new}. The fact that $W_{\Omega \times \G}(\bmu,\bnu)<+\infty$ will be  an immediate consequence of our main theorem, Theorem \ref{thm:Main}, in  which we will show that it is dominated by a static Wasserstein-type metric on $\P_2(\Omega \times \G)$.

\begin{prop} \label{almostdynamicmetric}
   Suppose $\G$ is symmetric and complete, $\theta$ satisfies Assumption \ref{interpolationassumption}, and $\Omega \subseteq \Rd$ is closed. Then $W_{\Omega \times \G}$ is nondegenerate, symmetric, and satisfies the triangle inequality on $\P(\Omega \times \G)$. \end{prop}
\begin{proof}  
         Assume $W_{\Omega \times \G}(\bmu,\bnu)=0.$   Then,  Proposition \ref{BLlem}(\ref{BLvsdynamic}) ensures $d_{BL}(\bmu,\bnu)=0$, so $\bmu = \bnu$. Conversely, if $\bmu=\bnu$, it is clear that $\brho_t \equiv \bmu$ satisfies $(\brho,0,0) \in \mathcal{C}(\bmu,\bnu)$ with $\|(\bu,\bv)\|_{\brho} = 0$, so $W_{\Omega \times \G} = 0$.
         To see $W_{\Omega \times \G}(\bmu,\bnu)= W_{\Omega \times \G}(\bnu,\bmu),$ note that for any   $(\brho,\bu,\bv) \in \mathcal{C}(\bmu,\bnu)$, defining $\tilde{\brho}_t = \brho_{1-t}$, $\tilde{\bu}_t := \bu_{1-t}$, and $\tilde{\bv}_t = \bv_{1-t}$, we have $(\tilde{\brho},\tilde{\bu},\tilde{\bv}) \in \mathcal{C}(\bnu,\bmu)$ and $\|(\bu,\bv)\|_{\brho} = \|(\tilde{\bu},\tilde{\bv})\|_{\tilde{\brho}}$. 

It remains to show the triangle inequality.  First, we claim  that, when $W_{\Omega \times \G}(\bmu,\bnu)<+\infty$, \begin{equation} \label{lengthspace}
        W_{\Omega \times \G}(\bmu,\bnu) = \inf_{(\brho, \bu,\bv) \in \mathcal{C}(\bmu,\bnu)} \int_{0}^{1} \| (\bu_t,\bv_t)\|_{{\brho}_t} dt.  
    \end{equation} 
    The fact that the left hand side is greater than or equal to the right hand side is an immediate consequence of   H\"older's inequality.
     For the other inequality, fix $(\brho,\bu,\bv) \in \mathcal{C}(\bmu,\bnu)$ with $ \int_0^1 \|(\bu_t,\bv_t)\|^2_{\brho_t} dr<+\infty$, and for $\epsilon >0$  and $t \in [0,1]$,  define \begin{equation*}
        s_{\epsilon}(t):= \int_{0}^{t} (\epsilon + \| (\bu_r,\bv_r)\|_{{\brho_r}}^{2})^{\frac{1}{2}} dr.
    \end{equation*} Note that $s_{\epsilon}'(t)= (\epsilon + \| (\bu_t,\bv_t)\|_{{\brho_t}}^{2})^{\frac{1}{2}} \geq \sqrt{\epsilon} > 0$ for a.e. $ t\in [0,1]$. Letting $S_{\epsilon}:= s_{\epsilon}(1),$     the inverse map $t_{\epsilon}:[0,S_{\epsilon}] \longrightarrow [0,1]$ of $s_\epsilon(t)$ is also absolutely continuous and, by the inverse function theorem  \begin{equation*}
        t_{\epsilon}' \circ s_{\epsilon}(t) = (\epsilon + \| (\bu_t,\bv_t)\|_{{\brho_t}}^{2})^{-\frac{1}{2}}, \hspace{2mm} \mbox{a.e.} \hspace{1mm} \mbox{in} \hspace{1mm} (0,1).
    \end{equation*} 
    By Lemma \ref{timerescaling}, defining $\tilde{\brho}^{\epsilon}_s = \brho_{t_{\epsilon}(s)},$ and $(\tilde{\bu}^{\epsilon}_s,\tilde{\bv}^{\epsilon}_s) = t_{\epsilon}'(s)(\bu_{t_\epsilon(s)},\bv_{t_\epsilon(s)}) ,$  we have that $(\tilde{\brho}^{\epsilon},\tilde{\bu}^{\epsilon},\tilde{\bv}^{\epsilon})$ is a solution of the vector valued continuity equation on $[0,S_\epsilon]$ and $r \mapsto (\tilde{\brho}^\epsilon_{S_\epsilon r}, S_\epsilon \tilde{\bu}^\epsilon_{S_\epsilon r},S_\epsilon \tilde{\bv}^\epsilon_{ S_\epsilon r}  )$ is a solution on $[0,1]$. Thus,
      \begin{align*}
        W_{\Omega \times \G}^2(\bmu,\bnu) &\leq  \int_{0}^{1} S_\epsilon^2 \|(\tilde{\bu}_{S_\epsilon r}^{\epsilon},\tilde{\bv}_{S_\epsilon r}^{\epsilon} )\|_{\tilde{\brho}^\epsilon_{S_\epsilon r}}^2 dr  = S_{\epsilon} \int_{0}^{S_{\epsilon}} \|(\tilde{\bu}_s^{\epsilon},\tilde{\bv}_s^{\epsilon} )\|_{\tilde{\brho}_s}^2 ds = S_{\epsilon} \int_{0}^{S_{\epsilon}} t'_\epsilon(s)^2 \|( {\bu}_{t_\epsilon(s)}^{\epsilon}, {\bv}_{t_\epsilon(s)}^{\epsilon} )\|_{ {\brho}_{t_\epsilon(s)}}^2 ds  \\
        & = S_{\epsilon} \int_{0}^{1} \frac{\|\bu_t,\bv_t \|_{\brho_t}^{2}}{\epsilon + \|\bu_t,\bv_t \|_{\brho_t}^{2}} (\epsilon + \|\bu_t,\bv_t \|_{\brho_t}^{2})^{\frac{1}{2}}dt \leq S_{\epsilon} \int_{0}^{1}   (\epsilon + \|\bu_t,\bv_t \|_{\brho_t}^{2})^{\frac{1}{2}}dt  = S_{\epsilon} s_\epsilon(1) = S_\epsilon^2.
    \end{align*} Hence, $       W_{\Omega\times \G}(\bmu,\bnu) \leq S_{\epsilon} = \int_{0}^{1} \sqrt{\epsilon + \|\bu_t,\bv_t \|_{\brho_t}^{2}}dt$. Sending $\epsilon \to 0$, via the dominated convergence theorem, we conclude $W_{\Omega \times \G}(\bmu,\bnu) \leq \int_{0}^{1} \|\bu_t,\bv_t \|_{\brho_t}dt,$ which shows  equation (\ref{lengthspace}).

     We now use this to prove the triangle inequality. Fix $\bmu,\bnu, \bsigma \in \P_2(\Omega \times \G)$. We aim to show
     \[ W_{\Omega \times \G}(\bmu,\bnu) \leq W_{\Omega \times \G}(\bmu, \bsigma) + W_{\Omega \times \G}(\bsigma,\bnu). \]
     Without loss of generality, we may suppose that both of the terms on the right hand side are finite. Fix $\epsilon >0$ arbitrary. Then, by equation (\ref{lengthspace}),  there are $(\tilde{\brho},\tilde{\bu},\tilde{\bv}) \in \mathcal{C}(\bmu,\bsigma),$ and $(\tilde{\tilde{\brho}},\tilde{\tilde{\bu}},\tilde{\tilde{\bv}})\in \mathcal{C}(\bsigma,\bnu),$ such that \begin{align*}
        &\int_{0}^{1} \|\tilde{\bu}_t,\tilde{\bv}_t \|_{\tilde{\brho}_t} dt \leq   W_{\Omega \times \G} (\bmu,\bsigma) +\epsilon,\\
        &\int_{0}^{1} \|\tilde{\tilde{\bu}}_t,\tilde{\tilde{\bv}}_t \|_{\tilde{\tilde{\brho}}_t} dt \leq   W_{\Omega \times \G} (\bsigma,\bnu) +\epsilon
    \end{align*} Then,  after applying the gluing result from  Lemma \ref{timerescaling} we conclude, \begin{equation*}
         W_{\Omega\times \G}(\bmu,\bnu) \leq W_{\Omega\times \G}(\bmu,\bsigma) +  W_{\Omega\times \G}(\bsigma,\bnu) +2 \epsilon.
    \end{equation*} Since $\epsilon> 0$ was arbitrary, this gives the result. 
\end{proof}

\subsection{Existence of minimizers for dynamic distance} \label{subsection:Existenceofminimizers}
The goal of the present section is to show that, in the PDE constrained optimization problem by which we define the dynamic distance, equation (\ref{dynamic}), minimizers   exist whenever the distance is finite. Our strategy   generalizes the argument used in the classical case of the 2-Wasserstein metric on $\P_2(\Rd)$; see, for example, \cite[Chapter 5]{santambrogio2015optimal}. In particular, we begin in Proposition \ref{compactnesspropertyfordirectmethod} by proving that sublevel sets of the optimization problem are relatively compact in the bounded Lipschitz distance. Then, in Proposition \ref{SantambrogioGraphProp5.18}, we consider a version of the action $\|(\bu,\bv)\|_{\brho}^2$ in momentum coordinates,  which is convex and lower semicontinuous with respect to narrow convergence. Finally, in Theorem \ref{existenceofminimizers}, we combine these ingredients to prove that, whenever $W_{\Omega \times \G}(\bmu,\bnu)<+\infty$, there exists $(\brho, \bu,\bv) \in \mathcal{C}(\bmu,\bnu)$ that minimizes $ \int_0^1 \|(\bu_t,\bv_t)\|^2_{\brho_t} dt$. Since our main Theorem \ref{thm:Main} will ensure that $W_{\Omega \times \G}(\bmu,\bnu)<+\infty$ always holds, this then implies that minimizers always exist.
\begin{prop} \label{compactnesspropertyfordirectmethod}
Suppose $\G$ is symmetric and connected, $\theta$ satisfies Assumption \ref{interpolationassumption}, and $\Omega \subseteq \Rd$ is closed. Given $\bmu \in \P_2(\Omega \times \G)$,  $\bnu \in \P(\Omega \times \G)$, and   $\{(\brho^k,\bu^k,\bv^k)\}_{k \in \mathbb{N}} \subseteq \mathcal{C}(\bmu,\bnu)$, suppose that there exists $C>0$ so that 
\begin{align} \label{objfnbound} \sup_{k \in \mathbb{N}}  \int_0^1 \| (\bu^k_t,\bv^k_t)\|^2_{{\brho^k_t}} dt \leq  C  .
\end{align}
Then, there exists  $\brho \in C([0,1]; \P(\Omega\times \G))$, $\bm \in \M_s([0,1] \times\Omega )^n$, and $\bsigma  \in \M_s([0,1] \times \Omega)^{n \times n}$ so that,
 for all $i,j=1, \dots, n$ and  $\eta \in C^\infty_c([0,1]\times \Rd)$, we have
\begin{align}\label{mctyeqn}
 \int_0^1 \int_\Rd \partial_t \eta d \rho_i dt +  \int_{[0,1] \times\Rd} \nabla \eta  dm_i+ \frac{1}{2} \sum_{j = 1}^n  \int_{[0,1] \times\Rd} \eta d(\sigma_{ij} - \sigma_{ji})= 0 
\end{align}
and, up to subsequence, 
\begin{align}
&\rho_{i,t}^k \xrightarrow{k \to +\infty} \rho_{i,t} \text{ in the duality with $C_b(\Rd)$, for a.e. $t \in [0,1]$;} \label{rhoconv} \\
&u_{i,t}^k d\rho_{i,t}^k dt \xrightarrow{k \to +\infty}  m_{i} \text{ in the duality with } C_b([0,1] \times \Rd); \label{mconv} \\
&v_{ij,t}^k \theta \left( \frac{ d \rho^k_{i,t}}{d \bar{\rho}^k_t},\frac{ d \rho^k_{j,t} }{ d \bar{\rho}^k_t}  \right) q_{ij} d\bar{\rho}^k_t dt \xrightarrow{k \to +\infty}  \sigma_{ij} \text{ in the duality with } C_b([0,1] \times \Rd). \label{sigmaconv} \end{align}
\end{prop}
\begin{proof}

We begin by applying Arzel\'a-Ascoli to obtain convergence of $(\brho^k)_{k \in \mathbb{N}}$. First, we show that $(\brho^k)_{k \in \mathbb{N}}$ is   uniformly 1/2-H\"older continuous with respect to the bounded Lipschitz distance. Fix $0 \leq t_0 \leq t_1 \leq 1$.  By Lemma \ref{timerescaling}, for $\tau(t) = t_0 + t(t_1-t_0)$, $(\brho_{k,\tau(t)},\tau'(t)\bu_{k,\tau(t)},\tau'(t)\bv_{k,\tau(t)})$ is a solution of the vector valued continuity equation on $[0,1]$. Thus, Proposition \ref{BLlem} ensures that, there exists $C_\G $ depending on the edge weights $q_{ij}$ so that, for  all $ 0 \leq t_0 \leq t_1 \leq 1$,
\begin{align*}
 d_{BL}(\brho_{k,t_0},\brho_{k,t_1}) &\leq C_\G W_{\Omega \times \G}(\brho_{k,t_0},\brho_{k,t_1})  \leq \tau'(t) C_\G  \left( \int_0^1 \| (\bu_{\tau(t)}, \bv_{\tau(t)}) \|_{\brho_{\tau(t)}}^2 dr \right)^{1/2}  \leq (t_1-t_0) C_\G   C^{1/2}  .
\end{align*}
Thus, $(\brho^k)_{k \in \mathbb{N}}$ is equicontinuous with respect to the bounded Lipschitz distance.

Next, we show that there exists $C'$ so that, for all $k \in \mathbb{N}$ and $t \in [0,1]$, 
\begin{align} \label{Sdef}\brho^k_t \in S, \quad S:=\{ \bmu \in \P(\Omega \times \G) : \int |x| d \mu_i \leq C' , \ \forall i =1, \dots , n \}.
\end{align}
 This will then ensure that $\{\brho^k_t\}_{k \in \mathbb{N}}$ is  relatively compact with respect to the narrow topology for each $t \in [0,1]$ \cite[Remark 5.1.5]{ambrosiogiglisavare}, hence the bounded Lipschitz distance.

Define
\begin{align} \label{Cprimedef}
C' :=(3C + M_2(\bar{\mu})) \left( 1+ 3e^{3} \right).
\end{align}
Since $\bmu \in \P_2(\Omega \times \G)$, we see that $C'<+\infty$. We will now prove (\ref{Sdef}) by showing that, for all $k \in \mathbb{N}$, $t \in [0,1]$, and $i =1,\dots, n$, we have $\int |x| d \rho^k_{i,t} \leq C'$. Since this estimate is uniform in $k \in \mathbb{N}$, for simplicity of notation, we will suppress dependence on $k$.

Fix a smooth, radially decreasing cutoff function $\omega_R(x) = \omega(|x/R|)$,  satisfying $\omega(y) \equiv 1$ for $|y| \leq 1$, $\omega(y) \equiv 0$ for $|y| \geq 2$, and $|\omega'(y)| \leq 2$ for all $y \in \R$. Applying  Lemma \ref{energyestimate}(\ref{rhobarequation}) to the test function $\eta(x) = |x|^2 \omega_R(x)$ and integrating in time, for all $s \in [0,1]$ and $R \geq 1$, the Cauchy-Schwarz inequality and the facts   $|\omega_R(x)| \leq 1$ and $|\nabla \omega_R(x)| \leq 2R^{-1} 1_{B_{2R}(0)}(x)$ ensure 
\begin{align*}
&  \int_\Rd |x|^2 \omega_R(x) d \bar{\rho}_s (x) -  \int_\Rd |x|^2 \omega_R(x) d \bar{\mu} (x) \\
& = 2 \sum_{i=1}^n \int_0^s \int_\Rd x \cdot u_{i,t}(x) \omega_R(x) d \rho_{i,t}(x) dt + \sum_{i=1}^n \int_0^s \int_\Rd |x|^2 \nabla \omega_R(x) \cdot u_{i,t}(x) d \rho_{i,t}(x) dt  \\
&\leq  \sum_{i=1}^n \int_0^s \int_\Rd |x|^2 d \rho_{i,t}(x) dt +  \sum_{i=1}^n \int_0^s |u_{i,t}(x)|^2 d \rho_{i,t}(x) dt + 2\sum_{i=1}^n \int_0^s \int_{B_{2R}(0)} \frac{ |x|}{R}  |x| | u_{i,t}(x)| d \rho_{i,t}(x) dt\\
&\leq \int_0^s \int_\Rd |x|^2 d \bar{\rho}_t(x) dt +  \sum_{i=1}^n \int_0^s \int_\Rd |u_{i,t}(x)|^2 d \rho_{i,t}(x) dt \\
&\quad + 2 \sum_{i=1}^n   \int_0^s \int_\Rd |u_{i,t}(x)|^2 d \rho_{i,t}(x) dt + 2 \sum_{i=1}^n \int_0^s \int_\Rd |x|^2 d \rho_{i,t}(x) dt  \\
&\leq 3 \int_0^s \int_\Rd |x|^2 d \bar{\rho}_t(x) dt + 3 C .
\end{align*}

By the Monotone Convergence Theorem, we may send $R \to +\infty$  to conclude that
\begin{align*}
 M_2(\bar{\rho}_s) \leq 3 \int_0^s M_2(\bar{\rho}_t) dt + 3 C + M_2(\bar{\mu}) , \quad  \forall \ s \in [0,1] .
\end{align*}
By Gronwall's inequality, this ensures
\begin{align*}
M_2(\bar{\rho}_s) \leq (3C + M_2(\bar{\mu})) \left( 1+ 3e^{3} \right)  , \ \forall s \in [0,1] .
\end{align*}
Recalling that the above estimate applies to $\{\brho^k_{s}\}_{s \in [0,1], k \in \mathbb{N}}$, uniformly in $k \in \mathbb{N}$ shows (\ref{Sdef})

Therefore, we have shown that $(\brho^k)_{k \in \mathbb{N}}$ is uniformly 1/2-H\"older and pointwise  in time relatively compact, with respect to the bounded Lipschitz distance. Thus,  Arzel\'a-Ascoli ensures that there exists $\brho:[0,T]: \to \P(\Omega \times \G)$ such that $\rho_i:[0,T] \to \mathcal{M}(\Omega)$ is narrowly continuous for all $i = 1, \dots, n$, and $\rho_{i,t}^k \xrightarrow{k \to +\infty} \rho_{i,t}$ narrowly, for all $t \in [0,1]$. This shows equation (\ref{rhoconv}). As an immediate consequence, for any $\eta\in C_b([0,1]\times\Rd)$, by the dominated convergence theorem
\begin{align*}
\lim_{k \to +\infty} \int_0^1 \int_\Rd \eta(x,t) d \rho^k_{i,t} dt = \int_0^1 \int_\Rd \eta(x,t) d \rho_{i,t} dt ,
\end{align*}
so $\rho^k_{i,t} \otimes dt \to \rho_{i,t} \otimes dt$ narrowly in $\M([0,1] \times \Rd)$.

Now, we consider (\ref{mconv}-\ref{sigmaconv}). By hypothesis (\ref{objfnbound}) and H\"older's inequality, for any $S \subseteq [0,1] \times \Omega$,
\begin{align} \label{firstpuppy}
&\iint_{S} |u_{i,k}| d \rho^k_{i,t} dt \leq \left(\int_0^1 \int_\Omega |u_{i,k}|^2 d \rho^k_{i,t} dt \right)^{1/2}  \left(\iint_{S}  d \rho^k_{i,t} dt \right)^{1/2}    \leq \sqrt{C}   \left(\iint_{S}  d \rho^k_{i,t} dt \right)^{1/2}   ,
\end{align}
and, combining H\"older's inequality with Remark \ref{interpfunorder},
\begin{align}
&\iint_{S} |v^k_{ij,t} | \theta \left(\frac{d\rho^k_{i,t}}{d\bar{\rho}^k_t},\frac{d\rho^k_{j,t}}{d\bar{\rho}^k_t} \right) q_{ij} d\bar{\rho}^k_t dt \nonumber \\
&\quad  \leq \left( \int_0^1 \int_\Omega |v^k_{ij,t} |^2 \theta \left(\frac{d\rho^k_{i,t}}{d\bar{\rho}^k_t}, \frac{d\rho^k_{j,t}}{d\bar{\rho}^k_t} \right) q_{ij} d\bar{\rho}^k_t dt \right)^{1/2}  \nonumber
 \left( \iint_{S}  \theta \left(\frac{d\rho^k_{i,t}}{d\bar{\rho}^k_t},\frac{d\rho^k_{j,t}}{d\bar{\rho}^k_t} \right) q_{ij} d\bar{\rho}^k_t dt \right)^{1/2} \nonumber \\
 &\quad \leq \sqrt{C}   q_{ij}  \left( \iint_{S}  d\bar{\rho}^k_t dt \right)^{1/2} . \label{secondpuppy}
\end{align}

For $S = [0,1] \times \Omega$, using the fact that $\bar{\rho}^k_t = \sum_{i=1}^n \rho^k_{i,t} \in \P(\Rd)$ for all $t \in [0,1]$, we see that 
\begin{align} \label{twosequences} \{ u_{i,k} \ d\rho_{i,t}^k dt \}_{k \in \mathbb{N}} \text{ and } \left\{ v_{ij,t}^k \  \theta \left(\frac{d\rho^k_{i,t}}{d\bar{\rho}^k_t},\frac{d\rho^k_{j,t}}{d\bar{\rho}^k_t} \right) q_{ij} d\bar{\rho}^k_t dt  \right\}_{k \in \mathbb{N}} \end{align}
 are sequences of signed Borel measures that are uniformly bounded in total variation norm.

Next, note that, by Prokhorov's theorem, the convergence $\rho^k_{i,t} \otimes dt \to \rho_{i,t} \otimes dt$ narrowly in $\M([0,1] \times \Rd)$ ensures that  $\{\bar{\rho}^k_{t} \otimes dt \}_{k \in \mathbb{N}}$ is tight, so for all $\epsilon >0$, there exists $K_\epsilon \subseteq [0,1]\times \Omega$ so that $\iint_{K_\epsilon^c} d \bar{\rho}^k_t dt < \epsilon$. Taking $S = K_\epsilon^c$ in the above estimates, we likewise see that  (\ref{twosequences}) are   tight. Thus, up to a subsequence, both converge  in the duality with $C_b([0,1]\times \Rd)$ \cite[Theorem 8.6.2]{bogachev2007measure}. Letting $m_i \in \mathcal{M}_s([0,1] \times \Rd)$ and $\sigma_{i,j} \in \mathcal{M}_s([0,1] \times \Rd)$ denote the  corresponding limits, for $i, j =1, \dots n$, shows (\ref{mconv}-\ref{sigmaconv}).

Finally, since $(\brho^k,\bu^k,\bv^k) \in \mathcal{C}(\bmu,\bnu)$ for all $k \in \mathbb{N}$, for any $\eta \in C^\infty_c([0,1] \times \Rd)$ and $i= 1, \dots, n$, the weak formulation of the vector valued continuity equation (\ref{dualvvcty2}) is satisfied.
 Sending $k \to +\infty$ in this equation gives (\ref{mctyeqn}).

\end{proof}

We now study a version of the energy $\|(\bu,\bv)\|_{\brho}^2$ in momentum coordinates, inspired by the approach to proving existence of minimizers of the dynamic formulation of the classical 2-Wasserstein metric; see, e.g., \cite[Proposition 5.18]{santambrogio2015optimal}. As we will see in   Theorem \ref{existenceofminimizers} on existence of minimizers, finiteness of this new energy in momentum coordinates will allow us to conclude that the   solution of the continuity equation from the previous proposition, equation (\ref{mctyeqn}), in fact satisfies $(\brho,\bu,\bv) \in \mathcal{C}(\bmu,\bnu)$ for some choice of $\bu$ and $\bv$.

\begin{prop} \label{SantambrogioGraphProp5.18} For  $K \subseteq \mathbb{R}^{2+d}$, as in Lemma \ref{Klem}, and a closed subset $X \subseteq \R^m$ for $m \in \mathbb{N}$, define  the functional
$  \mathcal{B}_\mathcal{G}: \mathcal{M}(X) \times \mathcal{M}(X) \times \mathcal{M}_s(X)^d  \to \mathbb{R} \cup \{+\infty\}$ by
\[\mathcal{B}_\mathcal{G}(\rho_{1},\rho_{2},\bsigma) = \sup \left\{ \int_{X} a(x)d\rho_{1}(x) + \int_{X} b(x)d\rho_{2}(x) + \int_{X} c(x) \cdot d\bsigma(x) : (a,b,c)\in C_{b}(X;K) \right\} . \]
Then,  $\mathcal{B}_\mathcal{G}$ is convex and lower semi-continuous with respect to narrow convergence. Moreover, the following properties hold:
    \begin{enumerate}[(i)]
       \item $\mathcal{B}_\mathcal{G}(\rho_{1},\rho_{2},\bsigma) = \sup \{ \int a(x)d\rho_{1}(x) + \int b(x)d\rho_{2}(x) + \int c(x) \cdot  d\bsigma(x) : (a,b,c)\in L^{\infty}(X;K)\}$; \label{CtoLinfty}

       \item If $\rho_{1},\rho_{2},\bsigma \ll \lambda,$ for $\lambda \in \mathcal{M}(X)$, then  \label{absctslambda}
        $\mathcal{B}_\mathcal{G}(\rho_{1},\rho_{2},\bsigma) = \int_{X} \alpha (\frac{d\rho_{1}}{d\lambda},\frac{d\rho_{2}}{d\lambda},\frac{d\bsigma}{d\lambda})d\lambda,$ where $ \alpha$ is as in equation (\ref{definitionofpsi});

        \item $\mathcal{B}_\mathcal{G}(\rho_{1},\rho_{2},\bsigma)<+\infty $ and $\rho_1, \rho_2 \ll \tilde{\rho}$ $\implies 
         \bsigma \ll \theta(\frac{d\rho_{1}}{d\tilde{\rho}},\frac{d\rho_{2}}{d\tilde{\rho}})\tilde{\rho}$  \label{finitedistance}  

        \item Assume 
        $\bsigma \ll \theta(\frac{d\rho_{1}}{d\tilde{\rho}},\frac{d\rho_{2}}{d\tilde{\rho}})\tilde{\rho}$ and $\rho_1, \rho_2 \ll \tilde{\rho}  $. Then $\mathcal{B}_\mathcal{G}(\rho_{1},\rho_{2},\bsigma) = \int_{X} |v|^{2}\theta(\frac{d\rho_{1}}{d\tilde{\rho}},\frac{d\rho_{2}}{d\tilde{\rho}}) d\tilde{\rho},$ where $\bsigma = v \theta(\frac{d\rho_{1}}{d\tilde{\rho}},\frac{d\rho_{2}}{d\tilde{\rho}}) \tilde{\rho}.$ \label{integralformulation}
    \end{enumerate}
\end{prop}

\begin{proof}
    
The fact that $\mathcal{B}$ is convex and lower semicontinuous with respect to narrow convergence is an immediate consequence of the fact that it is a supremum of functions that are affine and continuous in the narrow topology.

We now prove part  (\ref{CtoLinfty}). Let $(a,b,c)\in L^{\infty}(X;K)$. Fix $\epsilon>0.$ Then, by Lusin's Theorem, 
 there exists  a compact set $A_{\epsilon}$ such that   $(\rho_1 + \rho_2 + |\bsigma|)(A_{\epsilon}^{c})< \epsilon$, and 
  $a|_{A_{\epsilon}},b|_{A_{\epsilon}},$ and $c|_{A_{\epsilon}}$ are continuous.
There exists $R>0$ so that $(a,b,c)\in L^{\infty}(X;K \cap B_{R}),$ and there exists  a homeomorphism $h: K \cap B_{R} \to [-1,1]^{2+d}.$ Then $h\circ (a,b,c) \in L^{\infty}(X;[-1,1]^{2+d}),$ and $h\circ (a,b,c)|_{A_{\epsilon}}$ is a continuous function. 
 By the Tietze Extension Theorem,
 there exists  $f\in C(X;[-1,1]^{2+d})$ so that $f = h\circ (a,b,c)$ on $A_{\epsilon}.$ Thus, $h^{-1}\circ f = (a,b,c)$ on $A_{\epsilon},$ and $h^{-1}\circ f \in C(X; K \cap B_{R}).$ Hence, there exist $(\tilde{a}, \tilde{b}, \tilde{c}) :=   h^{-1}\circ f \in C(X; K \cap B_{R})$ so that 
 $\tilde{a}|_{A_{\epsilon}} = a, \hspace{0.5mm} \tilde{b}|_{A_{\epsilon}} = b, \hspace{0.5mm} \tilde{c}|_{A_{\epsilon}} = c  $, and $\|\tilde{a}\|_{\infty}, \| \tilde{b}\|_{\infty}, \| \tilde{c}\|_{\infty} \leq R$. 
 Then, \begin{align*}
             \left|\int  (a-\tilde{a})d\rho_1  + \int (b-\tilde{b})d\rho_2 + \int (c-\tilde{c}) \cdot d \bsigma \right| 
         \leq \int  |a-\tilde{a}| + |b-\tilde{b}| + |c-\tilde{c}| d(\rho_1 + \rho_2 + |\bsigma|) \leq 6R\epsilon.
         \end{align*}
 Since $\epsilon$ is arbitrary, we obtain  \begin{equation*}
     \mathcal{B}_\mathcal{G}(\rho_{1},\rho_{2},\bsigma) = \sup \left\{ \int a(x)d\rho_{1}(x) + \int b(x)d\rho_{2}(x) + \int c(x) \cdot  d\bsigma(x) : (a,b,c)\in L^{\infty}(X;K) \right\}.
 \end{equation*}

We now prove part (\ref{absctslambda}).
Assume $\rho_{1},\rho_{2},\bsigma \ll \lambda,$ and recall the definitions of the functions $\alpha$ and $\beta$ from equations (\ref{definitionofpsi}-\ref{definitionofvarphi}). Then, since $\alpha$ and $\beta$ are convex conjugate functions, and $\alpha$ is a lower semicontinuous, proper, convex function, by \cite[Theorem 2]{rockafellar1968integrals}, \begin{align*}
        &\sup \left\{ \int  a d\rho_{1}  + \int  bd\rho_{2}+ \int  c\cdot d\bsigma : (a,b,c)\in L^{\infty}(X,K) \right\}\\
       &\quad  = \sup \left\{ \int a\frac{d\rho_{1}}{d\lambda}d\lambda + \int  b\frac{d\rho_{2}}{d\lambda}d\lambda + \int c \cdot \frac{d\bsigma}{d\lambda}d\lambda : (a,b,c)\in L^{\infty}_\lambda(X,K) \right\}\\
        &\quad=  \sup \left\{ \int \left[  a\frac{d\rho_{1}}{d\lambda} +  b\frac{d\rho_{2}}{d\lambda} +  c\cdot \frac{d\bsigma}{d\lambda} - \chi_K(a,b,c) \right] d\lambda  : (a,b,c)\in L^{\infty}_\lambda(X,\R^{2+d})\right\}\\
     &\quad   =\sup \left\{ \int \left[ a\frac{d\rho_{1}}{d\lambda} +  b\frac{d\rho_{2}}{d\lambda} +  c \cdot \frac{d\bsigma}{d\lambda}- \beta(a,b,c) \right] d\lambda : (a,b,c)\in L^{\infty}_\lambda(X,\R^{2+d})  \right\}\\
     &\quad   = \int \alpha \left( \frac{d\rho_{1}}{d\lambda},\frac{d\rho_{2}}{d\lambda},\frac{d\bsigma}{d\lambda} \right) d\lambda. 
    \end{align*}

Next, we prove part (\ref{finitedistance}). 
 First, we will show that $\bsigma \ll \tilde{\rho}.$ 
Assume, for the sake of contradiction, that  there is a set $B$ such that $|\bsigma|(B)>0$ and $\tilde{\rho} (B)=0.$  Since Lemma \ref{Klem} ensures that 
\[ \{ a+\frac{\|c\|^{2}}{8}\leq 0 \} \cap \{ b+\frac{\|c\|^{2}}{8}\leq 0 \} \subseteq K,\] define $a(x) = b(x) = -\frac{n^{2}}{8}\mathbf{1}_{B}(x)$ and $c(x) = n \mathbf{1}_{B}(x) \bsigma(B)/ |\bsigma(B)|$. Then, $\mathcal{B}_\mathcal{G}(\rho_{1},\rho_{2},\bsigma)\geq n\bsigma(B)$ for all $n \in \mathbb{N}$. Letting $n \rightarrow +\infty$ and we obtain $\mathcal{B}_\mathcal{G}(\rho_{1},\rho_{2},\bsigma)=+\infty,$ which contradicts our assumption. Hence, $\bsigma \ll \tilde{\rho},$ and there exists a Radon-Nikodym derivative $\frac{d\bsigma}{d\tilde{\rho}}$ of $\bsigma$ with respect to $\tilde{\rho}.$

Now, we will show $\bsigma \ll \theta(\frac{d\rho_{1}}{d\tilde{\rho}},\frac{d\rho_{2}}{d\tilde{\rho}}) \tilde{\rho}$.
Assume there is a set $C$ such that $|\bsigma|(C)>0$ and $\int_{C}\theta(\frac{d\rho_{1}}{d\tilde{\rho}},\frac{d\rho_{2}}{d\tilde{\rho}})d\tilde{\rho}=0.$ Thus, $\theta(\frac{d\rho_{1}}{d\tilde{\rho}},\frac{d\rho_{2}}{d\tilde{\rho}})=0 \hspace{1mm} \tilde{\rho}$ a.e.  on C.  By part (\ref{absctslambda}), for $\lambda = \tilde{\rho},$ we obtain \begin{equation*}
      \int_{C} \alpha \left( \frac{d\rho_{1}}{d\tilde{\rho}},\frac{d\rho_{2}}{d\tilde{\rho}},\frac{d\bsigma}{d\tilde{\rho}} \right) d\tilde{\rho} \leq  \mathcal{B}_\mathcal{G}(\rho_{1},\rho_{2},\bsigma) < + \infty.
\end{equation*} Since $\theta(\frac{d\rho_{1}}{d\tilde{\rho}},\frac{d\rho_{2}}{d\tilde{\rho}})=0 \hspace{1mm} \tilde{\rho}$ a.e. on C, and $\int_{C} \alpha \left( \frac{d\rho_{1}}{d\tilde{\rho}},\frac{d\rho_{2}}{d\tilde{\rho}},\frac{d\bsigma}{d\tilde{\rho}} \right) d\tilde{\rho} < + \infty,$ from the definition of function $\alpha$ in equation (\ref{definitionofpsi}), we have $\frac{d\bsigma}{d\tilde{\rho}} = 0 \hspace{1mm} \tilde{\rho}$-a.e. on C. This is a contradiction, since $|\bsigma|(C)>0.$
Hence, $\bsigma \ll \theta(\frac{d\rho_{1}}{d\tilde{\rho}},\frac{d\rho_{2}}{d\tilde{\rho}})\tilde{\rho}.$

Finally, part (\ref{integralformulation}) is an immediate consequence of part (\ref{absctslambda}) for $\lambda = \tilde{\rho}$.
      
\end{proof}

We conclude  by combining the previous compactness and lower semicontinuity properties to prove  that, whenever $W_{\Omega \times \G}(\bmu,\bnu)<+\infty$, a minimizer of the dynamic distance  exists. 

\begin{thm}[Existence of minimizers] \label{existenceofminimizers}  Suppose $\G$ is symmetric and connected,  $\theta$ satisfies Assumption \ref{interpolationassumption}, and $\Omega \subseteq \Rd$ is closed.
 Then, for any $\bmu,\bnu \in \P(\Omega \times \G)$ with $W_{\Omega \times \G}(\bmu,\bnu)<+\infty$, there exists $(\brho,\bu,\bv) \in \mathcal{C}(\bmu,\bnu)$ that achieves the infimum in the definition of the dynamic vector valued distance, equation (\ref{dynamic}).
\end{thm} 

\begin{proof}

    Choose $\{(\brho^k,\bu^k,\bv^k)\}_{k \in \mathbb{N}} \subseteq \C(\bmu,\bnu)$ satisfying \begin{equation*}
        \lim_{k \to \infty} \left( \int_0^1 \| (\bu_{t}^{k},\bv_{t}^{k}) \|_{\brho_{t}^{k}}^{2} dt \right)^{\frac{1}{2}} = W_{\Omega \times \G}(\bmu,\bnu) <+\infty.
    \end{equation*} 
    By the compactness Proposition \ref{compactnesspropertyfordirectmethod}, there exists $(\brho,\bm,\bsigma)$ satisfying a momentum formulation of the vector valued continuity equation, equation (\ref{mctyeqn}), so that, up to a subsequence,          \begin{equation*}
            (d\brho_{t}^{k}dt,\bu_{t}^{k}\odot d\brho_{t}^{k}dt,\bv_{t}^{k}\odot {\bf \Theta}_{t}^{k} d\bar{\brho}_{t}^{k}dt) \to (d\brho_tdt,\bm,\bsigma), \hspace{2mm} \mbox{as} \hspace{1mm} k \to \infty, 
        \end{equation*} where we abbreviate \begin{equation*}
            {\bf \Theta}_{t}^{k} := \left[ \theta \left( \frac{d\rho_{i,t}^{k}}{d\bar{\rho}_{t}^{k}},\frac{d\rho_{j,t}^{k}}{d\bar{\rho}_{t}^{k}} \right)q_{ij} \right]_{ij} .
        \end{equation*} 

Next, we recall the definition of the momentum formulation of the kinetic energy for the classical 2-Wasserstein metric.
 For  $K_{2} := \{ (a,b) \in \mathbb{R}^{1+d}: a+\frac{\|b\|^2}{2} \leq 0 \}$  and a closed subset $X \subseteq \R^m$, $m \in \mathbb{N}$, define 
 $  \mathcal{B}: \mathcal{M}(X) \times \mathcal{M}_s(X)^d  \to \mathbb{R} \cup \{+\infty\}$ by
 \[\mathcal{B}(\rho,\bsigma) := \sup \left\{ \int_{X} a(x)d\rho(x) + \int_{X} b(x) \cdot d\bsigma(x) : (a,b)\in C_{b}(X;K_{2}) \right\} . \]
Then, recalling  \cite[Proposition 5.18]{santambrogio2015optimal} which characterizes $\mathcal{B}$ and Proposition \ref{SantambrogioGraphProp5.18} which characterizes the graph analogue $\mathcal{B}_\G$,
       by definition of the vector valued energy,  \begin{align*}  
 \int_0^1 \| (\bu_{t}^{k},\bv_{t}^{k}) \|_{\brho_{t}^{k}}^{2} dt &= {\int_{0}^{1} \sum_{i=1}^n \int_\Omega | u_{i,t}^{k}|^2   d \rho_{i,t}^{k} dt}  + {\int_{0}^{1} \frac12 \sum_{i,j=1}^n \int_\Omega |v_{ij,t}^{k} |^2 { \Theta}_{ij,t}^{k} d\bar{\rho}_{t}^{k} dt.} \\
 &=\sum_{i=1}^{n} \mathcal{B}(d\rho_{i,t}^{k} dt,u_{i,t}^{k}d\rho_{i,t}^{k}dt) + \frac12 \sum_{i,j=1}^{n}\mathcal{B}_{\G}(d\rho_{i,t}^{k} dt,d\rho_{j,t}^{k} dt,v_{ij,t}^{k}{  \Theta}_{ij,t}^{k}d\bar{\rho}_{t}^{k} dt).
 \end{align*}
 By the lower semicontinuity of the functions $\mathcal{B}$ and $\mathcal{B}_\mathcal{G}$ with respect to narrow convergence,  \begin{align*}
    & W_{\Omega \times \G}^{2}(\bmu,\bnu) =  \liminf_{k \to \infty} \int_0^1 \| (\bu_{t}^{k},\bv_{t}^{k}) \|_{\brho_{t}^{k}}^{2} dt\\
     &\quad  = \liminf_{k \to \infty} \left( \sum_{i=1}^{n} \mathcal{B}(d\rho_{i,t}^{k} dt,u_{i,t}^{k}d\rho_{i,t}^{k}dt)+ \frac{1}{2}\sum_{i,j=1}^{n}\mathcal{B}_{\G}(d\rho_{i,t}^{k} dt,d\rho_{j,t}^{k} dt,v_{ij,t}^{k}{  \Theta}_{ij,t}^{k}d\bar{\rho}_{t}^{k} dt \right) \\
   &\quad  \geq \liminf_{k \to \infty} \left( \sum_{i=1}^{n} \mathcal{B}(d\rho_{i,t}^{k} dt,u_{i,t}^{k}d\rho_{i,t}^{k}dt) \right) +
     \liminf_{k \to \infty} \left( \frac{1}{2} \sum_{i,j=1}^{n}\mathcal{B}_{\G}(d\rho_{i,t}^{k} dt,d\rho_{j,t}^{k} dt,v_{ij,t}^{k}{  \Theta}_{ij,t}^{k}d\bar{\rho}_{t}^{k} dt) \right)  \\ 
    &\quad  \geq \sum_{i=1}^{n} \mathcal{B}(d\rho_{i,t} dt,m_{i}) + \frac{1}{2}\sum_{i,j=1}^{n}  \mathcal{B}_{\G}(d\rho_{i,t}dt,d \rho_{j,t} dt, \sigma_{ij}).
     \end{align*} 
 
   Thus, $\sum_{i=1}^{n} \mathcal{B}(d\rho_{i,t} dt,m_{i}) + \frac{1}{2} \sum_{i,j=1}^{n} \mathcal{B}_{\G}(d\rho_{i,t}dt,d\rho_{j,t}dt,\sigma_{ij})$ is   finite, and it follows by \cite[Proposition 5.18]{santambrogio2015optimal} and Proposition \ref{SantambrogioGraphProp5.18} that there exist $\bu$ and $\bv$ so that \begin{align*}
         &d m_{i} = u_{i} d \rho_{i,t}dt,  \quad d \sigma_{ij} = v_{ij} {  \Theta}_{ij} d\bar{\rho}_{t}dt, \\
         &\mathcal{B}(d\rho_{i,t} dt,m_{i}) = \int_{0}^{1}  \int_\Omega | u_{i,t}|^2   d \rho_{i,t} dt, \quad 
         \mathcal{B}_{\G}(d\rho_{i,t}dt,d\rho_{j,t}dt,\sigma_{ij}) = \int_{0}^{1} \int_\Omega |v_{ij,t} |^2 { \Theta}_{ij} d\bar{\rho}_{t} dt.
     \end{align*} 
     and $(\brho, \bu,\bv) \in \C(\bmu,\bnu)$.

     Finally, since \begin{equation*}
         \sum_{i=1}^{n} \mathcal{B}(d\rho_{i,t} dt,m_{i}) + \frac{1}{2}\sum_{i,j=1}^{n} \mathcal{B}_{\G}(d\rho_{i,t}dt,d\rho_{j,t}dt,\sigma_{ij}) = \int_0^1 \|(\bu_t,\bv_t)\|_{\brho_t}^2 dt , \end{equation*}  
     we see that $(\brho, \bu,\bv)$ attains the infimum in the definition of the dynamic distance (\ref{dynamic}).

\end{proof}

\subsection{Lifted space and static (semi) metrics} \label{subsection:Liftedspaceandstaticsemimetrics}

In the classical 2-Wasserstein setting, the dynamic formulation of the optimal transport distance has an equivalent static formulation, in terms of a Kantorovich problem. On one hand, it is known that  a classical Kantorovich formulation does not exist for the 2-Wasserstein metric $W_\G$ on graphs: as shown by Maas, for any 2-Wasserstein metric on a graph defined via a Kantorovich formulation with respect to a ground distance, the only geodesics on $\P(G)$ are constant curves, which stands in contrast to the many nonconstant geodesics that exist for $W_{\G}$  \cite[Remark 2.1]{maas2011gradient}. Since our dynamic distance reduces to $W_\G$ in the case the spatial domain is a singleton, $\Omega = \{ x_0\}$,  a classical Kantorovich formulation for (\ref{dynamic}) cannot hold. In spite of this, due to the importance of the static formulation of the classical 2-Wasserstein metric in classification tasks, we seek a static analogue that is, in some sense, equivalent. We consider two candidates: a metric defined via canonical lifts and a static semimetric. The semimetric is inspired by the static definition Hellinger-Kantorovich distance by Liero et al. \cite{liero2016optimal}. The  metric via canonical lifts generalizes the static distance introduced in the graph setting by Maas and Erbar \cite{erbar2012ricci}. Our main Theorem \ref{thm:Main} shows that   both dominate the dynamic metric \ref{dynamic}, and as a Corollary \ref{thm:EquivalenceTopologies}, we conclude that all three quantities are bi-H\"older equivalent.

In order to define both   static formulations of vector valued optimal transport, we begin by observing that our vector valued measures $\P(\Omega \times \G)$ may be viewed as projections of   probability measures   on a  \emph{lifted space}, $\P(\Rd \times \Delta^{n-1})$, which inherits the same graph geometry encoded in our dynamic metric. The projection operator $\PP$ defined in equation (\ref{projectiondef}) of the introduction may be expressed in terms of test functions $\eta \in C_b(\Rd)$ as
\begin{align*}
  \int_\Rd \eta(x) d (\PP \lambda)_j(x) &= \int_{\Rd \times \Delta^{n-1}  } p_j(r) \eta(x) d \lambda(x,r) 
  \\ & =\begin{cases} \int_{\Rd \times \Delta^{n-1}  } r_j \eta(x) d \lambda(x,r)  &  \text{ for } j=1, \dots, n-1, \\  
  \int_{\Rd \times \Delta^{n-1} } (1-\sum_{j=1}^{n-1}r_j)  \eta(x) d \lambda(x,r) &\text{ for }j=n,\end{cases}
\end{align*}
 where  $\bp: \Delta^{n-1} \to \P(\G)$, $\bp(r) = [p_j(r)]_{j=1}^n$ is the isometry from equation (\ref{Idef}).

Since  $(\Delta^{n-1},d_{\Delta^{n-1}})$ has finite diameter and a topology equivalent to the usual Euclidean topology (see Lemma \ref{lem:EquivalenceTopologies}), $\bp$ is bounded and continuous, and $\PP$ is continuous with respect to the narrow topologies on its domain and range.
The fact that $\PP$ is   surjective   follows from the existence of a \emph{canonical lift} $\lambda_\bmu \in \P(\Rd \times \Delta^{n-1})$ of $\bmu \in \P(\Rd \times \G)$  satisfying   $\PP \lambda_{\bmu} = \bmu$; see equation (\ref{eqn:Introcanonical}).

 \begin{rem}[Disintegration of lifts] \label{propertiesofprojectionofliftedmeasures} 
%

We will often use the fact that, if $\PP \lambda = \bmu$, then
\begin{align} \label{densitywrtbarmu1}
\pi_{\Rd} \# \lambda = \bar{\mu},
\end{align} 
where $\bar{\mu} = \sum_{i=1}^n \mu_i$, and $\pi_\Rd: \Rd \times \Delta^{n-1} \to \Rd$ denotes projection onto the first component. 

Likewise, let $\{\lambda_{x} \}_{x \in \Rd} \subseteq \P(\Delta^{n-1})$ be the disintegration of $\lambda$ with respect to $\pi_{\Rd} \# \lambda$; that is
\[ \int_{\Rd \times \Delta^{n-1}} \eta(x,r) d \lambda(x,r) = \int_\Rd \left( \int_{\Delta^{n-1}} \eta(x,r) d \lambda_{x}(r) \right) d (\pi_\Rd\#\lambda)(x) , \quad \forall \eta \in C_b(\Rd \times \Delta^{n-1}) \ . \]
Then $d \lambda_{x}(r)$ represents the distribution  over $\Delta^{n-1}$ at a fixed spatial location $x$ and
\begin{align} \label{densitywrtbarmu}  \int_{\Delta^{n-1} } p_j(r) d \lambda_{x}(r) = \frac{d \mu_j}{d \bar{\mu}}(x) \text{ for } j=1, \dots, n, 
\end{align}
 represents the proportion of mass of coordinate $j$ at location $x$. 
\end{rem}


By considering canonical lifts of probability measures on $\Rd \times \G$, we can now  use the 2-Wasserstein distance on $\P_2(\Rd \times \Delta^{n-1})$  to define a static metric $W_{2, \mathcal{W}}$ on $\mathcal{P}_2(\Rd \times \mathcal{G})$, as in equation \ref{staticsingletypedistance}. This metric is well-defined since  for any lifting $\lambda$ of $\mu \in \P_2(\Rd \times \G)$, we have $\lambda \in \P_2(\Rd \times \Delta^{n-1})$. To see this, note that,
by Remark \ref{propertiesofprojectionofliftedmeasures}, since $\PP \lambda = \bmu,$ we have $\pi_{\Rd} \# \lambda = \bar{\mu}.$ Thus, for any reference point $(x_{0},r_{0})\in \Rd \times \Delta^{n-1},$ we obtain \begin{align}
    M_{2}(\lambda) &= \int_{\Rd \times \Delta^{n-1}} d_{\Rd}^{2}(x,x_{0}) + W_{\G}^{2}(r,r_{0}) d\lambda(x,r) \nonumber \\
    &= \int_{\Rd} d_{\Rd}^{2}(x,x_{0}) d\bar{\mu}(x) + \int_{\Delta^{n-1}} W_{\G}^{2}(r,r_{0}) d(\pi_{\Delta^{n-1} } \#  \lambda)(r)  = M_{2}(\bar{\mu}) + C_{\Delta^{n-1}} < +\infty, \label{secondmomentupstairs}
\end{align} since  $ \bar{\mu} \in \P_{2}(\Rd)$ and $(\Delta^{n-1},d_{\Delta^{n-1}})$ has finite diameter.

We choose the notation $W_{2, \mathcal{W}}$, since  this metric reduces to a Kantorovich-type metric on $\P(\G)$ studied by Erbar and Maas when   $\supp \mu_i = \supp \nu_i = \{x_0\}\subseteq \Rd$ for all $i =1, \dots, n$. Erbar and Maas's version of this metric is defined in terms of the ground distance on the graph $\G$ given by
\[ d_{\mathcal{W}}: \G \times \G \to [0,+\infty) , \quad
 d_{\mathcal{W}}(i, j) := W_\G(\delta_i, \delta_j) . \]
 The following lemma shows the connection between our metric  $W_{2, \mathcal{W}}$ and that studied by Erbar and Mass, and  as well as  showing that   $W_{2, \mathcal{W}}$ is a special case of the static vvOT distances studied by Bacon  \cite{bacon2020multi}; see equation (\ref{BaconKantorovich}).
\begin{prop}[Kantorovich formulation of $W_{2, \mathcal{W}}$] \label{W2WKantorovich}
For any $\bmu, \bnu \in \P_2(\Rd \times \G)$,
\begin{align*} W^2_{2, \mathcal{W}}(\bmu, \bnu) &=
 \min_{\substack{\boldsymbol{\Gamma} \in \P((\Rd \times \G) \times (\Rd \times \G)) \\ \pi_1 \# \boldsymbol{\Gamma} = \bmu , \pi_2 \# \boldsymbol{\Gamma} = \bnu} } \iint_{(\Rd \times \G) \times (\Rd \times \G)} d_\Rd^2(x, y) + d^2_{\mathcal{W}}(i,j)  \ d \boldsymbol{\Gamma}((x,i),(y,j)) \\
&=   \min_{\boldsymbol{\Gamma} \in \Pi(\bmu,\bnu)}   \sum_{i,j=1}^{n} \iint_{\Rd \times \Rd} d_\Rd^2(x, y) + W_\G(\delta_i,\delta_j) d\Gamma_{ij}(x,y)  \end{align*}
\end{prop}
\begin{proof} By definition, 
\begin{align*}
W^2_{2, \mathcal{W}}(\bmu, \bnu) &= W^2_{\Rd \times \Delta^{n-1}}(\lambda_{\bmu}, \lambda_{\bnu})  \\
&= W^2_{\Rd \times \Delta^{n-1}} \left(  \sum_{j=1}^{n} \mu_j(x) \otimes \delta_{e_{j}}(r) ,  \sum_{j=1}^{n} \nu_j(x) \otimes \delta_{e_{j}}(r)   \right) \\
&  =\min_{\substack{ \Gamma \in \P(({\Rd \times \Delta^{n-1}}) \times ({\Rd \times \Delta^{n-1}})) \\  \pi_1 \# \Gamma =  \sum_{j=1}^{n} \mu_j(x) \otimes \delta_{e_{j}}(r) \\ \pi_2 \# \Gamma =  \sum_{j=1}^{n} \nu_j(x) \otimes \delta_{e_{j}}(r) }}   \iint_{({\Rd \times \Delta^{n-1}}) \times ({\Rd \times \Delta^{n-1}})}   d^2_{\Rd}(x, \tilde{x}) + d^2_{\Delta^{n-1}}(r, \tilde{r})  \ d\Gamma ((x,r),(\tilde{x},\tilde{r})) .
\end{align*}
Since the marginals of $\Gamma$ are supported on $S:=\Rd \times \{e_j\}_{j=1}^n \subseteq \Rd \times \Delta^{n-1}$, 
\begin{align*}
\Gamma((S \times S)^c) &= \Gamma((S^c \times (\Rd \times \Delta^{n-1}) ) \cup((\Rd \times \Delta^{n-1}) \times S^c))  \\
&\leq \Gamma(S^c\times (\Rd \times \Delta^{n-1})) + \Gamma((\Rd \times \Delta^{n-1}) \times S^c) = \pi_1\# \Gamma(S^c) + \pi_2 \# \Gamma(S^c) = 0 ,
\end{align*}
so $\Gamma$ is supported on $S \times S $. Furthermore, when $r, \tilde{r} \in \{ e_j\}_{j=1}^n$, 
\[ d^2_{\Delta^{n-1}}(r, \tilde{r}) =  W^2_{\G}(\bp(r), \bp(\tilde{r})) = d^2_{\mathcal{W}}(i, j) , \ \text{ for } \bp(r) = \delta_i, \   \bp(\tilde{r}) = \delta_j. \] 
The map $(\id \times \bp) : \Rd \times \{e_j\}_{j=1}^n \to \Rd \times \{  \delta_g  : g \in \G \}$ is a bijection, and the map $\mathbf{d}:\{  \delta_g  : g \in \G \} \to \G: \delta_g \mapsto g$ is a bijection. Thus, defining $\mathbf{f}:= (\id \times (\mathbf{d} \circ \bp))$, we see that there is a one to one correspondence between $\Gamma \in \P((\Rd \times \Delta^{n-1}) \times (\Rd \times \Delta^{n-1}))$ supported on $S \times S$ and $\boldsymbol{\Gamma} := (\mathbf{f} \times \mathbf{f}) \# \Gamma \in \P((\Rd \times \G) \times (\Rd \times \G))$. This gives the result.
\end{proof}

In conjunction with the static metric $W_{2,\mathcal{W}}$, we also study the static formulation $ D_{\Rd\times \mathcal{G}}$, as defined in equation (\ref{semimetricdef}). We now prove Proposition \ref{semimetricprop}, which shows that $ D_{\Rd\times \mathcal{G}}$  is, indeed, a semimetric. 

\begin{proof}[Proof of Proposition \ref{semimetricprop}]
Non-negativity and symmetry of $D_{\Rd \times \G}$ follow from the analogous properties of  $W_{\Rd \times \Delta^{n-1}}$. Finiteness follows from inequality (\ref{secondmomentupstairs}).

We conclude by showing that $D_{\Rd \times \G}(\bmu, \bnu)$ vanishes only when $\bmu = \bnu$. If $\bmu = \bnu$, it is clear that it vanishes. Now, assume $D_{\Rd \times \G}(\bmu,\bnu)=0.$ Then, there exist two sequences of measures $\lambda_{1}^{k},\lambda_{2}^{k}\in \P(\Rd \times \Delta^{n-1})$ such that $\PP \lambda_{1}^{k} = \bmu, \PP \lambda_{2}^{k} = \bnu,$ for all $k,$ and 
\begin{align} \label{ktoinftyzero}
\lim_{k\to \infty} W_{\Rd \times \Delta^{n-1}}(\lambda_{1}^{k},\lambda_{2}^{k}) = 0.
\end{align} By  equation (\ref{secondmomentupstairs}), we obtain  $
    \sup_{k} M_{2}(\lambda_{1}^{k}) < +\infty$ and $\sup_{k} M_{2}(\lambda_{2}^{k}) < +\infty.$
 Thus $\{ \lambda_{1}^{k} \}_{k\in \mathbb{N}}$ and $\{ \lambda_{2}^{k} \}_{k\in \mathbb{N}}$ are relatively   compact with respect to the narrow topology induced by $d_{\Rd \times \Delta^{n-1}}.$ Let $\lambda_1$ and $\lambda_2$ denote two limit points. By the joint lower semicontinuity of the $W_{\Rd \times \Delta^{n-1}}$ metric with respect to narrow convergence \cite[Proposition 7.1.3]{ambrosiogiglisavare}
\begin{align} \label{lscsecondpart} \lim_{k \to +\infty}W_{\Rd \times \Delta^{n-1}}(\lambda_{1}^{k},\lambda_{2}^{k}) \geq W_{\R^d \times \Delta^{n-1}}(\lambda_{1},\lambda_{2})  . 
\end{align}
Combining (\ref{ktoinftyzero}) and (\ref{lscsecondpart}), we see that $\lambda_1 = \lambda_2$. Finally, by the continuity of $\PP$ with respect to narrow convergence, we obtain $\PP \lambda_1 = \lim_{k \to +\infty} \PP \lambda_1^k = \bmu$ and $\PP \lambda_2 = \lim_{k \to +\infty} \PP \lambda_2^k = \bnu$, so $\bmu = \bnu$.
  
\end{proof}

\section{Comparison of distances}
\label{sec:Comparison}

In this section, we study the relation between the different notions of distance over the space $\mathcal{P}_2(\Omega \times \G)$  introduced in Section \ref{sec:Distances}. Our main result, Theorem \ref{thm:Main}, establishes a one-sided ordering between these metrics.

 Since the second inequality in Theorem \ref{thm:Main} follows immediately from our definition of $D_{\Rd \times \G}$, the key challenge is to prove the first inequality.  Given arbitrary $\lambda_0, \lambda_1 \in \mathcal{P}_2({\Rd \times \Delta^{n-1}})$ with $\mathfrak{P} \lambda_0 = \bmu$ and $\mathfrak{P} \lambda_1 = \bnu$, we achieve this by  showing that $W_{\Rd \times \Delta^{n-1}}(\lambda_0, \lambda_1)$  admits a dynamic formulation in terms of a least action principle along solutions of the continuity equation on $\Rd \times (\Delta^{n-1})^\circ$ connecting $\lambda_0$ and $\lambda_1$; see Definition \ref{def:ContEqN}. This reformulation will put us in a good position to compare $W_{\Rd \times \Delta^{n-1}}(\lambda_0, \lambda_1)$ with $  W_{\Omega \times \G}(\bmu,\bnu)$ because, as we will show in Section \ref{sec:UpstairsDownstairs}, solutions to the continuity equation on ${\Rd \times (\Delta^{n-1})^\circ}$ project down to  solutions to the vector valued continuity equation \eqref{vvcty}, where the latter has smaller kinetic energy than the former. 

Unfortunately, the equivalence between static and dynamic Wasserstein distances over probability measures on complete Riemannian manifolds (see, e.g.,  \cite{lisini2007characterization}) does not apply to the space of probability measures on ${\Rd \times \Delta^{n-1}}$ for two reasons. First, the Riemannian metric that is induced on ${\Rd \times (\Delta^{n-1}})^\circ$  by the graph structure  degenerates as we approach the boundary of $\Delta^{n-1}$. Second, we cannot overcome this obstacle by simply  restricting ourselves to closed subsets of  $\Rd \times (\Delta^{n-1})^\circ$, since the Riemannian metric on these subsets is not complete. In particular, it is possible to find points in $(\Delta^{n-1})^\circ$ for which the constant speed geodesic between them touches the boundary; see Remark \ref{geodesicintersectingtheboundary}.

In order to circumvent the issues described above, we prove that, for any  constant speed geodesic $\{ \lambda_t \}_{t \in [0,1]}$ connecting   measures $\lambda_0$ and $\lambda_1$, there exists an approximation $\{ \lambda_t^a \}_{t \in [0,1]}$, for $a>0$ small, that is uniformly bounded away from the boundary of $\Delta^{n-1}$ and for which the metric slope is controlled by the metric slope of $\lambda_t$, up to a small correction that vanishes as $a \to 0$. 
A key feature of this approximation is that, for each $i \in \G$, the total mass assigned by the projection $\mathfrak{P}(\lambda_t^a)$ to component $i$ is constant in the regularization parameter $a>0$, for both $t=0$ and $t=1$. This final property plays a crucial role in the proof of Theorem \ref{thm:Main};  see Remark \ref{rem:TypeTotal}.
 
The remainder of this section is organized as follows. In subsection \ref{sec:ConstructionRegularized}, we construct the approximate geodesics $\{ \lambda_t^a \}_{t \in [0,1]}$ that avoid the boundary. In section \ref{sec:UpstairsDownstairs}, we show that solutions of the continuity equation on ${\Rd \times (\Delta^{n-1})^\circ}$ project down to  solutions to the vector valued continuity equation \eqref{vvcty} with smaller kinetic energy. In section \ref{sec:ProofMainTheorem}, we combine these ingredients to prove Theorem \ref{thm:Main}. In section \ref{sec:TopoEquiv}, we prove that in the case $\bmu, \bnu \in \P(\Omega \times \G)$ for $\Omega$ compact and convex, all (semi)-metrics that we have considered $d_{BL}$, $W_{\Omega \times \G}$, $D_{\Rd \times \G}$, and $W_{2, \mathcal{W}}$ are bi-H\"older equivalent, hence induce equivalent topologies. Furthermore, under Assumption \ref{Assumption:Uniqueness}, we prove that the linearized vector valued OT metric defined in equation \eqref{def:LOT} induces the same topology as all the aforementioned metrics; see Proposition \ref{prop:LOT}.    Finally, in section  \ref{sec:Examples}, we present examples   of both equality and strict inequality in our main theorem, which show that our estimates are sharp. As a consequence, we also see that the semimetric $D_{\Rd \times \G}$ fails to satisfy the triangle inequality.

\subsection{Boundary avoidant approximation of geodesics on lifted space}
\label{sec:ConstructionRegularized}

In this section, we construct regularized versions of geodesics in  $\P({\Rd \times \Delta^{n-1}})$ that   avoid the boundary of the simplex; see equation \ref{strictinteriordef} for the definition of $\Delta^{n-1}_{c}$ for $c>0$. 
\begin{prop}
\label{prop:RegularMeasures}
Suppose $\G$ is symmetric and connected, $\theta$ satisfies Assumption \ref{interpolationassumption}, and $\Omega \subseteq \Rd$ is closed on convex. Consider $\bmu_0 , \bmu_1 \in \P(\Omega \times \G)$ satisfying 
\begin{align} \label{m0m1forcurveperturb}  m_{i,0}:= \mu_{i,0} (\Omega) >0 , \  m_{i,1}:= \mu_{i,1} (\Omega) > 0 \ , \quad \forall i = 1, \dots, n . \end{align}
 Suppose $\lambda_0, \lambda_1 \in \P({\Rd \times \Delta^{n-1}})$ satisfy $\mathfrak{P}(\lambda_0) = \bmu_0$ and $\mathfrak{P}(\lambda_1) = \bmu_1$. Then,  for every   $a \in (0,1)$, there is a solution $ (\lambda^{a}, \bw^{a})$ of the continuity equation on $\Rd \times (\Delta^{n-1})^\circ$ and a constant  $C= C_{ {m}_0,  {m}_1,\G}$  depending   on  $\min_i  m_{i,0} >0$, $ \min_i m_{i,1}>0$, and   the graph $\G$ for which the following hold:
\begin{enumerate}[(i)]
    \item for all $t \in [0,1]$,   $\supp(\lambda_{t}^{a}) \subseteq \Omega \times \Delta^{n-1}_{a C}$; \label{4p2ininterior}
    \item  $\mathfrak{P}(\lambda_0^{a})_i(\Rd ) = m_{i,0}$ and $\mathfrak{P}(\lambda_1^{a})_i(\Rd ) =  m_{i,1}$ for all $i= 1, \dots, n$;\label{balancedmass}

    \item $  \int_0^1 \int_{\Rd \times \Delta^{n-1}} \lVert \bw^{a}(x,r,t) \rVert^2_{{\rm Tan}_{(x,r)}(\Rd \times \Delta^{n-1})} d \lambda_t^{a}(x,r) dt  \leq  W_{\Rd \times \Delta^{n-1}}^2(\lambda_0, \lambda_1) + a C$;   \label{thirdlisinipart}
    \item $ \lim_{a \rightarrow 0} W_{\Rd \times \Delta^{n-1}}^2(\lambda_0, \lambda_0^{a}) =0 $ and $ \lim_{a \rightarrow 0}  W_{\Rd \times \Delta^{n-1}}^2(\lambda_1, \lambda_1^{a}) =0 .$ \label{Wconvergenceatozero}
\end{enumerate}

\end{prop}

\begin{proof} 
 Let $\Gamma$ be an optimal transport plan from $\lambda_0$ to $\lambda_1$, i.e., a minimizer of \eqref{eqn:Kantorovich} with $\lambda=\lambda_0 $ and $\tilde{\lambda} = \lambda_1$. By Lemma \ref{simplexgeodesicapproximationlem}, for any $r, \tilde{r} \in \Delta^{n-1}$, there exists a constant speed geodesic $\gamma_{r, \tilde{r}}:[0,1] \mapsto \Delta^{n-1}$ from $r$ to $\tilde{r}$, and the map $(r, \tilde{r}) \mapsto \gamma_{r, \tilde{r}}$ can be chosen to be Borel measurable. Thus,   $t \mapsto ((1-t)x + t \tilde{x}, \gamma_{r, \tilde{r}}(t))$ is a constant speed geodesic in $\Rd \times \Delta^{n-1}$ and
 the map 
 \[ g:(\Rd \times \Delta^{n-1}) \times (\Rd \times \Delta^{n-1}) \to AC([0,1]; \Rd \times \Delta^{n-1}): ((x,r), (\tilde{x},\tilde{r})) \mapsto ((1-t)x + t \tilde{x}, \gamma_{r, \tilde r }(t)) \]
  can be chosen to be measurable. As a consequence, $\lambda_t := g_t \# \Gamma$ is a constant speed geodesic from $\lambda_0$ to $\lambda_1$ with respect to $W_{\Rd \times \Delta^{n-1}}$ and  $\supp(\lambda_{t}^{a}) \subseteq \Omega \times \Delta^{n-1} $ for all $t\in [0,1]$ \cite[Theorem 9.13, Theorem 10.6]{ambrosio2021lectures}.
%
%

By Lemma \ref{simplexgeodesicapproximationlem} with $s_0 = \bp^{-1}(m_0)$, $s_1 = \bp^{-1}(m_1)$, there exists an approximation 
 \[ \gamma_{r, \tilde{r}}^a := (1-a) \gamma_{r,\tilde{r},t} + a((1-t) s_0 + t s_1)
\in AC([0,1];\Delta^{n-1}) \] of the geodesic $\gamma_{r, \tilde{r}}$ for which the map 
\[ g^{a}:(\Rd \times \Delta^{n-1}) \times (\Rd \times \Delta^{n-1}) \to AC([0,1]; \Rd \times \Delta^{n-1}): ((x,r), (\tilde{x},\tilde{r})) \mapsto ((1-t)x + t \tilde{x}, \gamma^{a}_{r, \tilde r }(t)) , \]
is likewise measurable with range contained in $\Rd \times \Delta_{aC}^{n-1}$ for $C =C_{m_0, m_1, \G}$ as above. 
Using this, we in turn define the measures 
\begin{equation}
\lambda_{t}^{a} :=g^{a}_{t} \# \Gamma, \quad t \in [0,1].
\end{equation}
which satisfy $\supp(\lambda_t^{a}) \subseteq \Omega \times \Delta_{aC}^{n-1}$, so part (\ref{4p2ininterior}) holds. 

Next, observe that, for $i =1, \dots, n-1$, the definition of $\gamma^a_{r_0,r_1}$ in equation (\ref{gammaa}) ensures
\begin{align*}
\mathfrak{P}(\lambda_t^{a})_i(\Rd) &=  \int_{\Rd \times \Delta^{n-1}} r_i d\lambda_t^{a}(x,r)\\
 & =\int_{\Rd \times \Delta^{n-1}} \int_{\Rd \times \Delta^{n-1}} \pi_{i}(\gamma_{r, \tilde{r}}^{a}(t)) d\Gamma((x,r),(\tilde x , \tilde r)) 
\\& =(1-a) \int_{\Rd \times \Delta^{n-1}} \int_{\Rd \times \Delta^{n-1}}  \pi_{i}(\gamma_{r,\tilde{r},t})d\Gamma((x,r),(\tilde x , \tilde r))   + a ((1-t)m_{i,0}+tm_{i,1}) . 
\end{align*}
Noting that, for  $t =0$, $ \pi_i(\gamma_{r,\tilde{r},t}) = r_i$ and for $t = 1$, $ \pi_i(\gamma_{r,\tilde{r},t}) = \tilde{r}_i$, we obtain 
\[ \mathfrak{P}(\lambda_0^{a})_i(\Rd)  = (1-a)\mathfrak{P}(\lambda_0)_i(\Rd)  +a m_{i,0}  =   (1-a)m_{i,0}  +a m_{i,0} = m_{i,0}  . \]
 and, likewise, $ \mathfrak{P}(\lambda_1^{a})_i(\Rd)  = m_{i,1}$. Thus, part (\ref{balancedmass}) holds.



%

We now provide an Eulerian description of the evolution $t   \mapsto \lambda_t^{a}$ by proving the existence of a velocity field $\bw^a$ for which $(\lambda^a, \bw^a)$ solves the continuity equation on $\Rd \times (\Delta^{n-1})^\circ$, as in Definition \ref{def:ContEqN}. For this, we use a result due to Lisini \cite{lisini2007characterization} characterizing absolutely continuous curves in the Wasserstein space of probability measures over a suitable \textit{complete} Riemannian manifold. However, we first need to adapt our setting slightly because the manifold is not necessarily complete.
We use a simple partition of unity to construct a smooth extension  of the Riemannian metric $\langle \cdot , \cdot  \rangle_{{\rm Tan}_{(x,r)}(\Rd \times \Delta^{n-1})}$ defined over the set $\Rd \times \Delta^{n-1}_{aC} $ to all of $\R^d \times \R^{n-1}$.

Let $\zeta_a: \R^{n-1} \rightarrow [0,1] $ be a smooth cutoff function that equals one in the set $\Delta_{a C}^{n-1} $ and   zero outside  $(\Delta^{n-1})^\circ$. For each $(x,r) \in \R^d \times \R^{n-1}$,   define the Riemannian metric on the tangent space ${\rm Tan}_{(x,r)} (\R^d \times \R^{n-1}) \cong \R^d \times \R^{n-1}$ by
\[ \langle \cdot , \cdot \rangle_{\Tan^a_{(x,r)}(\R^d \times \R^{n-1})} = \begin{cases} \zeta_a(r) \langle \cdot, \cdot \rangle_{{\rm Tan}_{(x,r)}(\Rd \times \Delta^{n-1})} + (1-\zeta_a(r)) \langle \cdot, \cdot \rangle_{\Rd\times\R^{n-1}}, &\text{ if } r \in (\Delta^{n-1})^\circ , \\ \langle \cdot, \cdot \rangle_{\Rd\times\R^{n-1}} &\text{ if } r \in \R^{n-1} \setminus((\Delta^{n-1})^\circ). \end{cases} \]
 Endowed with this metric, it is straightforward to see, by the Hopf-Rinow theorem and the equivalence of the induced topology with the usual Euclidean topology, that $\R^{d} \times \R^{n-1}$ is a complete smooth Riemmanian manifold; see, e.g., Proposition \ref{basicfactWG}(\ref{lem:ComparisonMetricsDiscrete}).

Denote by $d_a$ the geodesic distance on $\R^d \times \R^{n-1}$ induced by the above metric and let $W_a$ be the corresponding $2$-Wasserstein distance on $\P(\R^d \times \R^{n-1})$.  By \cite[equation (2.4), Corollary 3.8, Proposition 3.18]{burtscher2015length} and the fact that $g^a((x,r),(\tilde{x},\tilde{r})) \in AC([0,1]; \Rd \times \Delta^{n-1})$, for almost every $t \in [0,1]$,
\begin{align*} |{{g}^{a}}((x,r), (\tilde x , \tilde r))'|_{(\Rd \times \Delta^{n-1})}(t) &= \left\| \frac{d}{dt} {g}^{a}_t((x,r), (\tilde x , \tilde r)) \right\|_{{\rm Tan}_{ {{g}^{a}}((x,r), (\tilde x , \tilde r)}( \Rd \times \Delta^{n-1})} \\
&\geq \left\| \frac{d}{dt} {g}^{a}_t((x,r), (\tilde x , \tilde r)) \right\|_{{\rm Tan}^a_{ {{g}^{a}}((x,r), (\tilde x , \tilde r)} (\Rd \times \R^{n-1})} \\
&=  |{{g}^{a}}((x,r), (\tilde x , \tilde r))'|_{d_a}(t) . \end{align*}

%
Thus, 
\begin{align*}
d_a^2\left(g_s^{a}((x,r), (\tilde x , \tilde r)), g_t^{a}((x,r), (\tilde x , \tilde r)) \right) &\leq \left( \int_s^t |{{g}^{a}}((x,r), (\tilde x , \tilde r))'|_{d_a}(\tau) d \tau \right)^2 \\
&\leq |t-s| \int_s^t |{{g}^{a}}((x,r), (\tilde x , \tilde r))'|^2_{(\Rd \times \Delta^{n-1})}(\tau)  d \tau \\
&= |t-s| \int_s^t |x- \tilde{x}|^2 + |({\gamma}^a_{r, \tilde{r}})' |^2_{\Delta^{n-1}}(\tau) d \tau \\
&\leq |t-s|^2\left( |x-\tilde{x}|^2 +  (1-a) d_{\Delta^{n-1}}^2(r ,  \tilde{r})    + aC \right) ,
\end{align*}
where 
 the last inequality follows from inequality (\ref{eq:AuxRegularized1}). Therefore,
\begin{align*}
W_{a}^2(\lambda_s^a, \lambda_t^a ) & \leq \int_{\R^d \times \R^{n-1}}  \int_{\R^d \times \R^{n-1}}  d_a^2((x,r), (\tilde x , \tilde r)) d(g_s^{a},g_t^{a})\#\Gamma)((x,r),(\tilde x , \tilde r))
\\& =\int_{\R^d \times \R^{n-1}} \int_{\R^d \times \R^{n-1}} d_a^2\left(g_s^{a}((x,r), (\tilde x , \tilde r)), g_t^{a}((x,r), (\tilde x , \tilde r)) \right) d\Gamma((x,r), (\tilde x , \tilde r)) 
\\& = |t-s|^2 \int_{\R^d \times \R^{n-1}} \int_{\R^d \times \R^{n-1}}  |x-\tilde{x}|^2 +  (1-a) d_{\Delta^{n-1}}^2(r ,  \tilde{r})    + aC   \ d\Gamma((x,r), (\tilde x , \tilde r))
\\& \leq |t-s|^2(W_{\Rd \times \Delta^{n-1}}^2(\lambda_0, \lambda_1)+ aC ).
\end{align*}

  The above inequality implies that $t   \mapsto \lambda_t^a$ is  Lipschitz continuous  with respect to the Wasserstein metric $W_a$. Thus, we may apply   \cite[Theorem 7 and Remark 8]{lisini2007characterization} to conclude that there exists a time dependent vector field $\{ \bw_t^a\}_{t \in [0,1]}$ for which $ (\lambda^a,  \bw^a) $ solves the continuity equation in $\Rd \times \R^{n-1}$, with respect to the Riemannian structure defined above,  and  
\begin{align*}
   \int_0^1 \int_{\Rd \times \R^{n-1}} \lVert \bw^{a}(x,r,t) \rVert^2_{{\rm Tan}^a_{(x,r)}(\Rd \times \R^{n-1})} d \lambda_t^{a}(x,r) dt  & \leq  W_{\Rd \times \Delta^{n-1}}^2(\lambda_0, \lambda_1)+ aC.
\end{align*}
Since $\lambda_{t}^a$ is supported in $\Rd \times \Delta_{aC}^{n-1}$ for every $t \in [0,1]$, and since the metrics $\langle \cdot , \cdot \rangle_{\Tan^a_{(x,r)}(\Rd \times \R^{n-1})}$  and $\langle \cdot , \cdot \rangle_{\Tan_{(x,r)}(\Rd \times \Delta^{n-1})}$ coincide for $(x,r) \in \Rd \times \Delta_{aC}^{n-1}$, we may interpret the continuity equation as in Definition \ref{def:ContEqN} and in the above inequality the norm  can be simply taken to be $ \lVert\cdot  \rVert^2_{{\rm Tan}_{(x,r)}(\Rd \times \Delta^{n-1})}$. This shows part (\ref{thirdlisinipart}).



To prove part (\ref{Wconvergenceatozero}), note that
\begin{align*}
W_{\Rd \times \Delta^{n-1}}^2(\lambda_0, \lambda_0^{a}) & \leq \int_{\Rd \times \Delta^{n-1}} \int_{\Rd \times \Delta^{n-1}}d_{\Rd \times \Delta^{n-1}}^2((x,r),g^{a}_0((x,r),(\tilde x , \tilde r)) )d\Gamma((x,r),(\tilde x, \tilde r)) 
\\ & = \int_{\Rd \times \Delta^{n-1}} \int_{\Rd \times \Delta^{n-1}}d_{\Delta^{n-1}}^2(r,\gamma^a_{r, \tilde{r},0} )d\Gamma((x,r),(\tilde x, \tilde r)).
\end{align*}
By Proposition \ref{basicfactWG}, we know that $(\Delta^{n-1}, d_{\Delta^{n-1}})$ has finite diameter. Thus \eqref{eq:AwayBoundary} and the dominated convergence theorem ensures $\lambda_0^{a} \xrightarrow{a \to 0} \lambda_0$. The convergence of $\lambda_1^{a} \xrightarrow{a \to 0} \lambda_1$ is deduced analogously.

\end{proof}

%





\subsection{From continuity equation on lifted space to vector valued continuity equation}
\label{sec:UpstairsDownstairs}
A key ingredient in the proof of our main Theorem \ref{thm:Main} is the following proposition, which states that, given any solution of the continuity equation on $\R^d \times (\Delta^{n-1})^\circ$, the projection operator $\PP$ induces a solution to the vector valued continuity equation with smaller kinetic energy. See section \ref{Deltametricsection} for the definitions of the matrices $B(p)$ and $\Xi$.

\begin{prop}  \label{coveringspaceinduced}
Let $(\lambda, {\mathbf w})$ solve the continuity equation on $\Rd \times (\Delta^{n-1})^\circ$, for the time interval $[0,T]$, in the sense of Definition \ref{def:ContEqN}, and suppose that $\supp \lambda_t \subseteq \Omega \times \Delta^{n-1}$ for all $t \in [0,T]$. (For simplicity of notation, we suppress time dependance of all quantities throughout this proposition.)

Then there exists a measurable function $\psi: \Rd \times (\Delta^{n-1})^\circ \to \R^{n-1}$ so that
\begin{align} \label{psiDeltaequation} B(\mathbf{p}(r))\Xi \psi(x,r) = \Xi w_{2}(x,r) \text{ for } \lambda\text{-a.e } (x,r)\in \Omega \times \Delta^{n-1}. \end{align}
Likewise, there exist  measurable functions $\bu, \bv$, defined for $\pi_\Rd \# \lambda$-a.e. $x \in \Rd$ by
\begin{align*}
\bu(x) &= [u_j(x)]_{j=1}^n , \quad u_j(x) := \frac{\int_{\Delta^{n-1}} p_j(r) w_{1}(x,r) d\lambda_{x}(r)}{\int_{\Delta^{n-1}} p_{j}(r) d\lambda_{x}(r)} \\
 \bv(x) &= [v_{ij}(x)]_{i,j = 1}^n, \quad v_{ij}(x) = \frac{\int_{\Delta^{n-1}} \left(\nabla_{\G} \Xi \psi (x,r)\right)_{ij} \theta(p_{i}(r),p_{j}(r)) d\lambda_{x}(r)}{\theta \left( \int_{\Delta^{n-1}} p_{i}(r) d\lambda_{x}(r),\int_{\Delta^{n-1}} p_{j}(r) d\lambda_{x}(r) \right)}
 \end{align*}
 so that, defining $\brho = \PP \lambda \in C(0,T;\P(\Omega \times \G))$, we have that $(\brho,\bu,\bv)$ is a solution of the vector valued continuity equation on the time interval $[0,T]$, in the sense of Definition \ref{def:ContEqMultiSpecies}. Furthermore,
 \begin{equation} \label{inequalityupstairsdownstairs}
    \int_{\Rd \times \Delta^{n-1}} \| \bw(x,r) \|_{{\rm Tan}_{(x,r)} (\Rd \times \Delta^{n-1})}^2 \  d \lambda(x,r) \geq \| (\bu ,\bv )\|^2_{\brho}.
\end{equation}
  \end{prop}

\begin{proof}
First, since Definition \ref{def:ContEqN}(\ref{continuityNinterior}) ensures  $\supp (\lambda) \subseteq \Rd \times (\Delta^{n-1})^\circ$, we have $\bp(r) \in (\P(\G))^\circ$ for $\lambda$-a.e. $(x,r) \in \Omega \times \Delta^{n-1}$. Thus, as shown in section \ref{Deltametricsection}, there exists $\psi$ measurable satisfying equation (\ref{psiDeltaequation}).

In a similar way,  $\supp \lambda \subseteq \Rd \times (\Delta^{n-1})^\circ$ ensures $ \int_{\Delta^{n-1}} p_j(r) d \lambda_x(r) >0$ for $\pi_{\Rd}\# \lambda$-a.e. $x \in \Rd$ and all $j = 1, \dots, n$, so the denominators of $u_i$ and $v_{ij}$ are strictly positive $\pi_{\Rd}\# \lambda$-almost everywhere. The bound on $w_1$ from  Definition \ref{def:ContEqN} (\ref{continuityNgrowth}) ensures that the numerator of $u_j(x)$ is well-defined for $\pi_{\Rd}\# \lambda$-a.e. $x \in \Rd$. Likewise, the bound on $w_2$ from Definition \ref{def:ContEqN} (\ref{continuityNgrowth}) ensures, for $\pi_{\Rd}\# \lambda$-a.e. $x \in \Rd$,
\begin{align*} +\infty &> \int_{\Delta^{n-1}} \|w_2(x,r) \|_{{\rm Tan}_{r}(\Delta^{n-1})} d \lambda_x(r) = \int_{\Delta^{n-1}} \langle w_2(x,r), \Xi^t B(\bp(r))^\dagger \Xi w_2(x,r) \rangle  d \lambda_x(r)   \\
&= \int_{\Delta^{n-1}} \langle w_2(x,r), \Xi^t \Xi \psi(x,r) \rangle  d \lambda_x(r)  = \int_{\Delta^{n-1}} \langle B(\bp(r)) \Xi \psi(x,r),   \Xi \psi(x,r) \rangle  d \lambda_x(r) \\
&
= \int_{\Delta^{n-1}} \langle  \nabla_\G \Xi \psi(x,r),  \nabla_\G \Xi \psi(x,r) \rangle_{{\rm Tan}_{\bp(r)}(\P(\G))}  d \lambda_x(r) ,
\end{align*}
where in the last line we use the definition of $B$ in equation (\ref{Bpdef}), integration by parts on the graph, and the definition of the Riemannian metric on the graph in equation (\ref{innerproductgraphdef}). This ensures the numerator of $v_{ij}(x)$ is well-defined for $\pi_{\Rd} \# \lambda$-a.e. $x \in \Rd$.

Next, we show that $(\brho, \bu, \bv)$ is a solution of the vector valued continuity equation. Since $t \mapsto \lambda_t$ is narrowly continuous in $\P(\Rd \times \Delta^{n-1})$,   $t \mapsto \PP \lambda_t = \brho_t$ is narrowly continuous. Likewise, by equations (\ref{densitywrtbarmu1}-\ref{densitywrtbarmu}) and Definition \ref{def:ContEqN} (\ref{continuityNgrowth})
\begin{align*}
\int_0^T \int_\Rd |u_{i,t}(x) | d \rho_{i,t}(x) dt  & = \int_0^T \int_\Rd |u_{i,t}(x) | \frac{ d \rho_{i,t}}{d \bar{\rho}_t} (x) d \bar{\rho}_t(x) dt 
\\ &\leq \int_0^T \int_{\Rd \times \Delta^{n-1}} p_j(r) |w_{1}(x,r,t)| d \lambda_t(x,r) dt<+\infty,
\end{align*}
and
\begin{align*}
\int_0^T \int_\Rd |v_{ij,t}(x)| & \theta \left( \frac{ d \rho_{i,t}}{d \bar{\rho}_t }(x),  \frac{ d \rho_{j,t}}{d \bar{\rho}_t}(x) \right)  d \bar{\rho}_t(x) dt 
\\ &\leq \int_0^T \int_{\Rd \times \Delta^{n-1}} \left| \left(\nabla_{\G} \Xi \psi (x,r)\right)_{ij} \right| \theta(p_{i}(r),p_{j}(r)) d \lambda_t(x,r) dt <+\infty.
\end{align*}
Given $\eta \in C^\infty_c(\Rd)$, via the definition of the continuity equation on the lifted space, equation    (\ref{dualCTYN}),
\begin{align*}
\frac{d}{dt}  \int_\Rd \eta(x) d \rho_{j}(x) 
&  = \frac{d}{dt}    \int_{\Rd \times \Delta^{n-1}} p_j(r) \eta(x) d \lambda(x,r)    \\
&\quad =   \int_{\Rd \times \Delta^{n-1}} \la \nabla_{\Rd \times \Delta^{n-1}}( p_j(r) \eta (x)) , \bw(x,r) \ra_{{\rm Tan}_{(x,r)}(\Rd \times \Delta^{n-1})} d \lambda(x,r)    \\
&\quad = \underbrace{ \int_{\Rd \times \Delta^{n-1}}   p_j(r) \nabla_x \eta (x) \cdot  w_1(x,r)   d \lambda(x,r)}_{\text{(I)}} \\
&\qquad \qquad + \underbrace{ \int_{\Rd \times \Delta^{n-1}} \eta (x) \la  \nabla_{ \Delta^{n-1}} p_j(r),w_2(x,r) \ra_{{\rm Tan}_r \Delta^{n-1}}   d \lambda(x,r)  }_{\text{(II)}}.
\end{align*}
We begin by considering the first term. Using the above definition of $\bu$,  
 \begin{align*}
\text{(I)} &=   \int_{\Rd } \nabla_x \eta(x) \cdot \left(  \int_{\Delta^{n-1}}  p_j(r)    w_1(x,r)  d \lambda_x(r) \right) d \pi_{\Rd} \#\lambda(x)\\
  &=   \int_{\Rd }\nabla_x \eta(x) \cdot u_{j}(x)  \left(   \int_{\Delta^{n-1}} p_j(r)d\lambda_{x}(r) \right) d \bar{\rho}(x) 
    =  \int_{\Rd }  \nabla_x \eta(x) \cdot     u_j(x)   d \rho_j(x)  .
\end{align*}
As for the second term, let $f_j \in \R^{n}$ denote the standard basis vector, so that $ \nabla_{\R^{n-1}} p_j(r) =  \Xi^t f_j  $. By equations (\ref{metricDeltatoEuclidean}-\ref{gradientdeltaformula}),   the definition of $B$ in equation (\ref{Bpdef}), integration by parts on the graph, and the definition of the Riemannian metric on the graph in equation (\ref{innerproductgraphdef}),
\begin{align*}
 \la  \nabla_{ \Delta^{n-1}} p_j(r),w_2(x,r) \ra_{{\rm Tan}_r \Delta^{n-1}}  &=  \la   w_{2}(x,r),\nabla_{\R^{n-1}} p_j(r) \ra_{\R^{n-1}}  = \la   \Xi^{-1} \Xi w_{2}(x,r),  \Xi^t f_j \ra_{\R^{n-1}} \\
 &=  \la B(\mathbf{p}(r)) \Xi \psi(x,r),  f_j \ra_\Rn  = \la\nabla_\G \Xi \psi(x,r) , \nabla_\G f_j \ra_{{\rm Tan}_{\bp(r)} \P(\G)} \\
 &= \frac12 \sum_{i =1}^n \left( (\nabla_\G \Xi \psi(x,r))_{i,j} - (\nabla_\G \Xi \psi(x,r))_{j, i} \right) \theta(p_i(r),p_j(r)) q_{ij} .
\end{align*}
Thus,
\begin{align*}
\text{(II)} 
&= \frac12 \int_{\Rd  \times \Delta^{n-1}} \sum_{i =1}^n \left( (\nabla_\G \Xi \psi(x,r))_{i,j} - (\nabla_\G \Xi \psi(x,r))_{j, i} \right) \theta(p_i(r),p_j(r)) q_{ij}    \eta(x)  d \lambda(x,r)\\ 
&=   \int_{\Rd }   \int_{ \Delta^{n-1}} \sum_{i =1}^n   (\nabla_\G \Xi \psi(x,r))_{i,j}   \theta(p_i(r),p_j(r)) q_{ij}  d \lambda_x(r)  \eta(x)   d (\pi_\Rd \#\lambda)(x)\\ 
&=   \int_{\Rd }    \sum_{i =1}^n  v_{ij}(x) \theta \left( \frac{d \rho_i}{d\bar{\rho}}(x), \frac{d \rho_j}{d\bar{\rho}}(x) \right)q_{ij}  \eta(x)  d \bar{\rho}(x).
\end{align*}
It follows that $(\brho, \bu, \bv)$ is a distributional solution of the vector valued continuity equation.
 
It remains to prove inequality (\ref{inequalityupstairsdownstairs}).   
For $i =1, \dots, n$, H\"older's inequality ensures
 \begin{align*}
     \left| \int_{\Delta^{n-1}} p_i(r) w_{1}(x,r) d\lambda_{x}(r) \right|^{2} \leq \int_{\Delta^{n-1}} p_i(r) d\lambda_{x}(r) \int_{\Delta^{n-1}} p_i(r) |w_1(x,r)|^2 d\lambda_{x}(r). 
\end{align*}
Thus,
 \begin{equation*}
     |u_i(x)|^{2} \frac{d\rho_{i}}{d\bar{\rho}}(x) \leq   \int_{\Delta^{n-1}} p_i(r) |w_1(x,r)|^2 d\lambda_{x}(r). 
\end{equation*} 
Summing over $i$ and integrating with respect to $d \bar{\rho}(x)$, we conclude that 
\begin{align} \label{spatialpartkebound}
\sum_{i=1}^n \int_\Omega | u_i|^2   d \rho_i \leq \int_{\Rd \times \Delta^{n-1}} |w_1(x,r)|^2 d \lambda(x,r) . 
\end{align}
 On the other hand,  by Jensen's inequality applied to the convex function $\alpha$ defined in Lemma \ref{Klem},
\begin{align*}
    \frac{1}{2} &\int_{\Rd} \sum_{i,j=1}^{n} |v_{ij} |^{2}\theta \left( \frac{d\rho_{i}}{d\bar{\rho}},\frac{d\rho_{j}}{d\bar{\rho}}\right) q_{ij}d\bar{\rho} \\
        &=\frac{1}{2} \int_{\Rd} \sum_{i,j=1}^{n} \frac{  \left| \int_{\Delta^{n-1}}\left( \nabla_{\G} \Xi \psi(x,r)\right)_{ij} \theta(p_i(r),p_j(r)) d\lambda_{x}(r)\right|^{2}}{\theta \left(  \int_{\Delta^{n-1}} p_i(r) d \lambda_x(r) ,\int_{\Delta^{n-1}} p_i(r) d \lambda_x(r) \right)} q_{ij}d\bar{\rho}(x)\\
    &\leq \frac{1}{2} \int_{\Rd} \sum_{i,j=1}^{n} \int_{\Delta^{n-1}} \frac{ \left| \left(\nabla_{\G} \Xi \psi(x,r)\right)_{ij}\theta(p_i(r),p_j(r)) \right|^{2}}{\theta(p_i(r),p_j(r))} q_{ij} d\lambda_{x}(r) d\bar{\rho}(x)\end{align*}
    This is then equal to
    \begin{align*}
        &  \frac{1}{2} \int_{\Rd \times \Delta^{n-1}} \sum_{i,j=1}^{n}     \left| \left(\nabla_{\G} \Xi \psi(x,r)\right)_{ij}\right|^{2}\theta(p_i(r),p_j(r))   q_{ij} d\lambda(x,r)\\
        &\quad = \frac{1}{2} \int_{\Rd \times \Delta^{n-1}}  \la \nabla_{\G} \Xi \psi(x,r) , \nabla_{\G} \Xi \psi(x,r) \ra_{{\rm Tan}_{\bp(r)} \P(\G)}  d\lambda(x,r)\\
                &\quad = \frac{1}{2} \int_{\Rd \times \Delta^{n-1}}  \la B(\bp(r)) \Xi \psi(x,r) ,  \Xi \psi(x,r) \ra_{\Rn}   d\lambda(x,r)\\
                                &\quad = \frac{1}{2} \int_{\Rd \times \Delta^{n-1}}  \la  \Xi w_2(x,r) ,  B(\bp(r))^\dagger \Xi w_2(x,r)\ra_{\R^{n}}   d\lambda(x,r)\\
    &\quad= \int_{\Rd \times \Delta^{n-1}} \| w_{2}(x,r) \|_{{\rm Tan}_{r}\Delta^{n-1}}^{2} d\lambda(x,r).
\end{align*}    Combining the above with inequality (\ref{spatialpartkebound}) gives the result.

\end{proof}

%

\subsection{Proof of main theorem}
\label{sec:ProofMainTheorem}

We now turn to the proof of our main result, Theorem \ref{thm:Main}, which relates the three different metrics over vector valued measures $\bmu, \bnu \in \P_2(\Omega \times \G)$ induced by the graph geometry on $\G$.

\begin{proof}[Proof of Theorem \ref{thm:Main}]

The inequality $D_{\Omega \times \G}(\bmu,\bnu)  \leq   W_{2, \mathcal{W}}(\bmu,\bnu)$ is immediate from the fact that the canonical lifts of the vector valued measures $\bmu$ and $\bnu$ are particular choices of measures in $\P(\Rd \times \Delta^{n-1})$ that project down to  $\bmu$ and $\bnu$. To prove the other inequality, we consider two cases.

\textbf{Case 1:} First assume that $\bmu$ and $\bnu$  satisfy $\mu_{i}(\Omega), \nu_{i}(\Omega) >0$ for all $i =1 , \dots, n$. Let $\lambda_0,$ $\lambda_1 \in \P(\Rd \times \Delta^{n-1})$ satisfy $\mathfrak{P}(\lambda_0) =\bmu $ and $\mathfrak{P}(\lambda_1) =\bnu $. For $a\in (0,1)$ sufficiently small, consider the solution $(\lambda^a,\bw^a)$ of the continuity equation on $\Rd \times (\Delta^{n-1})^\circ$ from Proposition \ref{prop:RegularMeasures}. For $\brho^{a}= \mathfrak{P}(\lambda^{a}) \in C([0,T];\P(\Omega \times \G))$, let $(\bu^{a}, \bv^{a})$ be the spatial and graph velocities from Proposition \ref{coveringspaceinduced}. It then follows   that, for all $a\in (0,1)$, if  $\bmu^a = \mathfrak{P}(\lambda_0^a) = \brho^a_0$ and $\bnu^a = \mathfrak{P}(\lambda_1^a) = \brho^a_1$
\begin{align}
\begin{split}
    W_{\Omega \times \G}^2 (\bmu^{a} , \bnu^{a}) \leq \int_0^1 \lVert (\bu_t^a, \bv_t^a) \rVert_{\brho_t^a}^2 dt &\leq  \int_0^1 \int_{\Rd \times \Delta^{n-1}} \lVert \bw_t^{a} \rVert_{{\rm Tan} (\Rd \times \Delta^{n-1})}^2 d \lambda_t^{a} dt \\
    &  \leq  W_{\Rd \times \Delta^{n-1}}^2(\lambda_0, \lambda_1) + a C,
    \label{eq:AuxThm41}
    \end{split}
\end{align}
where and $C= C_{ {m}_0,  {m}_1,\G}$  only depends on $\min_i  \mu_{i}(\Rd) >0$, $ \min_i \nu_{i}(\Rd)>0$ and on the graph $\G$.

From  Proposition \ref{prop:RegularMeasures}(\ref{balancedmass}), considering $\bmu, \bnu \in \P(\Omega \times \G) \subseteq \P(\Rd \times \G)$, we know that
\begin{equation}  
\mu_i(\Rd ) =  \mu ^{a}_{i}(\Rd ) = m_{i,0}>0 \text{ and } \nu_i(\Rd) = \nu^a_i(\Rd) = m_{i,1}>0\text{ for all } i = 1 , \dots, n.
\label{eq:AuxSameMassType}
\end{equation}
Hence, for each $i$, the Benamou-Brenier formula for the classical $W_2$ metric on $\P_2(\Rd)$ ensures that there exist weak solutions of the continuity equation, flowing from $\mu_i/m_{i,0}$ to $\mu_i^a/m_{i,0}$ and from $\nu_i/m_{i,1}$ to $\nu_i^a/m_{i,1}$, with kinetic energy equal to the square of the classical Wasserstein distance between the probability measures \cite[Theorem 4.1.3]{figalli2023invitation}. These weak solutions of the continuity equation for each coordinate $i = 1, \dots, n$ induce  weak solutions of the vector valued continuity equation with no mutation between coordinates, i.e., $\bv \equiv 0$, for which we may control the action as follows:
\begin{align} \label{WOmegaRHS} W^2_{\Omega \times \G} (\bmu, \bmu^{a} ) \leq \sum_{i =1}^n m_{i,0} W_\Rd^2\left(\frac{\mu_i}{m_{i,0}},\frac{ \mu_i^{a}}{m_{i,0}} \right) \text{ and }
 W^2_{\Omega \times \G} (\bnu, \bnu^a ) \leq \sum_{i=1}^n m_{i,1}W_\Rd^2 \left( \frac{\nu_i }{m_{i,1}}, \frac{\nu_i^a}{m_{i,1}} \right).
 \end{align}

By Proposition \ref{prop:RegularMeasures}(\ref{Wconvergenceatozero}),    $\lambda_0^{a} \xrightarrow{a \to 0} \lambda_0$ in $W_{\Rd \times \Delta^{n-1}}$, hence narrowly and in second moments \cite[Remark 7.1.11]{ambrosiogiglisavare}. In turn, from the definition of the projection map $\mathfrak{P}$ and the fact that $\Delta^{n-1}$ has finite diameter, it is immediate that $\mu_{i}^{a}/m_{i,0}$ converges narrowly and in second moments toward $\mu_i/m_{i,0}$ and, similarly, $\nu_{i}^a/m_{i,1}$ to $\nu_i/m_{i,1}$.  Inequality (\ref{WOmegaRHS}) then ensures that
\[ \lim_{a \rightarrow 0} W_{\Omega \times \G}(\bmu,\bmu^{a}) =0 , \quad  \lim_{a \rightarrow 0} W_{\Omega \times \G}(\bnu,\bnu^{a}) =0. \]
From this, the triangle inequality, and \eqref{eq:AuxThm41} it follows that 
\[ W_{\Omega \times \G}(\bmu, \bnu) \leq W_{\Rd \times \Delta^{n-1}}(\lambda_0, \lambda_1). \]
Since $\lambda_0$ and $\lambda_1$ were arbitrary measures satsifying $\PP \lambda_0  = \bmu$,  $\PP \lambda_1 = \bnu$, we obtain $W_{\Omega \times \G}(\bmu,\bnu) \leq D_{\Rd \times \G}(\bmu,\bnu)$.

\ 

\textbf{Case 2:} Consider $\bmu, \bnu \in \P_2(\Omega \times \G)$  and let $\lambda_0, \lambda_1 \in \mathcal{P}_2(\Rd \times \Delta^{n-1})$ be such that $\mathfrak{P}(\lambda_0) = \bmu$ and $\mathfrak{P}(\lambda_1) = \bnu$, so in particular $\lambda_0$ and $\lambda_1$ are supported on $\Omega \times \Delta^{n-1}$. We introduce suitable approximations $\bmu^a, \bnu^a$ of $\bmu, \bnu$ with the property that $\bmu^a$ and $\bnu^a$ assign strictly positive mass to all coordinates $i =1, \dots, n$. We can then apply Case 1 to deduce the result.

For $a \in (0,1)$, consider the map $S^{a}: (\Rd \times \Delta^{n-1} )\to (\Rd \times \Delta^{n-1} )$   given by
\[  S^{a}(x, r) :=  ( x, r^a), \text{ where } r^a_i = (1-a)r_i + \frac{a}{n} \text{ for all } i = 1, \dots, n-1 .\]
 Define the measures
\[   \lambda^a_0:= S^a \#  \lambda_0, \quad  \lambda^a_1:= S^a \#\lambda_1,\]
which are likewise supported on $\Omega \times \Delta^{n-1}$ and define
\[ \bmu^a := \mathfrak{P}(\lambda_0^a) , \quad \bnu^a := \mathfrak{P}(\lambda_1^a) \in \P(\Omega \times \G) . \]

Let $\Gamma$ be a transport plan between $\lambda_0$ and $\lambda_1$ achieving the minimum in (\ref{eqn:Kantorovich}) and let 
\[ \Gamma^{a} := (S^a \times S^a) \# \Gamma, \]
so  $\pi^1 \# \Gamma^{a} = \lambda_0^{a}$ and $\pi^2 \# \Gamma^a = \lambda_1^{a}$. 
By construction, 
\begin{equation}
    \label{eq:SameMass}
\bar{\mu} = \pi_\Rd \# \lambda_0 = \pi_\Rd \# \lambda_0^a = \bar{\mu}^a  
\text{ and } \bar{\nu} = \pi_\Rd \# \lambda_1 = \pi_\Rd \# \lambda_1^a = \bar{\nu}^a,  \end{equation}
and, for any $\eta \in C_b(\Rd)$, 
\begin{align*}
 \int_\Rd \eta(x) d \mu_i^a(x) &= \int_{\Rd \times \Delta^{n-1}} \eta(x) p_i(r) d \lambda_0^a(x,r) \\
 &=\begin{cases}  \int_{\Rd \times \Delta^{n-1}} \eta(x) \left( (1-a) r_i + \frac{a}{n} \right) d \lambda_0(x,r) &\text{ for } i = 1 , \dots, n-1 , \\ 
 \int_{\Rd \times \Delta^{n-1}} \eta(x) \left( (1 -a) \left[1-\sum_{i=1}^{n-1}   r_i   \right] + \frac{a}{n}    \right) d \lambda_0(x,r) &\text{ for } i = n .
 \end{cases} 
 \end{align*}
 Thus, for all $i =1, \dots, n$,   the dominated convergence theorem ensures
 \begin{align} \label{narrowconvergenceasatozero} 
  \mu_i^a \xrightarrow{a \to 0} \mu_i \text{ and } \nu_i^a \xrightarrow{a \to 0} \nu_i \text{ narrowly }
  \end{align}
   and
\begin{align*}
\mu^a_i(\Rd) &=   (1-a) \mu_i(\Rd) + \frac{a}{n}  >0  \ , \quad \nu^a_i(\Rd)  =   (1-a) \nu_i(\Rd) + \frac{a}{n}  >0   .
\end{align*}

We now estimate $W_{\Omega \times \G} (\bmu, \bmu^{a} )$ by considering a candidate solution of the vector valued continuity equation from $\bmu$ to $\bmu^a$. Recall the measurable map $(r,\tilde{r})\mapsto \gamma_{r, \tilde{r}}$  from Lemma \ref{simplexgeodesicapproximationlem}, which sends pairs of points in $\Delta^{n-1}$ to a geodesic between them.
For $\bar{\mu}$-a.e. $x \in \Rd$ and $i=1, \dots, n$, let 
\[ r_0(x) = \bp^{-1} \left(\left[\frac{d \mu_i}{d \bar{\mu}}(x) \right]_{i=1}^n \right) \text{ and } r_0^1(x) = \bp^{-1} \left(\left[\frac{d \mu^a_i}{d \bar{\mu}}(x) \right]_{i=1}^n \right) , \]
and let $\beff_{t}(x) = \bp \left( \gamma_{r_0(x),r_0^a(x)}(t) \right)$. By definition, for $\bar{\mu}$-a.e. x, $t \mapsto \beff_t(x)$ is the  $W_\G$ geodesic from $ [\frac{d \mu_i}{d \bar{\mu}}(x)  ]_{i=1}^n \in \P(\G)$  to $ [\frac{d \mu_i^a}{d \bar{\mu}}(x)  ]_{i=1}^n \in \P(\G)$. Let $\bv_t(x) = [v_{ij,t}(x)]_{i,j=1}^n$ denote the   velocity field for which $(\beff(x),\bv(x))$ solves the graph continuity equation on $[0,1]$ for $\bar{\mu}$-a.e. $x \in \Rd$ (see Definition \ref{graphctydef}) and 
\[ \| \bv_t(x) \|_{{\rm Tan}_{\beff_t(x)}\P(\G)} = W_\G \left( \frac{ d \mu_i}{d \bar{\mu}}(x), \frac{ d \mu_i^a}{d \bar{\mu}}(x) \right)   \text{ for all } t \in [0,1], \text{ $\bar{\mu}$-a.e.} \, x\in \R^d  . \]
Such a $\bv$ exists by Proposition \ref{basicfactWG} (\ref{WGgeodesics}).
  The fact that $(\beff(x),\bv(x))$ satisfies the graph continuity equation and $\boldsymbol{f}_t(x)$ is measurable in $(x,t)$ ensures that, up to redefining $\bv_t(x)$ where $\check{\beff}_t(x)$ vanishes, we also have that $\bv_t(x)$ is measurable in $(x,t)$. Define $\brho \in C([0,1]; \P(\Omega \times \G))$ by $d\brho_t(x) = \beff_t(x) d \bar{\mu}(x)$, so $d\bar{\rho}_t = d \bar{\mu}$, $\ \forall t \in [0,1]$. Then, by Jensen's inequality and Remark \ref{interpfunorder},
\begin{align*}
& \frac12 \sum_{i,j = 1}^n q_{ij} \left( \int_0^1 \int_\Rd |v_{ij,t}(x)| \theta \left( \frac{ d \rho_{i,t}}{d \bar{\rho}_t }(x),  \frac{ d \rho_{j,t}}{d \bar{\rho}_t}(x) \right)  d \bar{\rho}_t(x) dt  \right)^2 \\
&\quad \leq \frac12 \sum_{i,j = 1}^n q_{ij} \int_0^1 \int_\Rd |v_{ij,t}(x)|^2 \theta \left( \frac{ d \rho_{i,t}}{d \bar{\rho}_t }(x),  \frac{ d \rho_{j,t}}{d \bar{\rho}_t}(x) \right)  d \bar{\rho}_t(x) dt  \\
&\quad =  \int_0^1 \int_\Rd \| \bv_t(x) \|^2_{{\rm Tan}_{\beff_t(x)}\P(\G)} d \bar{\mu}(x) dt \\
&\quad =   \int_\Rd W^2_\G \left( \frac{ d \mu_i}{d \bar{\mu}}(x), \frac{ d \mu_i^a}{d \bar{\mu}}(x) \right)  d \bar{\mu}(x).
 \end{align*}
 Thus, defining $\bu \equiv 0$ (i.e., there is no spatial transport), we see that $(\brho, \bu, \bv)$ is a solution of the vector valued continuity equation, in the sense of Definition \ref{def:ContEqMultiSpecies},   from $\brho_0 = \bmu$ to $\brho_1 = \bmu^a$ and
\begin{align*}
 W^2_{\Omega \times \G} (\bmu, \bmu^{a} ) \leq \int_0^1 \|(\bu_t,\bv_t)\|_{\brho_t}^2 dt & = \int_0^1 \int_\Rd \| \bv_t(x) \|^2_{{\rm Tan}_{\beff_t(x)}\P(\G)} d \bar{\mu}(x) dt 
 \\ & =
 \int_\Rd W^2_\G \left( \frac{ d \mu_i}{d \bar{\mu}}(x), \frac{ d \mu_i^a}{d \bar{\mu}}(x) \right)  d \bar{\mu}(x)  .
 \end{align*}
 
 By definition, for each $i =1, \dots, n$, $0 \leq  \frac{ d \mu_i^a}{d \bar{\mu}}(x) \leq 1$ $\bar{\mu}$-a.e., so $\left\{  \frac{ d \mu_i^a}{d \bar{\mu}}(x) \right\}_{a >0}$ is bounded in $L^2(\bar{\mu})$. Thus, up to a subsequence, as $a \to 0$, it converges in $L^1(\bar{\mu})$ and $\bar{\mu}$-almost everywhere.  By equation (\ref{narrowconvergenceasatozero}) and uniqueness of limits, we conclude that $ \frac{ d \mu_i^a}{d \bar{\mu}}(x)  \to  \frac{ d \mu_i}{d \bar{\mu}}(x) $ $\bar{\mu}$-almost everywhere. Therefore, by the fact that $W_\G$ has finite diameter and it induces an equivalent topology to the usual Euclidean distance (see Proposition \ref{basicfactWG}), the dominated convergence theorem applied to the above inequality ensures
 \begin{equation}
   \lim_{a \rightarrow 0} W_{\Omega \times \G}(\bmu, \bmu^a)=0, \quad \lim_{a \rightarrow 0} W_{\Omega \times \G}(\bnu, \bnu^a)=0. 
   \label{eqn:SameMass2}
\end{equation}

On the other hand, 
\begin{align*}
 W_{\Rd \times \Delta^{n-1}}^2(\lambda_0^a, \lambda_1^a) &\leq \int_{\Rd \times \Delta^{n-1}} d_{\Rd \times \Delta^{n-1}}^2((x,r),(\tilde x, \tilde r)) d\Gamma^a((x,r),(\tilde x , \tilde r)) , \\
 &=   \int d_{\Rd \times \Delta^{n-1}}^2(S^a(x,r),S^a(\tilde x, \tilde r)) d\Gamma((x,r),(\tilde x , \tilde r)). \end{align*}
Using  {Case 1} for the measures $\bmu^a, \bnu^a$, and the dominated convergence theorem, we deduce
\begin{align*}
\limsup_{a \rightarrow 0} W_{\Omega \times \G}^2(\bmu^a, \bnu^a) \leq    \limsup_{a \rightarrow 0} W_{\Rd \times \Delta^{n-1}}^2(\lambda_0^a, \lambda_1^a) & \leq   \int d_{\Rd \times \Delta^{n-1}}^2((x,r), (\tilde x , \tilde r)) d \Gamma((x,r),(\tilde x , \tilde r)) 
\\ & = W_{\R^d \times \Delta^{n-1}}^2(\lambda_0, \lambda_1).
\end{align*}
From the above, the triangle inequality, and \eqref{eqn:SameMass2} it follows that
\[  W_{\Omega \times \G}(\bmu, \bnu) \leq  W_{\R^d \times \Delta^{n-1}}(\lambda_0, \lambda_1).   \]
Finally, since the above is true for all $\lambda_0, \lambda_1$ whose projections are $\bmu, \bnu$, respectively, we deduce that $W_{\Omega \times \G}(\bmu,\bnu) \leq D_{\Rd \times \G}(\bmu,\bnu)$ for all $\bmu,\bnu \in \P_2(\Omega \times \G)$.
\end{proof}

\begin{rem}
 \label{rem:TypeTotal} 
 Note that, in the first step in the previous proof, Proposition \ref{prop:RegularMeasures} (\ref{balancedmass}) allowed us  to reduce the convergence of $\bmu^a$ toward $\bmu$ in $W_{\Omega \times \G}$ to proving the convergence of $\mu_i^a$ to $\mu_i$ in the classical Wasserstein metric on $\P_2(\Omega)$. The latter convergence, in turn, was a direct consequence of Proposition \ref{prop:RegularMeasures} (\ref{Wconvergenceatozero}).   
\end{rem}

\subsection{Proof of bi-H\"older equivalence}
\label{sec:TopoEquiv}

As a consequence of our main result, Theorem \ref{thm:Main}, which shows $W_{\Omega \times \G} \leq D_{\Rd \times \G} \leq W_{2, \mathcal{W}}$, we are able to obtain Corollary \ref{thm:EquivalenceTopologies}, which relates all four (semi)-metrics on $\P_2(\Omega\times \G)$ and, in the case that $\Omega \subseteq \Rd$ is convex and compact,  shows that all  are bi-H\"older equivalent to the bounded Lipschitz metric $d_{BL}$, thus induce the same topologies.

\begin{proof}[Proof of Corollary \ref{thm:EquivalenceTopologies}]
Inequality (\ref{inequalityallmetrics}) is an immediate consequence of  Proposition \ref{BLlem} and Theorem \ref{thm:Main}, so it remains to prove the second inequality.    Suppose $\Omega$ is bounded.
For every $\lambda_1, \lambda_2 \in \mathcal{P}(\Omega \times \Delta^{n-1}) \subseteq \mathcal{P}(\Rd \times \Delta^{n-1}) $ we have
\begin{align} \label{W1toW2} W_{\Rd \times \Delta^{n-1} }(\lambda_1, \lambda_2) \leq C_{\Omega, \Delta^{n-1}} \sqrt{W_{1, \Rd \times \Delta^{n-1}}(\lambda_1, \lambda_2)}, 
\end{align}
   where $W_{1, \Rd \times \Delta^{n-1}}$ is the 1-Wasserstein distance over $\mathcal{P}_1(\Rd \times \Delta^{n-1})$, defined in terms of the distance function $d_{\Rd \times \Delta^{n-1}}$. By Kantorovich-Rubinstein duality, we have
   \[   W_{1, \Rd \times \Delta^{n-1}}(\lambda_1, \lambda_2) = \sup_{f \text{ s.t. }\mathrm{Lip}_{\Rd \times \Delta^{n-1}}(f) \leq 1 }  \int f d (\lambda_1- \lambda_2) ,  \]
where we use $\mathrm{Lip}_{\Rd \times \Delta^{n-1}}(f)$ to denote the Lipschitz semi-norm of $f: \Rd \times \Delta^{n-1} \rightarrow \R$ with respect to the metric $d_{\Rd \times \Delta^{n-1}}$. Furthermore, note that modifying $f$ by addition of a constant does not change the value of the supremum in the above equation. Thus, the supremum can be taken to be over all $f$ with $\mathrm{Lip}_{\Rd \times \Delta^{n-1}}(f) \leq 1 $ and $\lVert f  \rVert_\infty \leq C_{\Omega, \Delta^{n-1}}$. For any such $f$, the maps $x \in \Omega \mapsto   f(x, e_i)$ satisfy $\|f(\cdot, e_i) \|_{W^{1,\infty}_2(\Rd)} \leq (1+ C_{\Omega,\Delta^{n-1}}^2)^{1/2}$, where $e_i$ denote the vertices of the simples for $i=1, \dots, n$. Thus, if $\lambda_\bmu$ and $\lambda_\bnu$ are the canonical liftings of $\bmu, \bnu \in \P(\Omega \times \G)$,
\begin{align*}
W_{1, \Rd \times \Delta^{n-1}}  (\lambda_{\bmu}, \lambda_{\bnu})  &= \sup_{f \text{ s.t. }\mathrm{Lip}_{\Rd \times \Delta^{n-1}}(f) \leq 1 } \int f d (\lambda_\bmu - \lambda_\bnu) \\
& = \sup_{f \text{ s.t. }\mathrm{Lip}_{\Rd \times \Delta^{n-1}}(f) \leq 1 } \sum_{i =1 }^n    \int_{\Omega} f(x, e_i) d (\mu_i-  \nu_i)(x)  \\
&\leq (1+ C_{\Omega,\Delta^{n-1}}^2)^{1/2} \sum_{i=1}^n  \|\mu_i-\nu_i\|_{BL} \\
&\leq \sqrt{n}  (1+ C_{\Omega,\Delta^{n-1}}^2)^{1/2}  d_{BL}(\bmu,\bnu) \end{align*}
where the bounded Lipschitz distance is as in equation (\ref{BLdef}). Since $W_{2, \mathcal{W}}(\bmu, \bnu) = W_{\Rd \times \Delta^{n-1}} (\lambda_\bmu, \lambda_{\bnu})$, combining this with inequality (\ref{W1toW2}) gives the result.
\end{proof}

As an immediate consequence, we are able to likewise complete the proof of our main theorem for the dynamic metric, Theorem \ref{maindynamictheorem}.

\begin{proof}[Proof of Theorem \ref{maindynamictheorem}]
By Corollary \ref{thm:EquivalenceTopologies},  for any  $\bmu,\bnu \in \P_2(\Omega \times \G)$, $W_{\Omega \times \G}(\bmu,\bnu)<+\infty$. Proposition \ref{almostdynamicmetric} then ensures that $W_{\Omega \times \G}$ is a metric, and Theorem \ref{existenceofminimizers} ensures minimizers exist.

\end{proof}

We conclude this section by discussing the implications of the bi-H\"older equivalence from the perspective of linearization. Indeed, we prove Proposition \ref{prop:LOT}, which states that, under Assumption \ref{Assumption:Uniqueness}, the topology induced by the linearized vector valued OT distance $d_{\mathrm{LOT}}$ defined in \eqref{def:LOT} is equivalent to the topologies induced by   the metrics $W_{\Omega \times \G}, D_{\R^d \times \G}, W_{2, \mathcal{W}}$.

\begin{proof}[Proof of Proposition \ref{prop:LOT}]
By Proposition \ref{basicfactWG}(\ref{lem:ComparisonMetricsDiscrete}), convergence in $ d_{\mathrm{LOT}}$ is equivalent to convergence in the metric 
\[ \tilde{d}(\bmu , \bnu):=  \left( \int_{\R^d \times \Delta^{n-1}} d_{\R^d \times \Delta^{n-1}}(T_{\bmu}(x,r), T_{\bnu}(x,r))^2 d\lambda_{\mathrm{ref}}(x,r) \right)^{1/2}.  \]
On the other hand, because $(T_{\bmu} \times  T_{\bnu}) \# \lambda_{\mathrm{ref}}$ 
is a coupling between $\lambda_{\bmu}$ and $\lambda_{\bnu}$, it follows that
\[ W_{2, \mathcal{W}}(\bmu, \bnu)=   W_{\R^d \times \Delta^{n-1}}(\lambda_{\bmu}, \lambda_{\bnu}) \leq \tilde{d}(\bmu, \bnu).   \]
From the above two facts it follows that convergence in $d_{\text{LOT}}$ implies convergence in $W_{2,\mathcal{W}}$. 

Conversely, suppose that 
\begin{equation}
0 = \lim_{n \rightarrow \infty} W_{2, \mathcal{W}}(\bmu_n , \bmu)  =  \lim_{n \rightarrow \infty} W_{\R^d \times \Delta^{n-1} } (\lambda_{n}, \lambda),
\label{eqn:AUxdLOT}
\end{equation}
where we use $\lambda_n$ and $\lambda$ for the canonical lifts of $\bmu_n$ and $\bmu$, respectively. Let $T_n$  be the optimal transport map between $\lambda_{\text{ref}}$ and $\lambda_n$, and let $T$ be the optimal transport map between $\lambda_{\text{ref}}$ and $\lambda$, following Assumption \ref{Assumption:Uniqueness}. We need to show that
\[ \lim_{n \rightarrow \infty} \int \lVert T_n(x,r) - T(x,r) \rVert^2 d\lambda_{\text{ref}}(x,r) = 0,\] 
i.e., that $T_n$ converges toward $T$ in $L^2(\R^{d + n-1} :\R^{d + n-1}; \lambda_{\text{ref}} )$, which in the remainder of this proof we abbreviate as $L^2$. Now, since 
\[   \int \lVert T_n(x,r)  \rVert^2d \lambda_{\text{ref}}(x,r) = \int \lVert (\tilde x , \tilde r)  \rVert^2d \lambda_{n}(\tilde x,\tilde r)       , \quad  \int \lVert T(x,r)  \rVert^2d \lambda_{\text{ref}}(x,r) = \int \lVert (\tilde x , \tilde r)  \rVert^2d \lambda(\tilde x,\tilde r),   \]
and given equality \eqref{eqn:AUxdLOT}, it follows that
\begin{equation}
 \lim_{n \rightarrow \infty} \int \lVert T_{n}(x,r) \rVert^2 d\lambda_{\text{ref}}(x,r)  = \int \lVert T(x,r) \rVert^2 d\lambda_{\text{ref}}(x,r).
 \label{eqn:AUxdLOT2}
\end{equation}
Due to this, it will suffice to prove that $T_n$ converges weakly in $L^2$ toward $T$. In turn, note that equation \eqref{eqn:AUxdLOT2} implies that $\{  T_n\}_{n \in \NN}$ is weakly precompact in $L^2$ and thus, without the loss of generality, can be assumed to converge weakly in $L^2$ toward some $\xi \in L^2$. Our goal then reduces to showing that this $\xi$ is equal to $T$. 

To see this, we start by observing that, by \cite[Proposition 7.1.3]{ambrosiogiglisavare}, $\Gamma_n:=   (\text{Id} \times T_n)_{\sharp } \lambda_{\text{ref}}$ converges weakly, up to subsequence, toward an optimal transport plan $\Gamma$ between $\lambda_{\text{ref}}$ and $\lambda$. Since this plan is, by Assumption \ref{Assumption:Uniqueness}, unique and takes the form $\Gamma= (\text{Id} \times T)_{\sharp } \lambda_{\text{ref}}$, we must have
\begin{align*}
\lim_{n \rightarrow \infty} \int \langle T_n(x,r) , \phi(x,r) \rangle d \lambda_{\text{ref}}(x,r) &= \lim_{n \rightarrow \infty} \int \langle (\tilde x , \tilde r) , \phi(x,r) \rangle d\pi_n((x,r), (\tilde x , \tilde r) )
\\&  = \int \langle (\tilde x , \tilde r) , \phi(x,r) \rangle d\pi((x,r), (\tilde x , \tilde r) )
\\& = \int \langle T(x,r) , \phi(x,r) \rangle d \lambda_{\text{ref}}(x,r),
\end{align*}
for all $\phi$ bounded and continuous. Due to the weak convergence in $L^2$ of $T_n$ toward $\xi$ and uniqueness of limits, we conclude $\xi = T$. 
\end{proof}



    

\color{black}

\subsection{Examples of equality and inequality.}
\label{sec:Examples}

We conclude with a series of examples in the case of a two node graph that will both illustrate  the sharpness of   inequality (\ref{inequalityallmetrics}) in Corollary \ref{thm:EquivalenceTopologies}, relating all four (semi)-metrics, and show  that $D_{\Rd \times \G}$ fails the triangle inequality.  Our examples rely on the following lemma, which provides  a candidate solution of the vector valued continuity equation, in the case of a two node graph.

\begin{lem}[A solution of the vector valued continuity equation on two node graph] \label{singleparticlemulti}
Suppose   $\theta$ satisfies Assumption \ref{interpolationassumption} and $q_{ij} \equiv q>0$. Let $\bu(x,t) = [u_1(x,t) ,u_2(x,t)]^t$ and  $\bv(x,t) = [ v_{ij}(x,t)]_{i,j = 1}^2$  for $v_{12}(x,t) = -v_{21}(x,t)= v(x,t)$, where all functions are  piecewise continuous on $\R \times [0,1]$.  Suppose that
\begin{align*} \brho(t) &=  \begin{bmatrix}
          r(t) \delta_{x_1(t)} \\
 	  (1-r(t)) \delta_{x_2(t)} \\
         \end{bmatrix}    ,  \begin{bmatrix}
         \dot{x}_1(t) \\
 	   \dot{x}_2(t)  \\
         \end{bmatrix} =  \begin{bmatrix}
         u_1(x_1(t), t)  \\
 	   u_2(x_2(t),t)  \\
         \end{bmatrix}    ,  \ \dot{r}(t) = \begin{cases} q  v(x_1(t), t) \theta(r(t),1-r(t)) &\text{if } x_1(t) = x_2(t) ,\\ 0 & \text{otherwise.} \end{cases}   \end{align*}
         where the ODEs are satisfied  in $\mathcal{D}(0,T)'  $.
      Finally, suppose either that $\theta$ vanishes at the boundary, that is, $\theta(s,0)= 0$ for all $s \geq0$, or suppose  that $x_1(t) \equiv x_2(t)$ for all $t \in [0,T]$.

Then  $(\brho, \bu,\bv)$ solves the vector valued continuity equation, in the sense of Definition \ref{def:ContEqMultiSpecies}, and
\[     \| (\bu(t),\bv(t)) )\|_{\brho(t)}^2  = |\dot{x}_1(t)|^2 r(t) + |\dot{x}_2(t)|^2(1-r(t)) + \frac{ |\dot{r}(t)|^2}{q  \theta(r(t), 1-r(t))} , \text{ for all } t \in (0,1). \]
\end{lem}
\begin{proof}
  For fixed  $\eta \in C_{c}^{\infty}(\mathbb{R}^{d}),$ we prove that     equation (\ref{dualvvcty}) holds  in $\mathcal{D}(0,T)' $. Let $m_1(t) = r(t)$ and $m_2(t) = 1-r(t)$. Then, for $j \neq i$, \begin{align*}
     &\frac{d}{dt} \int_{\mathbb{R}} \eta(x)d\rho_{i,t}(x) = \frac{d}{dt}(m_i(t)\eta(x_{i}(t))) = m_i(t)\nabla \eta(x_{i}(t)) \cdot \dot{x}_{i}(t) + \dot{m}_i(t)\eta(x_{i}(t)) \\     
    &= \int_{\mathbb{R}^{d}} \nabla \eta(x) \cdot u_{i}(x,t)d\rho_{i,t}(x) +  1_{x_1(t) = x_2(t)}(t) \int_{\mathbb{R}^{d}} \eta(x)\theta \left( \frac{d\rho_{i}}{d\bar{\rho}}(x),\frac{d\rho_{j}}{d\bar{\rho}}(x) \right)  v_{ij}(x,t) q d\bar{\rho}(x) .
    \end{align*}
If $\theta$ vanishes at the boundary, then when $x_1(t) \neq x_2(t)$, $\theta \left( \frac{d\rho_{i}}{d\bar{\rho}}(x),\frac{d\rho_{j}}{d\bar{\rho}}(x) \right) = 0$ for $\bar{\rho}$-a.e. x. Thus, we may remove the characteristic function $1_{x_1(t) = x_2(t)}(t)$ from the second term and the equality remains unchanged. On the other hand, if $\theta$ does not vanish at the boundary, then $1_{x_1(t) = x_2(t)}(t) \equiv 1$, and we may again remove it. 
By Remark \ref{twonoderemark}, this shows that $(\brho, \bu,\bv)$ solves the vector valued continuity equation and the energy has the given form.  

\end{proof}
 
We next recall the explicit formula for the distance on the simplex $\Delta^{n-1}$ in the two node case, $\Delta^{1} = [0,1]$, which is an immediate consequence of Maas's   formula for the graph Wasserstein distance in the two node case \cite[Theorem 2.4]{maas2011gradient}.  We defer the proof to appendix \ref{realappendix}.
\begin{prop} \label{explicitdistance2node}
Suppose $\G$ is a two node graph with $q_{ij} \equiv 1$ and   $\theta$ satisfies   Assumption  \ref{interpolationassumption}. Then, for any $0 \leq a_0 \leq a_1 \leq 1$, 
\[d_{[0,1]}(a_0,a_1)   = \frac{1}{\sqrt{q}}  \int_{a_0}^{a_1} \frac{1}{\sqrt{ \theta(a,1-a)}} da. \]
\end{prop}

With these facts in hand, we now turn to our main example. 

\begin{prop}[Examples of inequality and equality] \label{examplesprop} Consider a two node graph $\G$ with  $q_{ij}\equiv 1$ and suppose $\theta$ satisfies Assumption  \ref{interpolationassumption}. In addition, suppose that $\theta$ vanishes at the boundary, that is,  $\theta(a,0) = 0$ for all $a \geq 0$. For $a \in (0,1)$ and $b \in [1/2,1]$ consider the   vector valued measures in $\P(\R \times \G)$,
\begin{align*}
\bmu^1 =  \begin{bmatrix}
          (1/2) \delta_0 \\
 	 (1/2) \delta_0 \\
         \end{bmatrix} \ , \quad \bmu^2 =  \begin{bmatrix}
          (1/2) \delta_{-a} \\
 	 (1/2) \delta_{a} \\
         \end{bmatrix}  \ , \quad \bmu^3 =  \begin{bmatrix}
          b \delta_{0} \\
 	 (1-b) \delta_{0} \\
         \end{bmatrix}.
\end{align*}
Then,
\begin{align*} W_{\R \times \G}(\bmu^1, \bmu^2) = a \ , \quad  W_{\R \times \G}(\bmu^1, \bmu^3)  =   d_{[0,1]}\left(1/2,b \right) , \quad W_{\R \times \G}(\bmu^2, \bmu^3) \leq  a+  d_{[0,1]}\left(1/2 ,b \right)   \ , \\
D_{\R \times \G}(\bmu^1, \bmu^2) = a \ , \quad  D_{\R \times \G}(\bmu^1, \bmu^3)  =d_{[0,1]}\left(1/2 ,b \right) , \quad D_{\R \times \G}(\bmu^2, \bmu^3)= \sqrt{a^2 + \frac12 d_{[0,1]}^2 ( 0, 2b-1  ) } . 
\end{align*}

In addition, for $b$ sufficiently close to $\frac12$ and $a$ sufficiently close to 0,
\begin{align} W_{\R \times \G}(\bmu^2, \bmu^3)  < D_{\R \times \G}(\bmu^2, \bmu^3) \label{strictinequalityman}\end{align}
and
\begin{align} \label{notriangleineq} D_{\R \times \G}(\bmu^2, \bmu^3) > D_{\R \times \G}(\bmu^1, \bmu^2) + D_{\R \times \G}(\bmu^1, \bmu^3) .
\end{align}
\end{prop}

\begin{proof}
First, we consider the distance between $\bmu^1$ and $\bmu^2$. Note that, by Lemma \ref{singleparticlemulti}, 
\[ \brho_t =  \begin{bmatrix}
          (1/2) \delta_{-ta} \\
 	 (1/2) \delta_{ta} \\
         \end{bmatrix} \]
is a solution of the vector valued continuity equation for $u_1 \equiv -a$, $u_2 \equiv a$, and $v \equiv 0$, with $\|(\bu_t,\bv_t)\|_{\brho_t} \equiv |a|$. Integrating in time, we see that $W_{\R \times \G}(\bmu^1, \bmu^2) \leq  |a|$. On the other hand, by Lemma \ref{BLlem} we obtain $W_{\R \times \G}(\bmu^1,\bmu^2) \geq \| \bar{\mu}^1- \bar{\mu}^2\|_{BL} = |a|$. Thus, $W_{\R \times \G}(\bmu^1, \bmu^2) =  |a|$.

Next, we consider $D_{\R \times \G}(\bmu^1, \bmu^2)$. Note that $\lambda =  \frac12 \delta_{(-a,1)} + \frac12 \delta_{(a,0)}$ is the only $\omega \in \P(\R \times [0,1])$ with $\PP \omega = \bmu^2$. Likewise, note that  $\PP \sigma = \bmu^1$ only if $\sigma$ is supported on $x =0$. Therefore,
\begin{align*}
D^2_{\R \times \G}(\bmu^1, \bmu^2) &= \inf \left\{W_{{\R \times [0,1]}}^2(\sigma, \omega) : \PP \sigma = \bmu^1, \PP \omega = \bmu^2 \right\}  \\
&= \inf \left\{W_{{\R \times [0,1]}}^2(\sigma, \lambda) : \PP \sigma = \bmu^1  \right\}  \\
&= \inf_{\gamma \in \P((\R \times[0,1])^2)} \left\{ \int  |x_1 - x_2|^2 + d^2_{[0,1]}(r_1, r_2)  \ d \gamma  :  \PP (\pi_1 \# \gamma) = \bmu^1, \pi_2 \# \gamma = \lambda  \right\}  \\
&= \inf_{\gamma \in \P((\R \times[0,1]))^2} \left\{ \int |a|^2 + d^2_{[0,1]}(r_1, r_2)  \ d \gamma  : \PP (\pi_1 \# \gamma) = \bmu^1, \pi_2 \# \gamma = \lambda  \right\}  \\
&\geq |a|^2 .
\end{align*}
Conversely, since $\sigma = \frac12 \delta_{(0,1)} + \frac12 \delta_{(0,0)}$ satisfies $\PP \sigma = \bmu^1$ and $W_{\R \times [0,1]}(\sigma, \lambda) = |a|$, we see that $D_{\R \times \G}(\bmu^1, \bmu^2)  = |a|$.
         
 Next, we compute the distance between $\bmu^1$ and $\bmu^3$. By \cite[Proposition 2.7]{maas2011gradient}, there exists a  continuously differentiable solution $r: [0,1] \to [0,1]$ to
\[ \begin{cases} \dot{r}(t) = \sqrt{  q}d_{[0,1]}(1/2,b) \sqrt{ \theta(r(t), 1-r(t)) },  \\
r(0) = 1/2 , \end{cases}\]
that satisfies $r(1) = b$.
Then, by Lemma \ref{singleparticlemulti}, 
\[ \brho(t) =  \begin{bmatrix}
         r(t)  \delta_{0} \\
 	(1-r(t)) \delta_{0}\\
         \end{bmatrix} \]
is a solution of the vector valued continuity equation for $u_1 \equiv 0$, $u_2 \equiv 0$, and 
\[ v(x,t) = \frac{d_{[0,1]}(1/2,b)}{\sqrt{q\theta(r(t),1-r(t))}} , \]
with   $\|(\bu(t),\bv(t))\|_{\brho(t)} \equiv   d_{[0,1]}(1/2,b) $. Integrating in time, we see that $W_{\R \times \G}(\bmu^1, \bmu^3) \leq  d_{[0,1]}(1/2,b).$ 

To see the opposite inequality,  consider $(\brho,\bu,\bv) \in \mathcal{C}(\bmu^1, \bmu^3)$ that is a minimizer in the definition of $W_{\R \times \G}(\bmu^1, \bmu^3)$, as is guaranteed to exist by Theorem \ref{existenceofminimizers}. By Remark \ref{gradsofpot}, since we are interested in minimizers of the action, we may assume without loss of generality that $v_{ij} = -v_{ji}$. Thus, by Remark \ref{twonoderemark} and Jensen's inequality for the convex function $\alpha$ from equation (\ref{definitionofpsi}), for almost every $t \in [0,1]$,
 \begin{align} \label{actionexample}
    \|(\bu, \bv)\|_{\brho}^2 &\geq   \int_\R v_{12}^{2} \theta \left( \frac{d\rho_1}{d \bar{\rho}} ,\frac{d\rho_2}{d\bar{\rho}} \right) q  d\bar{\rho}  = q \int_\R \alpha  \left( v_{12}\theta \left(\frac{d\rho_{1}}{d\bar{\rho}}, \frac{d\rho_{2}}{d\bar{\rho}} \right), \frac{d\rho_{1}}{d\bar{\rho}}, \frac{d\rho_{2}}{d\bar{\rho}} \right)  d\bar{\rho}  \\
    &\geq q \alpha \left( \int_{\R}v_{12}\theta \left(\frac{d\rho_{1}}{d\bar{\rho}}, \frac{d\rho_{2}}{d\bar{\rho}} \right)  d\bar{\rho}, \int_{\R} d\rho_{1}, \int_{\R} d\rho_{2} \right) . \nonumber
\end{align} 
Define $p_i(t) = \int_{\R} d\rho_{i,t}$ for $i=1,2$. Fix a smooth, decreasing cutoff function $\eta:[0,+\infty) \to [0,1]$ such that $\eta \equiv 1$ on $[0,1]$, $\eta \equiv 0$ on $[2,+\infty)$ and $\|\eta'\|_\infty \leq 2$, and consider its dilations $\eta_R:\R \to \R$ given by $\eta_R(x) = \eta(|x|^2/R^2)$. Then $\eta_R \in C^\infty_c(\R)$ for all $ R>0$, so Remarks \ref{dualvvcty} and \ref{twonoderemark} ensure that, for all $t_0, t_1 \in [0,1]$,
\begin{align*}
\int_{t_0}^{t_1} \int_\R \eta_R d \rho_{1,t} dt = \int_{t_0}^{t_1} \int_\R \nabla \eta_R \cdot u_{1,t} d \rho_{1,t} dt +  \int_{t_0}^{t_1} \int_{\R} \eta_R \theta \left( \frac{ d \rho_{1,t}}{d \bar{\rho}_t },  \frac{ d \rho_{2,t}}{d\bar{\rho}_t}\right)  v_{12,t}  q d\bar{\rho}_t  dt .
\end{align*}
The fact that $\eta_R \nearrow 1$ pointwise as $R \to +\infty$ combined with  the integrability hypotheses required in our notion of solution to the vector valued continuity equation, Definition \ref{def:ContEqMultiSpecies}, ensures that, sending $R \to +\infty$, we obtain, for all $t_0, t_1 \in [0,1]$,
\begin{align*}
\int_{t_0}^{t_1} \int_\R d \rho_{1,t} dt =    \int_{t_0}^{t_1} \int_{\R}   \theta \left( \frac{ d \rho_{1,t}}{d \bar{\rho}_t },  \frac{ d \rho_{2,t}}{d\bar{\rho}_t}\right)  v_{12,t}  q d\bar{\rho}_t  dt .
\end{align*}
Thus, $t \mapsto p_1(t)$ is absolutely continuous and 
\[ \frac{d}{dt} p_1(t) = q \int_\R \theta \left( \frac{ d \rho_{1,t}}{d \bar{\rho}_t },  \frac{ d \rho_{2,t}}{d\bar{\rho}_t}\right)  v_{12,t}   d\bar{\rho}_t =: q\sigma(t)   . \]
An analogous cutoff argument shows $p_2(t) = 1-p(t)$ and $\frac{d}{dt} p_2 = -q\sigma$. Hence,    \cite[Proposition 3.3]{erbar2012ricci} ensures that $p(t)$   is absolutely continuous with respect   to $W_\G$, with metric derivative bounded by 
\[ |p'|^2_{W_\G}(t) \leq q \alpha \left( \sigma(t), p_1(t), p_2(t) \right) , \text{ for a.e. } t \in [0,1]. \]
Therefore, since $p_1(0)= 1/2$ and $p_1(1) = b$, using the definition of the metric derivative and the bound from inequality (\ref{actionexample}) we have
 \begin{align*}
    d_{[0,1]} \left(\frac{1}{2},b \right) = W_\G \left( \begin{bmatrix}1/2 \\ 1/2 \end{bmatrix},  \begin{bmatrix}b \\ 1-b \end{bmatrix} \right) \leq \int_0^1  |p'|_{W_\G}(t) dt  & \leq   \int_0^1 \sqrt{q \alpha \left( \sigma(t), p_1(t), p_2(t) \right)} dt  \\ & \leq  \left( \int_0^1    \|(\bu, \bv)\|_{\brho}^2 dt \right)^{1/2},
\end{align*} 
where the right hand side coincides with $W_{\Rd \times [0,1]}(\bmu^1,\bmu^3)$ given that $(\brho,\bu,\bv) \in \mathcal{C}(\bmu^1, \bmu^3)$ was a minimizer.
Therefore, we conclude that  $W_{\R \times \mathcal{G}}(\bmu^1, \bmu^3) \geq  d_{[0,1]}(1/2,b)$.

Next, we consider $D_{\R \times \G}(\bmu^1, \bmu^3)$. By definition, $\PP \sigma = \bmu^1$ and $\PP \omega = \bmu^3$ only if $d \sigma(x,r) = d \delta_0(x) \otimes d \tilde{\sigma}(r)$ and $d \omega(x,r) = d \delta_0(x) \otimes d \tilde{\omega}(r)$ for $\tilde{\sigma}, \tilde{\omega} \in \P([0,1])$ with $\int r d \tilde{\sigma}(r) = \frac12$ and $\int r d \tilde{\omega}(r) = b$. Furthermore, note that, since $W_\G^2$ is jointly convex along liner interpolations \cite[Proposition 2.11]{erbar2012ricci} and the isometry $\bp: [0,1] \to \P(\G)$ is linear, $d_{[0,1]}^2$ is jointly convex along linear interpolations. Thus,
\begin{align*}
&D^2_{\Omega \times \G}(\bmu^1, \bmu^3)\\
 &\quad = \inf \left\{W_{{\R \times [0,1]}}^2(\sigma, \omega) : \PP \sigma = \bmu^1, \PP \omega = \bmu^3 \right\}  \\
&\quad= \inf_{\gamma \in \P((\R \times[0,1])^2)} \left\{ \int  |x_1 - x_2|^2 + d^2_{[0,1]}(r_1, r_2)  \ d \gamma  : \PP (\pi_1 \# \gamma) = \bmu^1, \PP(\pi_2 \# \gamma) = \bmu^2  \right\}  \\
&\quad= \inf_{\gamma \in \P((\R \times[0,1])^2)} \left\{ \int  0+ d^2_{[0,1]}(r_1, r_2)  \ d \gamma  : \PP (\pi_1 \# \gamma) = \bmu^1, \PP(\pi_2 \# \gamma) = \bmu^2  \right\}  \\
&\quad= \inf_{\tilde{\gamma} \in \P( [0,1]^2), \ \tilde{\omega}, \tilde{\sigma} \in \P([0,1])} \left\{ \int   d^2_{[0,1]}(r_1, r_2)   d \tilde{\gamma}  :  \pi_1 \# \tilde{\gamma} = \tilde{\sigma}, \pi_2 \# \tilde{\gamma} = \tilde{\omega} , \int r d \tilde{\sigma}(r) = \frac12,   \int r d \tilde{\omega}(r) = b \right\}  \\
&\quad \geq d^2_{[0,1]}(1/2,b) ,
\end{align*}
where the last inequality follows from Jensen's inequality. Finally, we see that, for $\tilde{\omega} = \delta_b$ and $\tilde{\sigma} =\delta_{1/2}$, the infimum attains $d_{[0,1]}^2(1/2,b)$. Thus $D_{\Omega \times \G}(\bmu^1, \bmu^3) = d_{[0,1]}(1/2,b)$.
         
It remains to consider the distance and dissimilarity between $\bmu^2$ and $\bmu^3$. In order to show the bound on the dynamic distance, it suffices to find a path $(x_1(t),x_2(t),r(t))$ so that  the empirical measure $\brho(t)$ from Lemma \ref{singleparticlemulti} achieves action $ \left( a +  d_{[0,1]}\left(\frac12 ,b \right) \right)^2$. Fix $t_0 \in (0,1)$. Consider the paths
\begin{align}
x_1(t) &= \begin{cases} (t-t_0) a/ t_0 &\text{ for $t \in [0,t_0]$} , \\
0 &\text{ for $t \in [t_0, 1]$} \end{cases} \quad \ 
x_2(t) = \begin{cases} (t_0-t) a/ t_0&\text{ for $t \in [0,t_0]$} , \\
0 &\text{ for $t \in [t_0, 1]$} \end{cases} \\
r(t) &= \begin{cases} \frac{1}{2} &\text{ for $t \in [0,t_0]$} , \\
\tilde{r} \left( \frac{t-t_0}{1-t_0} \right)  &\text{ for $t \in [t_0, 1]$}\end{cases}
\end{align}
where $\tilde{r}:[0,1] \to [0,1]$ is the continuously differentiable solution  to
\[ \begin{cases} \dot{\tilde{r}}(t) = \sqrt{q} d_{[0,1]}(1/2,b) \sqrt{ \theta(\tilde{r}(t), 1-\tilde{r}(t)) } &\text{ for } t \in [0,1],  \\
\tilde{r}(0) = 1/2 , \end{cases}\]
which, by \cite[Proposition 2.7]{maas2011gradient}, exists and is the geodesic from $r(0) = 1/2$ to $r(1) = b$.
Via Lemma \ref{singleparticlemulti}, with
\[ u_1(x,t) = \frac{a}{t_0} 1_{[0,t_0]}(t) , \ \  u_2(x,t) = -\frac{a}{t_0} 1_{[0,t_0]}(t)  , \ \ v(x,t) = \frac{d_{[0,1]}(1/2,b)}{\sqrt{q\theta(r(t), 1-r(t))}} \frac{1}{1-t_0} 1_{[t_0,1]}(t) , \]
these paths induce a measure $\brho(t)$ from $\bmu_1$ to $\bmu_3$ with
\[ \int_0^1  \| (\bu(t),\bv(t)) )\|_{\brho(t)}^2 dt =  \frac{a^2}{t_0} + \frac{d^2_{[0,1]}(1/2,b)}{1-t_0} .\]
Finally, optimizing over $t_0 \in (0,1)$, we see that the action is minimized at  $t_0 = a/(a+  d_{[0,1]}(1/2,b))$, and for this value of $t_0$, 
\[\int_0^1  \| (\bu(t),\bv(t)) )\|_{\brho(t)}^2 dt = \left(a +    d_{[0,1]}(1/2,b) \right)^2 . \]
Thus, $W_{\R \times \G}(\bmu^2, \bmu^3) \leq  a+   d_{[0,1]}\left(1/2 ,b \right) $.
         
Finally, we compute $D_{\R \times \G}(\bmu^2, \bmu^3) $. As before, $\lambda =  \frac12 \delta_{(-a,1)} + \frac12 \delta_{(a,0)}$ is the only $\lambda \in \P({\R \times [0,1]})$ so that $\PP \lambda = \bmu^2$. Similarly, $\PP \omega = \bmu^3$ only if  $d \omega(x,r) = d \delta_0(x) \otimes d \tilde{\omega}(r)$ for $ \tilde{\omega} \in \P([0,1])$ with $\int r d \tilde{\omega}(r) = b$. Thus,
\begin{align*}
D_{\R \times \G}(\bmu^2, \bmu^3) &= \inf \left\{W_{{\R \times [0,1]}}(\lambda, \omega) : \PP \omega = \bmu^3 \right\}   \\
&= \inf_{\gamma \in \P((\R \times [0,1])^2)} \left\{ \left( \int |x_1 - x_2|^2 + d^2_{[0,1]}(r_1, r_2)  \ d \gamma \right)^{1/2} :  \pi_1 \# \gamma = \lambda, \PP(\pi_2 \# \gamma) = \bmu^2  \right\} \nonumber  \\
&=    \inf_{\gamma \in \P((\R \times [0,1])^2)} \left\{ a^2+ \int   d^2_{[0,1]}(r_1, r_2)  \ d \gamma  :  \pi_1 \# \gamma = \lambda, \PP(\pi_2 \# \gamma) = \bmu^2  \right\} \nonumber \\
&= a^2 + \inf_{\tilde{\omega} \in \P(\R \times [0,1])} \left \{ W^2_{[0,1]} \left( \frac12 \delta_0 + \frac12 \delta_1, \tilde{\omega} \right) : \int r d \tilde{\omega}(r) = b \right\} . \nonumber
\end{align*}  
       
For any transport plan $\tilde{\gamma}$ between $\frac12 \delta_0 + \frac12 \delta_1$ an $\tilde{\omega}$, disintegrating with respect to its first marginal gives
\[ d \tilde{\gamma}(r_1, r_2) = d \tilde{\gamma}_{r_1}(r_2) d (\pi_1 \# \tilde{\gamma})(r_1) = \frac12 d \tilde{\gamma}_0(r_2)  \otimes \delta_0(r_1) + \frac12 d \tilde{\gamma}_1(r_2)  \otimes \delta_1(r_1) . \] Then, using the convexity of $d_{[0,1]}^2$ along linear interpolations and Jensen's inequality, we obtain
\begin{align*} 
&\inf_{\tilde{\omega} \in \P([0,1])} \left \{ W^2_{[0,1]} \left( \frac12 \delta_0 + \frac12 \delta_1, \tilde{\omega} \right) : \int r d \tilde{\omega}(r) = b \right\}  \nonumber \\
&\quad = \inf_{\tilde{\gamma} \in \P(([0,1])^2)} \left \{ \int d_{[0,1]}^2(r_1,r_2) d \tilde{\gamma}(r_1,r_2) : \pi_1 \# \tilde{\gamma} = \frac12 \delta_0 + \frac12 \delta_1 , \ \int r_2 d \tilde{\gamma}(r_1,r_2) = b \right\} \nonumber \\
&\quad = \inf_{\tilde{\gamma}_0, \tilde{\gamma}_1 \in \P([0,1])} \left \{ \frac12 \int d_{[0,1]}^2(0,r_2) d \tilde{\gamma}_0(r_2) +\frac12 \int d_{[0,1]}^2(1,r_2) d \tilde{\gamma}_1(r_2):   \frac12 \int r_2 d \tilde{\gamma}_0(r_2)+ \frac12 \int r_2 d \tilde{\gamma}_1(r_2) = b \right\} \nonumber \\
&\quad \geq \inf_{\tilde{\gamma}_0 , \tilde{\gamma}_1 \in \P([0,1])} \left \{ \frac12  d_{[0,1]}^2\left(0, \int r_2 d \tilde{\gamma}_0(r_2)\right) +\frac12  d_{[0,1]}^2 \left(1,  \int r_2 d \tilde{\gamma}_1(r_2) \right):   \frac12 \int r_2 d \tilde{\gamma}_0(r_2)+ \frac12 \int r_2 d \tilde{\gamma}_1(r_2) = b \right\} \nonumber \\
&\quad = \inf_{b_0, b_1 \in [0,1]} \left \{ \frac12  d_{[0,1]}^2\left(0, b_0 \right) +\frac12  d_{[0,1]}^2 \left(1, b_1 \right):   \frac12 b_0 + \frac12 b_1 = b \right\} \nonumber \\
&\quad = \inf_{ b_1 \in [2b-1,1]} \left \{ \frac12  d_{[0,1]}^2\left(0, 2b-b_1 \right) +\frac12  d_{[0,1]}^2 \left(1, b_1 \right)  \right\} \nonumber \\
&\quad =  \frac12  d_{[0,1]}^2\left(0, 2b-1 \right) 
,
\end{align*}
where the last equality follows by the fact that the first and second terms on the second to last line are both decreasing in $b_1$. Furthermore, taking $\tilde{\omega} = \frac12 \delta_{2b-1} + \frac12 \delta_1$ in the original infimum shows that, in fact, equality holds throughout.  Thus, $D_{\R \times \G}(\bmu^2, \bmu^3) =  \sqrt{a^2 + \frac12 d_{[0,1]}^2 ( 0, 2b-1  ) }$.
%
%
%
 
We now show that the dynamic distance is strictly smaller than the static distance when $b$ is sufficiently close to $\frac12$ and $a$ is sufficiently small. By the continuity of $\theta$, assumption (A\ref{thetaregularity}), the normalization of $\theta$, assumption (A\ref{thetanormalization}), and the fact that $\theta$ vanishes at the boundary, that is,  $\theta(a,0) = 0$ for all $a \geq 0$, we have $\lim_{a \to 0^+} \theta(a,1-a) = 0$ and $\lim_{a \to 1/2} \theta(a,1-a) = 1/2$. Thus, the Lebesgue differentiation theorem ensures that, for $c>0$ sufficiently small,
\begin{align*}
\frac{1}{c} \int_0^{c} \frac{1}{\sqrt{ \theta(a,1-a)}}   da  > \frac{2}{c} \int_{1/2}^{(c+1)/2} \frac{1}{\sqrt{ \theta(a,1-a)}}   da .
\end{align*}
Combining this with the explicit formula for the distance $d_{[0,1]}$ from Proposition \ref{explicitdistance2node}, for $c>0$ sufficiently small, $d_{[0,1]}(0,c) > 2d_{[0,1]}(1/2,(c+1)/2)$ and thus 
\begin{align} \label{ceqn} d^2_{[0,1]}(0,c) > 4d^2_{[0,1]}(1/2,(c+1)/2). \end{align}
Therefore,  
\begin{align} \label{dynamicneqstaticeqn}
 & W^2_{\Omega \times \G}(\bmu^2, \bmu^3)  < D^2_{\Omega \times \G}(\bmu^2, \bmu^3)   \\
  &\quad \impliedby a^2 +2ad_{[0,1]}(1/2,b)   +d^2_{[0,1]}(1/2,b)  < a^2 + \frac12 d_{[0,1]}^2 ( 0, 2b-1  )  \nonumber \\
 &\quad \iff 2a + d_{[0,1]}(1/2,b)  <  \frac{  d_{[0,1]}^2 ( 0, 2b-1  ) }{2d_{[0,1]}(1/2,b)} \nonumber \\
 &\quad \iff a  < \frac12 \left( \frac{  d_{[0,1]}^2 ( 0, 2b-1  )}{2d_{[0,1]}(1/2,b)} - d_{[0,1]}(1/2,b) \right) = \frac12 \left( \frac{  d_{[0,1]}^2 ( 0, 2b-1  ) - 2d^2_{[0,1]}(1/2,b)}{2d_{[0,1]}(1/2,b)}  \right) ,\nonumber
\end{align}
where equation (\ref{ceqn}) ensures that  the right hand side is strictly positive for $c=2b-1$ and $b>1/2$ sufficiently close to $1/2$.

Finally, we show this condition on $a$ is also sufficient for the triangle inequality to be violated. Indeed,
\begin{align*}
&D_{\Omega \times \G}(\bmu^2, \bmu^3) > D_{\Omega \times \G}(\bmu^1, \bmu^2) + D_{\Omega \times \G}(\bmu^1, \bmu^3) \\
  &\quad \iff D^2_{\Omega \times \G}(\bmu^2, \bmu^3) > D^2_{\Omega \times \G}(\bmu^1, \bmu^2) + 2 D_{\Omega \times \G}(\bmu^1, \bmu^2)D_{\Omega \times \G}(\bmu^1, \bmu^3) +D^2_{\Omega \times \G}(\bmu^1, \bmu^3) \\
  &\quad \iff a^2 + \frac12 d_{[0,1]}^2 ( 0, 2b-1  ) > a^2 + 2a d_{[0,1]}(1/2,b) + d^2_{[0,1]}(1/2,b),
\end{align*}
which coincides with the condition in (\ref{dynamicneqstaticeqn}).

\end{proof}

We conclude by applying Proposition \ref{examplesprop} to prove Proposition \ref{sharpnesscor}, on the sharpness of the inequalities relating the four (semi)-metrics.

\begin{proof}[Proof of Proposition \ref{sharpnesscor}]
Part (\ref{triangleinequalitycor}) is an immediate consequence of Proposition \ref{examplesprop}, inequality (\ref{notriangleineq}). Thus, it remains to show part (\ref{sharppart}). Consider the first inequality in (\ref{inequalityallmetrics}). Define $\bmu = [\delta_0,0]$ and $\bnu  = [ 0,\delta_0 ]$ and suppose $q_{12} = q_{21} \equiv q \geq 1/2$ and $q_{11}=q_{22} = 0$. Then, $d_{BL}(\bmu,\bnu ) = \sqrt{2}$, so $\min\{ 1, Q^{-1/2}\}  d_{BL} (\bmu,\bnu) = \sqrt{2/q}$. Furthermore, if    $\theta(s,t) = (s+t)/2$, by Lemma \ref{singleparticlemulti},  $\brho(t) := [(1-t)\delta_{0},t \delta_0]$ is solution of the vector valued continuity equation for $u_1 \equiv 0$, $u_2 \equiv 0$, and $v \equiv -2/q$, with $\|(\bu_t,\bv_t)\|^2_{\brho_t} \equiv 2/q$, so $W_{\R \times \G}(\bmu,\bnu^a) \leq \sqrt{2/q}$. This shows equality holds in the first inequality in (\ref{inequalityallmetrics}). On the other hand, for the measures $\bmu^1$ and $\bmu^2$ from Proposition \ref{examplesprop}, we showed in the Proposition that $W_{\R \times \G}(\bmu^1,\bmu^2) = a$, while $d_{BL}(\bmu^1,\bmu^2) = \frac{a}{\sqrt{2}}$, so strict inequality can hold in the first inequality in (\ref{inequalityallmetrics}).

The sharpness of the constant in the second inequality is an immediate consequence of Proposition \ref{examplesprop}, in which it is shown that $W_{\R \times \G}(\bu^1, \bmu^2) = a = D_{\R \times \G}(\bmu^1, \bmu^2)$. The fact that strict inequality can hold follows from inequality (\ref{strictinequalityman}).

For the third inequality, the sharpness of the constant can be seen by considering 
\[ \bmu^a = \left[\frac{1}{2} \delta_{-a}, \frac{1}{2} \delta_{a} \right]^t \text{ for }a \in (0,1) \] 
and noting that the canonical lifting $\lambda_{\bmu^a} =  \frac12 \delta_{(-a,1)} + \frac12 \delta_{(a,0)}$ is the only probability measure $\lambda \in \P({\R \times [0,1]})$ satsifying $\PP \lambda = \bmu^a$, so $D_{\R \times \G}(\bmu^a, \bmu^{\tilde{a}}) = W_{2, \mathcal{W}}(\bmu^a, \bmu^{\tilde{a}})$ for all $a , \tilde{a} \in [0,1]$. The fact that strict inequality can hold follows from considering the measures $\bmu^1$ and $\bmu^3$ as in Proposition \ref{examplesprop}. The canonical liftings are $\lambda_{\bmu^1}  = \frac{1}{2} \delta_{0,0}+\frac{1}{2} \delta_{0,1}$ and $\lambda_{\bmu^3} = (1-b) \delta_{0,0} +b \delta_{0,1} $, so
\begin{align*}
W_{2, \mathcal{W}}(\bmu^1, \bmu^3) = W_{\R \times [0,1]} \left( \frac{1}{2} \delta_{0,0}+\frac{1}{2} \delta_{0,1},(1-b) \delta_{0,0} +b \delta_{0,1} \right) = \left(b-\frac{1}{2} \right)d_{[0,1]}(0,1) .
\end{align*}
On the other hand, as shown in Proposition \ref{examplesprop}, for $b \in [1/2,1]$,
\begin{align*}
D_{\R \times \G}(\bmu^1, \bmu^3) = d_{[0,1]}(1/2,b)  .
\end{align*}
In particular, taking $\theta$ to be the geometric interpolation function, $\theta(s,t) = \sqrt{st}$, and $q_{ij} \equiv 1$, a direct computation in terms of the distance function given in Proposition \ref{explicitdistance2node} shows that, for $b = 3/4$, $D_{\R \times \G}(\bmu^1, \bmu^3)= \pi/6 < \pi/4 = W_{2, \mathcal{W}}(\bmu^1, \bmu^3) $.

\end{proof}

\subsection*{Acknowledgements} 

The work of K. Craig and \DJ. Nikoli\'c has been supported by NSF DMS grant 2145900. The work of N. Garc\'ia Trillos has been supported by NSF DMS grant 2236447. The authors would like to thank Matthias Wink for many helpful discussions, and both Matthias Erbar and Jan Maas for helpful email correspondence.


\bibliographystyle{abbrv}
\bibliography{Blob.bib}

\appendix

\section{Appendix} \label{realappendix}

We begin with the proof of Lemma \ref{Klem}.

\begin{proof}[Proof of Lemma \ref{Klem}]
 We begin by showing part (i). First, note that, by definition, $\beta \geq 0$, thanks to (A\ref{thetapositivity}). From  (A\ref{thetapositivehomogenity}), we obtain $\beta(a,b,c)\in \{0, +\infty \}.$  Since $\beta$ is a supremum of affine, continuous functions, it is convex and lower-semicontinuous. Let  $K:= {\rm dom}(\beta) \subseteq \R^{d+2}$. Then $K$ must be closed and convex.
Furthermore, by Remark \ref{interpfunorder},  \begin{equation*}
        a t + b s + \frac{\|c \|^2}{4}\theta(t,s) \leq  (a+ \frac{\|c \|^2}{8}) t + (b+ \frac{\|c \|^2}{8})s.
    \end{equation*} Hence, we have $\{ a+\frac{\|c\|^{2}}{8}\leq 0 \} \cap \{ b+\frac{\|c\|^{2}}{8}\leq 0 \} \subseteq K.$ 

    Now, we turn to the proof of part (ii). By a direct computation, 
    \begin{align*}
    \alpha^{*}(a,b,c) &= \sup_{t,s \geq 0, x\in \Rd} \{ at + bs + c\cdot x - \alpha(t,s,x)\}\\
    &= \max \left( \sup_{t,s \geq 0,x\in \Rd,\theta(t,s)\neq 0} \{ at + bs + c\cdot x - \frac{\|x\|^2}{\theta(t,s)}\}, \sup_{t,s \geq 0,\theta(t,s)=0}  \{ at + bs \}  \right) \\
    &= \max \left( \sup_{t,s \geq 0,\theta(t,s)\neq 0} \{ at + bs + \frac{\|c\|^2}{4}{\theta(t,s)}\}, \sup_{t,s \geq 0,\theta(t,s)=0}  \{ at + bs \}  \right) \\
    &= \beta(a,b,c).
\end{align*}

Since $\alpha$ is a proper, lower semicontinuous and convex function \cite[Lemma 2.7]{erbar2012ricci}, by the Fenchel-Moreau theorem, we obtain $\alpha=\alpha^{**}.$ Hence, $\alpha = \beta^{*}.$

\end{proof}

We now turn to the proof of Proposition \ref{basicfactWG}.

\begin{proof}[Proof of Proposition \ref{basicfactWG}]
We begin with  part (\ref{lem:ComparisonMetricsDiscrete}). 
The lower bound follows from   \cite[Proposition 2.12]{erbar2012ricci} and the equivalence of $\ell^1$ and $\ell^2$ norms on $\R^n$:
\[ W_\G(p_0, p_1) \geq \frac{1}{\sqrt{2}} d_{TV}(p_0,p_1) =  \frac{1}{\sqrt{2}} \|p_0 - p_1\|_{\ell^1(\R^n)} \geq \frac{1}{\sqrt{2}} |p_0 - p_1|. \]

To see the upper bound, first define \[ C_\theta := \int_0^1 \frac{1}{\sqrt{\theta(1-a,1+a)}} dr    \text{ and }k :=\min_{ij: q_{ij}>0} q_{ij}. \]
Note that assumption (A\ref{thetapositivehomogenity}) ensures 
 $\theta(1-a,1+a) \geq \theta (1-a,1-a) = (1-a) \theta(1,1) $, so $C_\theta < +\infty$.

Let $W_{2,g}$ denote the 2-Wasserstein distance on $\P(\G)$ defined in terms of the Kantorovich formulation, where the ground distance $d_\G$ on $\G$ is given by  the minimal number of edges required to connect any two nodes, (i.e., nodes $i$ and $j$ are connected by an edge whenever $q_{ij}>0$) \cite[Section 2.3]{erbar2012ricci}.
 Then \cite[Proposition 2.14]{erbar2012ricci} ensures that
\begin{equation}
W_\G(p_0, p_1) \leq \frac{\sqrt{2}C_{\theta}}{\sqrt{k}} W_{2,g}(p_{0},p_{1}), \quad \forall p_0, p_1 \in \mathcal{P}(\G),
\label{eq:WGVsW2}
\end{equation}
We may control $W_{2,g}$ in terms of the 1-Wasserstein metric and the diameter of the graph,
\[ W_{2,g}(p_0, p_1) \leq  \sqrt{\max_{i,j \in \G} d_\G(i,j)} \sqrt{W_{1,g}(p_0,p_1)}.    \]
For $\phi: \G \to \R$, let $\|\phi\|_{{\rm Lip}_\G}$ denote its Lipchitz constant with respect to $d_\G$. Then  Kantorovich-Rubinstein duality, the fact that $\sum_i p_{i,0} = \sum_i p_{i,1}$, and the Cauchy-Schwarz inequality give
\[ W_{1,g}(p_0,p_1) = \max_{\substack{ \phi \in \R^n \\ \|\phi\|_{{\rm Lip}_\G} \leq 1 }} \sum_i \phi_i  (p_{i,1}  - p_{i,0}) = \max_{\substack{ \phi \in \R^n \\ \|\phi\|_{{\rm Lip}_\G} \leq 1 \\  \phi_1 = 0   }} \sum_i \phi_i  (p_{i,1}  - p_{i,0})\leq \max_{\substack{ \phi \in \R^n \\ \|\phi\|_{{\rm Lip}_\G} \leq  1  \\ \phi_1 = 0   }} |\phi| | p_0 - p_1 |  \leq \tilde{C}_{\G} | p_0 - p_1 |  ,\]
for a constant $\tilde{C}_\G$ that depends on the edge weights $q_{ij}$. Combining the three previous inequalities gives the upper bound in the inequality. This completes the proof of part (\ref{lem:ComparisonMetricsDiscrete}).

We now turn to the proof of part (\ref{WGgeodesics}). By \cite[Theorem 3.2]{erbar2012ricci}, there exists $p:[0,1] \to \P(\G)$ continuous and  $m:[0,1] \to \R^{n\times n}$ locally integrable on $(0,1)$ so that $W_\G(p_t,p_s) = |t-s| W_\G(p_0, p_1)$ for all $s,t \in [0,1]$, and we also have 
\begin{align*} \partial_t p +  \nabla_\G \cdot  m  = 0,  \text{  in the sense of distributions}\\
 \frac12 \sum_{i,j=1}^n \alpha(m_{ij,t},p_{i,t},p_{j,t}) q_{ij} = W_\G^2(p_0,p_1) \text{ for all $t \in [0,1]$}, 
 \end{align*}
where $\alpha$ is defined as in equation (\ref{definitionofpsi}). Since the right hand side of the above equation is finite, we see that $\alpha(m_{ij,t},p_{i,t},p_{j,t}) <+\infty$ for all $i,j = 1, \dots, n$ and $t \in [0,T]$, so   $\check{p}_{ij,t}=0$ only if  $m_{ij,t} = 0$. Therefore, for all $i,j = 1, \dots, n$ and $t \in [0,T]$, we may define
\[ v_{ij,t} := \begin{cases} m_{ij,t} / \check{p}_{ij,t} &\text{ if }\check{p}_{ij,t}\neq 0 \\ 0 &\text{ otherwise,} \end{cases} \]
and we then have
\[ m_{ij,t} = v_{ij,t} \check{p}_{ij,t}  . \]
Then $(p, v)$ is a solution of the graph continuity equation, in the sense of Definition \ref{graphctydef}, and furthermore,
\[ \| v_t \|_{{\rm Tan}_{p_t}(\P(\G))} = \frac12 \sum_{i,j=1}^n \alpha(m_{ij,t},p_{i,t},p_{j,t}) q_{ij} = W_\G^2(p_0,p_1) \text{ for all $t \in [0,T]$}. \]


Part (\ref{riemannianpart}) follows from \cite[Theorem 2.4]{erbar2012ricci} and the fact that
\[\{ \nabla_\G \psi  : \psi \in \R^n   \}  = \left\{ \nabla_\G \psi  : \psi \in \R^n , \ \sum_{i=1}^n \psi_i = 0 \right\} . \]

\end{proof}

Next, we prove Lemma \ref{lem:MollificationSimplex}.

\begin{proof}[Proof of Lemma \ref{lem:MollificationSimplex}]
For $a \in (0,1)$ and $t \in [0,1]$, let
\[  v_{ij,t}^{a} := \frac{(1-a) v_{ij,t} \theta(p_{i,t}  , p_{j,t}) + a u_{ij,t} \theta(z_{i,t} , z_{j,t})}{  \theta(     p_{i,t}^{a}, p_{j,t}^{a}     ) }, \]
which, by (A\ref{thetapositivity}),   is well defined as $p_{i,t}^{a} \in (\P(\G))^\circ$. A direct computation shows that 
\[ \nabla_\G \cdot( \check{p}^{a} v^{a}  ) = (1-a) \nabla_\G \cdot ( \check{p}  v ) + a\nabla_\G \cdot (\check {z}   u ),   \]
so that $ (p^{a}, v^{a})$ is a solution to the graph continuity equation.

 To see the inequality for the norm of the tangent vectors, recall the definition of $\alpha: \R \times \R \times \R^d \rightarrow \R \cup \{+ \infty \}$ from equation (\ref{definitionofpsi})   and observe that 
\begin{align*}
( v^a_{ij,t})^2 \theta(p_{i,t}^a, p_{j,t}^a) &= \alpha(p_{i,t}^a, p_{j,t}^a, (1-a) v_{ij,t} \theta(p_{i,t}  , p_{j,t}) + a u_{ij,t} \theta(z_{i,t} , z_{j,t})) \\
&\leq (1-a) \alpha(p_{i,t}, p_{j,t},  v_{ij,t} \theta(p_{i,t}  , p_{j,t}) )+ a \alpha(z_{i,t}, z_{j,t},  u_{ij,t} \theta(z_{i,t} , z_{j,t}))   \\
 & =   (1-a)  v_{ij,t}^2 \theta(p_{i,t}, p_{j,t})  + a  {u}_{ij,t}^2 \theta(z_{i,t}, z_{j,t}),
\end{align*}
from where it is now easy to deduce \eqref{eqn:LengthsConvexComb}, using the definition of the norm on the tangent space, equation (\ref{innerproductgraphdef}).
\end{proof}

We now prove Corollary \ref{cor:IneqConvex}.

\begin{proof}[Proof of Corollary \ref{cor:IneqConvex}]
We first identify $u$  for which $ (z, u)$ solves the graph continuity equation. Let $\Delta_\G$ denote the  {graph} Laplacian,
\[\Delta_\G \varphi_i  :=  {\rm div}_\G   \nabla_\G \varphi_i = \sum_{j \in \G}  ( \varphi_i  - \varphi_j)q_{ij}, \quad i =1, \dots, n.  \]
Since   $\G$ is connected, the null space of $\Delta_G$ is the span of $\varphi_0\equiv 1$ (see, e.g. \cite[Lemma 1.7]{Chung1996},  which under minor modifications can be adapted to the weighted graph case). This, combined with the fact that $m_1 -m_0$ is orthogonal to $\varphi_0$, implies that there exists a unique solution $\varphi: \G \rightarrow \R$ with $\sum_i \varphi_i = 0$ to the graph-Poisson equation
\begin{equation*}
  \Delta_{\G} \varphi =  {m}_1 -  {m}_0.
\end{equation*}
Using $\varphi$, we define the discrete vector field
\[ u_{ij,t}  :=  - \frac{\nabla_\G \varphi_{ij} }{\theta(z_{i,t} , z_{j,t})}. \]
Then, we have 
\[ - \nabla_\G \cdot ( \check{z}_t u_t) =  \nabla_\G \cdot ( \nabla _\G \varphi ) = \Delta_\G \varphi =  {m}_1 - {m}_0 = \partial_t z_t,  \]
and we conclude that $ (z, u)$ solves the graph continuity equation on $[0,1]$. 

We now claim that
 \begin{align} \|u_{t} \|_{\Tan_{z_{t }} \P(\G)} \leq \sqrt{ \tilde{C}_{m_0, m_1}/\lambda_G}  \ | m_1-m_0|  \label{claimyclaim},
 \end{align}
 for $\tilde{C}_{ {m}_0,  {m}_1}$ depending only on $\min_i  m_{i,0} >0$ and $ \min_i m_{i,1}>0$ and   $\lambda_G$   the first non-trivial eigenvalue of $\Delta_G$. First, note that, if $m_0 = m_1$, then $\varphi = 0$ and the result holds. Suppose now that $m_0 \neq m_1$, so $\varphi \neq 0$. By  (A\ref{thetanormalization})-(A\ref{thetamonotonicity}), we see that
\begin{align}
\begin{split}
 \|u_{t} \|_{\Tan_{z_{t }} \P(\G)}^2  &= \frac{1}{2} \sum_{i,j} | u_{ij,t} |^2 \theta(z_{i,t} , z_{j,t}   ) q_{ij}   =  \frac{1}{2} \sum_{i,j} \frac{ | \nabla_\G \varphi_{ij,t} |^2}{ \theta(z_{i,t} , z_{j,t}   )} q_{ij}     \leq \frac{\tilde{C}_{ {m}_0,  {m}_1}   }{2}  \sum_{i,j}   | \nabla_\G \varphi_{ij} |^2 q_{ij}, 
 \\ &= \tilde{C}_{ {m}_0,  {m}_1}   \sum_i \varphi_i \Delta_\G \varphi_i. 
 \end{split}
 \label{eq:IneqAux22}
 \end{align}
By the discrete Poincaré inequality, which follows from a spectral decomposition of $\Delta_G$, 
\[  \lVert \varphi \rVert^2  \leq \frac{1}{\lambda_\G}  \sum_i \varphi_i \Delta_\G \varphi_i, \]
 Combining this with inequality (\ref{eq:IneqAux22}) and the Cauchy-Schwarz inequality, we obtain
\[   \sum_i \varphi_i \Delta_\G \varphi_i =  \sum_i \varphi_i (m_{i,1} - m_{i,0}) \leq  |\varphi | | m_1-m_0 | \leq \lambda_\G^{-1/2}  \sqrt{ \sum_i \varphi_i \Delta_\G \varphi_i} \  | m_1-m_0 |.\]
 Dividing through by $ \sqrt{ \sum_i \varphi_i \Delta_\G \varphi_i} $ and substituting in  \eqref{eq:IneqAux22}, we deduce (\ref{claimyclaim}).
The result now follows from Lemma \ref{lem:MollificationSimplex}.

\end{proof}

Next, we show Lemma \ref{simplexgeodesicapproximationlem}.

\begin{proof}[Proof of Lemma \ref{simplexgeodesicapproximationlem}]
Part (\ref{existenceofconstantspeedgeodesicDelta})  is an immediate consequence of the corresponding fact for $(\P(G), W_\G)$, Proposition \ref{basicfactWG} (\ref{WGgeodesics}), and the fact that $(\Delta^{n-1}, d_{\Delta^{n-1}})$ is isometric to    $(\P(G), W_\G)$.
Part (\ref{measurabilityofgeodesicmap}) follows as in \cite[footnote 1, p94]{ambrosio2021lectures} and  \cite[Lemma 6.7.1, Theorem 6.9.7]{bogachev2007measure}.

We now consider perturbations of geodesics that stay away from the boundary of the simplex. Since  $p:=\bp(\gamma_{r_0,r_1})$ is a constant speed geodesic on $(\P(\G), W_\G)$ from $p_0 := \bp(r_0)$ to $p_1 := \bp(r_1)$,  by Proposition \ref{basicfactWG} (\ref{WGgeodesics}), there exists a velocity field $v$ so that $(p,v)$ solves the graph continuity equation on the time interval $[0,1]$ and 
\[  \| v_t \|_{\Tan_{p_t}\P(\G)} = W_\G(p_0,p_1) \text{ for a.e. } t \in [0,1] . \] Thus, for $a  \in (0,1)$ and $m_0:=\bp(s_0), m_1:=\bp(s_1) \in \P(\G)^\circ$, Corollary \ref{cor:IneqConvex} ensures there exists a velocity field $v^a$ so that $(p^a, v^a)$ is a solution to the graph continuity equation on the time interval $[0,1]$ satisfying, for all $i = 1, \dots, n$,
\begin{align*}
&p^a_{i,t} \geq aC_{m_0,m_1, \G}  \text{ and } \lim_{a \to 0} p^a_{i,t} \to p_{i,t}  \quad \text{ for all } t \in [0,1], \\
&\|v_{t}^{a}\|_{\Tan_{p_{t }^{a}} \P(\G)}^2  \leq (1-a)   \|v_{t}\|_{\Tan_{p_{t}} \P(\G)}^2    + a C_{ {m}_0,  {m}_1, \G} = (1-a) d^2_{\Delta^{n-1}}(r_0,r_1) + a C_{ {m}_0,  {m}_1, \G} , \text{ for } a.e. t \in [0,1].
\end{align*}
Thus, by \cite[Proposition 3.3]{erbar2012ricci}, $p^a \in AC([0,1]; \P(\G))$ and 
\[ |(\dot{p}^a)'|^2_{\Delta^{n-1}}(t)  \leq (1-a) d^2_{\Delta^{n-1}}(r_0,r_1) + a C_{ {m}_0,  {m}_1, \G} . \]
Define 
\begin{align*}
 \gamma_{r_0,r_1}^a &:= \bp^{-1} (p^a)  =(1-a) \bp^{-1}(p_t) + a \bp^{-1}((1-t) m_0 + t m_1) = (1-a) \gamma_{r_0,r_1,t} + a((1-t) s_0 + t s_1) . 
 \end{align*}Then properties (\ref{eq:AwayBoundary}-\ref{eq:AuxRegularized1}) follow from the corresponding properties for $p^a$.

\end{proof}

We turn to the proof of the time rescaling and gluing Lemma \ref{timerescaling}.

\begin{proof}[Proof of Lemma \ref{timerescaling}]
First we show part (i).
We may assume $\tau^{-1}$ is smooth and differentiable, and $(\tau^{-1})'>0.$
We choose $\tilde{\varphi}\in C_c(\Omega\times (0,T))$ and set $\varphi(x,s):= \tilde{\varphi}(x,\tau^{-1}(s)) \in C_c(\Omega\times (s_{0},s_{1})).$ Thus, \begin{align*} 
& \int_{s_0}^{s_1} \int_{\Omega} (\tau^{-1})'(s)\partial_t \tilde{\varphi}(x,\tau^{-1}(s)) d \rho_{i,s}(x) ds  \\
 &\quad= - \int_{s_0}^{s_1} \int_{\Omega}  \nabla \tilde{\varphi}(x,\tau^{-1}(s)) \cdot \bu_{i,s}(x)   d\rho_{i,s}  (x ) ds  \nonumber \\
&\quad \quad\quad  - \frac12 \sum_{j = 1}^n \int_{s_0}^{s_1} \int_{\Omega} \tilde{\varphi}(x,\tau^{-1}(s)) \theta \left( \frac{ d \rho_{i,s}}{d \bar{\rho}_s }(x),  \frac{ d \rho_{j,s}}{d \bar{\rho}_s}(x) \right) (\bv_{ij,s}(x) - \bv_{ji,s}(x)) q_{ij}  d \bar{\rho}_s(x) ds \nonumber \\
&\quad= \int_{s_0}^{s_1} (\tau^{-1})'(s) \left( - \int_{\Omega}  \nabla \tilde{\varphi}(x,\tau^{-1}(s)) \cdot \frac{\bu_{i,s}(x)}{(\tau^{-1})'(s)}   d\rho_{i,s}  (x) \right.  \nonumber \\
&\quad \quad\quad \left. - \frac12 \sum_{j = 1}^n  \int_{\Omega} \tilde{\varphi}(x,\tau^{-1}(s)) \theta \left( \frac{ d \rho_{i,s}}{d \bar{\rho}_s }(x),  \frac{ d \rho_{j,s}}{d \bar{\rho}_s}(x) \right) \left(\frac{\bv_{ij,s}(x)}{(\tau^{-1})'(s)} - \frac{\bv_{ji,s}(x)}{(\tau^{-1})'(s)} \right) q_{ij}  d \bar{\rho}_s(x) \right) ds \nonumber \\
&\quad= \int_0^{T}  \left( - \int_{\Omega}  \nabla \tilde{\varphi}(x,t) \cdot (\tau'(t)\bu_{i,\tau(t)}(x))   d\rho_{i,\tau(t)}  (x) \right.  \nonumber \\
&\quad \quad\quad \left. - \frac12 \sum_{j = 1}^n  \int_{\Omega} \tilde{\varphi}(x,t) \theta \left( \frac{ d \rho_{i,\tau(t)}}{d \bar{\rho}_{\tau(t)} }(x),  \frac{ d \rho_{j,\tau(t)}}{d \bar{\rho}_{\tau(t)}}(x) \right) \left(\tau'(t)\bv_{ij,\tau(t)}(x) - \tau'(t)\bv_{ji,\tau(t)}(x) \right) q_{ij}  d \bar{\rho}_{\tau(t)}(x) \right) dt \nonumber \\
&\quad= \int_0^{T}  \left( - \int_{\Omega}  \nabla \tilde{\varphi}(x,t) \cdot \tilde{\bu}_{i,t}(x)   d\tilde{\rho}_{i,t}  (x) \right.  \nonumber \\
&\quad \quad\quad \left. - \frac12 \sum_{j = 1}^n  \int_{\Omega} \tilde{\varphi}(x,t) \theta \left( \frac{ d \tilde{\rho}_{i,t}}{d \bar{\tilde{\rho}}_{t} }(x),  \frac{ d \tilde{\rho}_{j,t}}{d \bar{\tilde{\rho}}_t}(x) \right) \left(\tilde{\bv}_{ij,t}(x) - \tilde{\bv}_{ji,t}(x) \right) q_{ij}  d \bar{\tilde{\rho}}_t(x) \right) dt \nonumber \\
&=\int_0^{T} \int_{\Omega} \partial_t \tilde{\varphi}(x,t) d \tilde{\rho}_{i,t}(x) dt.
\end{align*}
Next, we show part (ii).
    We first prove that $(\brho,\bu,\bv)$ satisfies the vector valued  continuity equation (\ref{vvcty}) on $[0,1]$  in the sense of distributions.
First,  note that \begin{align*}
     \int_0^1 \int_{\Omega} \partial_t \eta(x,t) d \rho_{i,t}(x) dt     
     &= \int_0^\frac{1}{2} \int_{\Omega} \partial_t \eta(x,t) d \rho_{i,2t}^{1}(x) dt + \int_\frac{1}{2}^1 \int_{\Omega} \partial_t \eta(x,t) d\rho_{i,2t-1}^{2}(x) dt\\
     &= \int_0^1 \int_{\Omega} 2\partial_s \eta \left(x,\frac{s}{2}\right) d \rho_{i,s}^{1}(x)\frac{ds}{2} + \int_0^1 \int_{\Omega} 2\partial_s \eta \left(x,\frac{1+s}{2} \right) d\rho_{i,s}^{2}(x) \frac{ds}{2}.\\
\end{align*}
Likewise,
 \begin{align*}&-\int_0^\frac{1}{2} \int_{\Omega}  \nabla \eta(\cdot,t) \cdot u_{i,t}   d\rho_{i,t}   dt   -  \sum_{j = 1}^n \int_0^\frac{1}{2} \int_{\Omega} \eta(\cdot,t) \theta \left( \frac{ d \rho_{i,t}}{d \bar{\rho}_t } ,  \frac{ d \rho_{j,t}}{d \bar{\rho}_t}  \right) v_{ij,t} d \bar{\rho}_t  dt \\
&\quad - \int_\frac{1}{2}^1 \int_{\Omega}  \nabla \eta(\cdot,t) \cdot u_{i,t}   d\rho_{i,t}   dt    -  \sum_{j = 1}^n \int_\frac{1}{2}^1 \int_{\Omega} \eta(\cdot,t) \theta \left( \frac{ d \rho_{i,t}}{d \bar{\rho}_t },  \frac{ d \rho_{j,t}}{d \bar{\rho}_t} \right) v_{ij,t}  d \bar{\rho}_t dt \\
&\quad= - \int_0^1 \int_{\Omega}  \nabla \eta \left(\cdot,\frac{s}{2}\right) \cdot 2u_{i,s}^{1}    d\rho_{i,s}^{1}    \frac{ds}{2}   -  \sum_{j = 1}^n \int_0^1 \int_{\Omega} \eta(\cdot,t) \theta \left( \frac{ d \rho_{i,t}^{1}}{d \bar{\rho}_t^1 } ,  \frac{ d\rho_{j,s}^{1}}{d \bar{\rho}_s^1}  \right) 2v_{ij,s}^1  d \bar{\rho}_s^1  \frac{ds}{2} \\
&\quad \quad - \int_0^1 \int_{\Omega}  \nabla \eta \left(\cdot,\frac{1+s}{2}\right) \cdot 2u_{i,s}^2  d\rho_{i,s}^2    \frac{ds}{2}    -  \sum_{j = 1}^n \int_0^1 \int_{\Omega} \eta \left(\cdot,\frac{1+s}{2} \right) \theta \left( \frac{ d\rho_{i,s}^{2}}{d \bar{\rho}_s^2} , \frac{ d\rho_{j,s}^{2}}{d \bar{\rho}_s^2}  \right) 2v_{ij,s}^{2} d \bar{\rho}_s^2 \frac{ds}{2}.
\end{align*} The last equality follows from the definition, and part (i) of this lemma.

\end{proof}

We now prove Lemma \ref{energyestimate}]. 
\begin{proof}[Proof of Lemma \ref{energyestimate}]
    For any $0 \leq s_0 \leq s_1 \leq 1$, by Lemma \ref{timerescaling}, $(\brho_{(1-t)s_0 + t s_1}, (s_1-s_0) \bu_{(1-t)s_0 + t s_1}, (s_1-s_0) \bv_{(1-t)s_0 + t s_1}) \in \mathcal{C}(\brho_{s_0}, \brho_{s_1})$. Let $\tau(t) = (1-t) s_0 + t s_1$. Fix $\eta \in C^\infty_c(\Rd)$. Thus, applying the weak formulation of the continuity equation (\ref{dualvvcty}) and the fact $v_{ij}=-v_{ji}$, followed by a change of variables, we estimate
\begin{align*}
& \int_{\Rd} \eta(x) d \rho_{i,s_1}(x) -\int_{\Rd} \eta(x) d \rho_{i,s_0}(x)  \\
   &\quad  = (s_1-s_0) \int_{0}^{1} \int_{\Rd}  \nabla \eta(x) \cdot u_{i,\tau(t)}(x)   d\rho_{i,\tau(t)} (x)  dt  \\
   &\quad \quad  + (s_1-s_0)   \sum_{j = 1}^n \int_{0}^{1} \int_{\Rd} \eta(x) \theta \left( \frac{ d\rho_{i,\tau(t)}}{d \bar{\rho}_{\tau(t)} }(x),  \frac{ \rho_{j,\tau(t)}}{d \bar{\rho}_{\tau(t)}}(x) \right) v_{ij,\tau(t)}(x)  q_{ij}  d \bar{\rho}_{\tau(t)}(x)  dt \nonumber  \\
   &\leq  \int_{s_0}^{s_1} \int_\Rd |\nabla \eta| |u_{i,\tau}| d \rho_{i, \tau} d\tau    +   \sum_{j=1}^n   \int_{s_0}^{s_1} \int_\Rd |\eta| \theta \left(\frac{d \rho_{i,\tau}}{ d \bar{\rho}_{\tau} },  \frac{d \rho_{j,\tau}}{ d \bar{\rho}_{\tau} }  \right)   |v_{ij,t}| q_{ij}  d \bar{\rho}_\tau   d\tau.
\end{align*} 

Regarding   part (ii), we obtain \begin{align*}
\frac{d}{dt} \int_\Rd \eta (x) d \bar{\rho}_t (x) &= \sum_{i=1}^{n} \frac{d}{dt} \int_\Rd \eta(x) d \rho_{i,t} \\&= \sum_{i=1}^n \int_\Rd \nabla \eta(x) \cdot u_{i,t}(x) d \rho_{i,t}  \\
&\quad +\frac{1}{2} \int_\Rd \sum_{i=1}^n \sum_{j=1}^n \eta(x) \theta \left( \frac{d \rho_{i,t}}{d \bar{\rho}_t}(x)  , \frac{ d \rho_{j,t}}{d \bar{\rho}_t}(x) \right) (v_{ij,t}(x) - v_{ji,t}(x)) q_{ij} d \bar{\rho}_t \\
&= \sum_{i=1}^n \int_\Rd \nabla \eta(x) \cdot u_{i,t}(x) d \rho_{i,t} .
\end{align*}

\end{proof}

Finally we proof Proposition \ref{explicitdistance2node}.
\begin{proof}[Proof of Proposition \ref{explicitdistance2node}]
In the two node case $\G = \{ g_1, g_2\}$, \cite[Theorem 2.4]{maas2011gradient} guarantees that, for any $0 \leq a_0 \leq a_1 \leq 1$, 
\begin{align*} d_{[0,1]}(a_0,a_1) &= W_\G \left( \begin{bmatrix} a_{0} \\ 1 - a_{0} \end{bmatrix},  \begin{bmatrix} a_{1} \\ 1 - a_{1} \end{bmatrix} \right) = \frac{1}{\sqrt{2q}} \int_{1-2a_1}^{1-2a_0} \frac{1}{ \sqrt{\theta(1-r, 1+r)}} dr \\
&= \frac{1}{\sqrt{2q}} \int_{a_0}^{a_1} \frac{2da}{\sqrt{ \theta(2a,2(1-a))}}   = \frac{1}{\sqrt{q}}  \int_{a_0}^{a_1} \frac{da}{\sqrt{ \theta(a,1-a)}} ,
\end{align*}
where the last equality follows from the homogeneity of $\theta$, assumption (A\ref{thetamonotonicity}). 
\end{proof}

\end{document}